\documentclass[12pt]{amsart}
\usepackage{fullpage,url,amssymb,enumerate,colonequals}
\usepackage[all,cmtip]{xy} 
\usepackage{mathrsfs} 

\DeclareFontEncoding{OT2}{}{} 
\newcommand{\textcyr}[1]{%
 {\fontencoding{OT2}\fontfamily{wncyr}\fontseries{m}\fontshape{n}\selectfont #1}}
\newcommand{\Sha}{{\mbox{\textcyr{Sh}}}}

\usepackage{color}
\definecolor{darkgreen}{rgb}{0,0.3,0}

\newcommand{\TAU}{\tau}

\newcommand{\defi}[1]{\textsf{#1}} 
\newcommand{\Aff}{{\mathbb A}}
\newcommand{\C}{{\mathbb C}}
\newcommand{\F}{{\mathbb F}}
\newcommand{\G}{{\mathbb G}}

\newcommand{\PP}{{\mathbb P}}
\newcommand{\PPdual}{\widehat{\PP}}
\newcommand{\Q}{{\mathbb Q}}
\newcommand{\R}{{\mathbb R}}
\newcommand{\Z}{{\mathbb Z}}
\newcommand{\ZZ}{{\mathbb Z}}

\newcommand{\cc}{{\mathfrak c}}
\newcommand{\mm}{{\mathfrak m}}

\newcommand{\fD}{{\mathfrak D}}
\newcommand{\fP}{{\mathfrak P}}

\newcommand{\calG}{{\mathcal G}}

\newcommand{\calJ}{{\mathcal J}}
\newcommand{\calK}{{\mathcal K}}

\newcommand{\calN}{{\mathcal N}}
\newcommand{\calO}{{\mathcal O}}

\newcommand{\calS}{{\mathcal S}}
\newcommand{\calT}{{\mathcal T}}

\newcommand{\calX}{{\mathcal X}}
\newcommand{\calY}{{\mathcal Y}}
\newcommand{\calZ}{{\mathcal Z}}

\newcommand{\LL}{{\mathscr L}}

\newcommand{\OO}{{\mathscr O}}
\newcommand{\TT}{{\mathscr T}}

\DeclareMathOperator{\odd}{odd}
\DeclareMathOperator{\ad}{ad}
\DeclareMathOperator{\tr}{tr}

\DeclareMathOperator{\Frob}{Frob}
\DeclareMathOperator{\coker}{coker}
\DeclareMathOperator{\rk}{rk}
\DeclareMathOperator{\id}{id}
\DeclareMathOperator{\Char}{char}

\DeclareMathOperator{\im}{im}

\DeclareMathOperator{\End}{End}

\DeclareMathOperator{\Hom}{Hom}

\DeclareMathOperator{\Aut}{Aut}
\DeclareMathOperator{\Gal}{Gal}
\DeclareMathOperator{\Disc}{Disc}

\DeclareMathOperator{\Res}{Res}

\DeclareMathOperator{\Br}{Br}

\DeclareMathOperator{\divisor}{div}

\DeclareMathOperator{\Selsym}{Sel}
\newcommand{\Sel}[1]{\Selsym^{#1}}
\newcommand{\Selfake}[1]{\Selsym^{#1}_\mathrm{fake}}
\newcommand{\Seltrue}[1]{\Selsym^{#1}_\mathrm{true}}
\newcommand{\Seltf}[1]{\Selsym^{#1}_\mathrm{true/fake}}
\DeclareMathOperator{\Cl}{Cl}

\DeclareMathOperator{\Div}{Div}
\DeclareMathOperator{\Pic}{Pic}
\DeclareMathOperator{\Princ}{Princ}
\DeclareMathOperator{\Jac}{Jac}
\DeclareMathOperator{\Alb}{Alb}

\DeclareMathOperator{\PIC}{\bf Pic}
\DeclareMathOperator{\Spec}{Spec}

\DeclareMathOperator{\Proj}{Proj}

\DeclareMathOperator{\pr}{pr}

\newcommand{\good}{{\textup{good}}}

\newcommand{\tors}{{\operatorname{tors}}}

\newcommand{\unr}{{\operatorname{unr}}}
\newcommand{\tH}{{\text{th}}}

\newcommand{\SL}{\operatorname{SL}}
\newcommand{\Sp}{\operatorname{Sp}}

\newcommand{\surjects}{\twoheadrightarrow}
\newcommand{\injects}{\hookrightarrow}
\newcommand{\To}{\longrightarrow}
\newcommand{\isom}{\simeq}

\newcommand{\del}{\partial}
\newcommand{\Intersection}{\bigcap} 
\newcommand{\intersect}{\cap} 
\newcommand{\union}{\cup} 
\newcommand{\tensor}{\otimes}

\newcommand{\isomto}{\overset{\sim}{\rightarrow}}

\newcommand{\fes}{\Delta}
\DeclareMathSymbol{\rightarrowhead}{\mathrel}{AMSa}{"4B}
\newcommand{\dotrightarrow}{\mathrel{\cdots\!\rightarrowhead}}
\newcommand{\corr}[3]{{#2} \stackrel{#1}\dotrightarrow {#3}}  

\newtheorem{theorem}{Theorem}[section]
\newtheorem{lemma}[theorem]{Lemma}
\newtheorem{corollary}[theorem]{Corollary}
\newtheorem{proposition}[theorem]{Proposition}

\theoremstyle{definition}
\newtheorem{definition}[theorem]{Definition}

\newtheorem{hypothesis}[theorem]{Hypothesis}
\newtheorem{example}[theorem]{Example}
\newtheorem{examples}[theorem]{Examples}

\theoremstyle{remark}
\newtheorem{remark}[theorem]{Remark}

\usepackage{microtype}

\usepackage[alphabetic,lite]{amsrefs} 
\begin{document}

\title[Explicit descent and genus-3 curves]{Generalized explicit descent and its application to curves of genus~3}
\subjclass[2010]{Primary 11G30; Secondary 11G10, 14G25, 14H45}
\keywords{Descent, Selmer group, genus~3 curve}

\author{Nils Bruin}
\address{Department of Mathematics, Simon Fraser University,
         Burnaby, BC V5A 1S6, Canada}
\email{nbruin@sfu.ca}
\urladdr{http://www.cecm.sfu.ca/~nbruin/}

\author{Bjorn Poonen}
\address{Department of Mathematics, Massachusetts Institute of Technology, Cambridge, MA 02139-4307, USA}
\email{poonen@math.mit.edu}
\urladdr{http://math.mit.edu/~poonen/}

\author{Michael Stoll}
\address{Mathematisches Institut,
         Universit\"at Bayreuth,
         95440 Bayreuth, Germany}
\email{Michael.Stoll@uni-bayreuth.de}
\urladdr{http://www.mathe2.uni-bayreuth.de/stoll/}

\thanks{N.B.\ was partially supported by NSERC.  B.P.\ was partially supported by the Guggenheim Foundation and by National Science Foundation grants DMS-0301280, DMS-0841321, and DMS-1069236.  M.S.\ was partially supported by the Deutsche Forschungsgemeinschaft.}

\date{December 11, 2013}

\begin{abstract}
We introduce a common generalization of essentially all known methods
for explicit computation of Selmer groups, which are used to 
bound the ranks of abelian varieties over global fields.
We also simplify and extend the proofs relating what is computed 
to the cohomologically-defined Selmer groups.
Selmer group computations have been practical for many Jacobians
of curves over~$\Q$ of genus up to~$2$ since the 1990s,
but our approach is the first to be practical for general curves of genus~$3$.
We show that our approach succeeds
on some genus-$3$ examples defined by polynomials with small coefficients.
\end{abstract}

\maketitle

\section{Introduction}

\subsection{Background}

The Mordell--Weil theorem \cites{Mordell1922,Weil1929}
states that for any abelian variety~$J$ over a number field~$k$,
the abelian group~$J(k)$ is finitely generated.
One of the main steps of the proof involves showing the
finiteness of $J(k)/nJ(k)$ for some $n \ge 2$.
And there is essentially only one known proof of this finiteness,
based on a vast generalization of Fermat's method of infinite descent.
In modern terms, the proof embeds $J(k)/nJ(k)$
into a Selmer group~$\Sel{n}(J)$,
a finite group that is computable in principle.

But the Selmer group is defined as a subgroup of 
a Galois cohomology group~$H^1(k,J[n])$,
and $1$-cocycles for the absolute Galois group~$\calG$ of~$k$
are not objects that a computer can deal with directly.
Fortunately, sometimes one can find more concrete representations 
for elements of~$H^1(k,J[n])$.
For example, if $J$ is an elliptic curve with $J[2] \subseteq J(k)$,
then $J[2]$ is isomorphic to $\mu_2 \times \mu_2$ as $\calG$-module,
and the Kummer sequence yields
$H^1(k,J[2]) \isom k^\times/k^{\times 2} \times k^\times/k^{\times 2}$.

In many higher-dimensional situations, however, the number field over which 
all the points of~$J[n]$ become rational is too large for the
required computations.
Instead one tries to find exact sequences
relating $J[n]$ to modules induced from~$\Z/n\Z$ or~$\mu_n$,
since cohomology of induced modules can be computed by 
Shapiro's lemma~\cite{Atiyah-Wall1967}*{\S4,~Proposition~2}.
For example, if $J$ is the Jacobian of a hyperelliptic curve $y^2=f(x)$
with $\deg f$ odd,
and $\Delta$ is the $\calG$-set of Weierstrass points not including
the one at infinity, then elements of~$J[2]$ are represented
by degree-$0$ divisors supported on $\Delta \union \{\infty\}$,
and we obtain a split exact sequence
\[
	0 \to \frac{\Z}{2\Z} \to \left(\frac{\Z}{2\Z}\right)^{\Delta} 
		\to J[2] \to 0
\]
in which $(\Z/2\Z)^{\Delta}$ is a direct sum of induced modules: 
this was exploited in~\cite{Schaefer1995}.
For $y^2=f(x)$ with $\deg f$ even, there is still a relationship
between $J[2]$ and induced modules, but it is more involved,
and it becomes much harder to relate $H^1(k,J[2])$ 
to concrete objects: this problem, together with its generalization 
to $y^p=f(x)$ for larger primes~$p$,
was addressed in~\cite{Poonen-Schaefer1997}.
The situations of the two previous sentences 
were called true descent and fake descent, respectively, 
in~\cite{Poonen-Schaefer1997}.

\subsection{Goal of this article}

The main goal of this article is to develop a practical generalization
of true and fake descent that contains essentially all 
previous instantiations of explicit descent.
Our generalization is broad enough to suggest an explicit approach
for Jacobians of arbitrary curves, using the $\calG$-set of
odd theta characteristics.
We demonstrate its practicality by computing the rank of~$J(\Q)$
for the Jacobian of a non-hyperelliptic curve genus-$3$ curve~$X$
with no special property beyond having small discriminant: 
see Section~\ref{S:disc 4727 example}.
This is the first time that Selmer group computations 
for ``general'' genus-$3$ Jacobians have been possible.

\begin{remark}
Practical Selmer group calculations rely on the computation
of class groups of number fields, and this is usually the bottleneck,
except in situations where the $\calG$-action on certain torsion points
is much smaller than expected for a general curve.
For a general genus-$3$ curve over~$\Q$, our method requires the class group
of a number field of degree~$28$; this seems to be the smallest possible,
given that $28$ is the smallest index of a proper subgroup of $\Sp_6(\F_2)$.
For general genus-$4$ curves, 
the corresponding index is~$120$, which is likely to remain outside the realm
of practical computation for some time.
\end{remark}

Setting the computational advantages of our approach aside,
the main theoretical advances in our article are as follows.
\begin{itemize}
\item Our approach of taking Cartier duals before taking cohomology
(Sections \ref{S:true explicit definition} and~\ref{S:fake cohom def})
leads quickly to concrete results on groups such as
$H^1(k,J[2])$ over any field of characteristic not~$2$;
this approach has already found outside applications: 
see~\cite{Bhargava-Gross-Wang-AIT2-preprint}*{Section~4}.
In particular, our approach eliminates the use of generalized Jacobians and 
the group scheme $\calJ_\mm$ in~\cite{Poonen-Schaefer1997},
a significant simplification even in the setting of hyperelliptic curves.
\item The introduction of the middle rows in 
\eqref{E:true nz+y diagram} and~\eqref{E:fake nz+y diagram}
provides short proofs of the comparison theorems 
relating the cohomological and explicit definitions of descent maps.
\item Appendix~\ref{S:SelDirect} shows how to augment the explicit descent
maps to produce an explicit description of the Selmer group itself 
instead of a ``fake'' approximation to it.
\end{itemize}

\begin{remark}
\label{R:Thorne}
For certain Jacobians~$J$,
there is also a representation-theoretic approach 
to understanding $J[2]$ and its Galois cohomology,
based on Vinberg theory, whose relevance for arithmetic problems
was first pointed out by Benedict Gross.
Namely, Jack Thorne has shown very generally how, starting
from a simple split algebraic group~$G$ of type $A$, $D$, or~$E$
over a field~$k$ of characteristic~$0$, 
one can produce a family of curves for which $J[2]$ can be identified
with the stabilizers for an action of the subgroup~$G^\theta$
fixed by an involution~$\theta$ in a 
particular canonical $G^{\ad}(k)$-conjugacy class of involutions of~$G$,
and from this one can obtain information about $H^1(k,J[2])$.
In particular, when $G$ is of type~$E_6$ (resp.~$E_7$),
Thorne's construction yields the universal family of non-hyperelliptic
genus-$3$ curves with a marked hyperflex 
(resp.,\ a marked flex that is not a hyperflex).
See~\cite{Thorne-thesis}; the families of genus-$3$ curves
appear explicitly in Theorem~4.8 there.
In the $E_6$~case, $\Gal(k(J[2])/k)$ is generically~$W(E_6)$, 
isomorphic to an index-$28$ subgroup of the group $\Sp_6(\F_2)$
that arises for a general genus-$3$ curve.
In the $E_7$~case, $\Gal(k(J[2])/k)$ is generically~$W(E_7)/\{\pm1\}$,
isomorphic to the full group $\Sp_6(\F_2)$.
It is reasonable to hope that 
it will eventually be possible to use Thorne's work 
to study $\Sel{2}(J)$ 
for any non-hyperelliptic genus-$3$ curve with a rational flex.
\end{remark}

\subsection{Road map to the rest of the article}

The first few sections are preliminary.
Section~\ref{S:notation} introduces the notation to be used 
throughout the rest of the article; much of it is standard.
Section~\ref{S:twisted powers} introduces twisted powers,
a slight generalization of induced modules and permutation modules.
Section~\ref{S:Weil pairings} uses Lang reciprocity to relate 
various definitions of Weil pairings.
Section~\ref{S:odd theta characteristics} reviews and develops 
the combinatorics of theta characteristics on a curve.

Section~\ref{S:generalized explicit descent} introduces the
key notions of the paper.
First it axiomatizes the settings to which our explicit approach applies,
formalizing them into the notions of \emph{true descent setup} 
and \emph{fake descent setup},
which are general enough to handle various isogenies $\phi \colon A \to J$
over a global field~$k$.
Given a true descent setup, we define a homomorphism
\begin{equation}
\label{E:intro Cassels map}
	\frac{J(k)}{\phi A(k)} \to \frac{L^\times}{L^{\times n}}
\end{equation}
that acts as a computation-friendly substitute 
for the connecting homomorphism
\begin{equation}
\label{E:intro descent map}
	\frac{J(k)}{\phi A(k)} \to H^1(k,A[\phi])
\end{equation}
appearing in the definition of the actual $\phi$-Selmer group.
In fact, the homomorphism~\eqref{E:intro Cassels map}
can be defined in two ways, either by using cohomology 
(good for comparing it to the homomorphism~\eqref{E:intro descent map}
used to define the actual $\phi$-Selmer group)
or by evaluating explicit rational functions on $0$-cycles
(good for computing the homomorphism).
Using our work on Weil pairings, we prove that the two definitions agree.
A more complicated argument establishes compatibility of 
analogous definitions for a fake descent setup; 
here $L^\times/L^{\times n}$ is replaced by $L^\times/L^{\times n} k^\times$.

Section~\ref{S:unramified classes} identifies the computation-friendly
analogue of the subgroup of classes in~$H^1(k,A[\phi])$ unramified outside
a finite set of places, which is essential for making the computations finite.
Section~\ref{S:Selmer groups and sets} defines a computation-friendly
analogue of the $\phi$-Selmer group, called the \emph{true} or
\emph{fake Selmer group}, 
and defines an analogue for a variety~$X$ whose Albanese variety is~$J$.
Section~\ref{S:relations between Selmer groups} uses
the notion of Shafarevich--Tate group
from Section~\ref{S:Sha1} to prove results
that often enable one to pass from knowledge of the true or fake Selmer group
to the actual $\phi$-Selmer group.
A more elaborate method that \emph{always} succeeds 
in calculating the $\phi$-Selmer group is presented in an appendix, 
but in some situations it may be impractical.

Section~\ref{S:compute fake} provides details on how to compute
true and fake Selmer groups.
Section~\ref{S:genus 3} specializes the approach 
to the case of non-hyperelliptic genus-$3$ curves,
and ends with several examples.

\section{Notation}
\label{S:notation}

If $S$ is a set and $n \in \Z_{\ge 0}$, let $\binom{S}{n}$ denote
the set of $n$-element subsets of~$S$.
For each field~$k$, choose a separable closure~$k_s$
(compatibly, when possible), and let $\calG=\calG_k=\Gal(k_s/k)$.
In general, for an object~$X$ over~$k$, we denote by~$X_s$ its base
extension to~$k_s$.
If $k$ is a global field, 
let $\Omega_k$ be the set of nontrivial places of~$k$.
For $v \in \Omega_k$, let $k_v$ be the completion of~$k$ at~$v$;
if moreover $v$ is non-archimedean, let $\calO_v$ be the valuation
ring in~$k_v$, let $\F_v$ be the residue field, and let $k_{v,u}$
be the maximal unramified extension of~$k_v$ inside a separable
closure~$k_{v,s}$.
All $\calG$-sets and $\calG$-modules are given the discrete topology,
and the $\calG$-action is assumed to be continuous.
If $M$ is a $\calG$-module,
then $M^\calG$ or~$M(k)$ denotes the
subgroup of $\calG$-invariant elements,
and $H^n(M)$ or $H^n(k,M)$ or~$H^n(\calG,M)$
denotes profinite group cohomology.

If $X$ is an integral $k$-scheme,
then $k(X)$ is its function field.
More generally, if $X$ is a disjoint union of integral $k$-schemes~$X_i$,
let $k(X)$ be the product of the~$k(X_i)$;
equivalently, $k(X)$ is the ring of global sections of the 
sheaf of total quotient rings (see \cite{Hartshorne1977}*{p.141}).
Let $\OO=\OO_X$ be the structure sheaf.

Call a variety \defi{nice} if it is smooth, projective, 
and geometrically integral.
Curves will be assumed nice unless otherwise specified.
For a nice variety~$X$,
let $\Div X$ (resp.\ $\Div^0 X$) be the group of divisors (resp.\ divisors
algebraically equivalent to $0$) on~$X$ over~$k$,
let $\calZ(X)$ (resp.\ $\calZ^0(X)$) be the group of $0$-cycles 
(resp.\ $0$-cycles of degree $0$),
and define $\Pic X$ as in \cite{Hartshorne1977}*{II.\S6};
if $X$ is a curve, 
also define $\Pic^0 X \colonequals  \ker(\Pic X \stackrel{\deg}\to \Z)$.
Alternatively, if $\Princ X$ is the group of principal divisors,
then $\Pic X = \Div X/\Princ X$.
If $f \in k(X)^\times$ and $z = \sum n_P P \in \calZ(X)$
is such that no closed point~$P$ appearing in~$z$ is a zero or pole of~$f$,
let $f(z) = \prod_P (N_{k(P)/k} f(P))^{n_P} \in k^\times$.
Let $J \colonequals \Alb_X$ be the Albanese variety of~$X$,
so $J$ is an abelian variety.
Then the Picard variety of~$X$
(i.e., the reduced subgroup scheme associated
to the connected component of the Picard scheme of~$X/k$)
may be identified with the dual abelian variety $\widehat{J}$.
If $X$ is a curve, then both $J$ and $\widehat{J}$ are the Jacobian $\Jac X$.
Let $\calY^0(X)$ be the kernel of the natural map $\calZ^0(X) \to J(k)$,
and let $J(k)_\circ$ be the image of this map.
More generally, if $G$ is a subgroup of~$J(k)$,
let $(J(k)/G)_\circ$ be the image of $\calZ^0(X) \to J(k)/G$.

\section{Twisted powers}
\label{S:twisted powers}

Fix a field~$k$.

\begin{definition}
\label{D:twisted power of G-module}
Given a $\calG$-module~$M$,
and a finite $\calG$-set~$\fes$,
the \defi{twisted power} $M^\fes$ 
is the $\calG$-module of maps from~$\fes$ to~$M$.
\end{definition}

\begin{remark}
The $\calG$-action on maps is the usual one:
if $\sigma \in \calG$ and $P \mapsto m_P$
is an element $m \in M^\fes$, then ${}^\sigma m$
is the map $P \mapsto \sigma(m_{\sigma^{-1} P})$.
\end{remark}

\begin{remark}
\label{R:Z^fes}
Applying the construction to~$\Z$ with trivial action yields~$\Z^\fes$.
Then $M^\fes = \Hom_\Z(\Z^\fes,M)$ for any $\calG$-module $M$.
\end{remark}

\begin{remark}
If $G$ is a commutative group scheme over~$k$,
we also use~$G$ to denote the $\calG$-module~$G(k_s)$,
and define~$G^\fes$ (at least as a $\calG$-module).
Similarly, a finite \'etale $k$-scheme~$\fes$
can be identified with a finite $\calG$-set.
\end{remark}

\begin{definition}
If $M$ is a $\calG$-module and $\fes$ is a finite $\calG$-set,
there is a homomorphism $\deg \colon M^\fes \to M$ that sums the coordinates,
and we let $M^\fes_{\deg 0}$ be its kernel.
\end{definition}

\begin{definition}
Given a finite $\calG$-module $M$ of size not divisible by~$\Char k$,
the \defi{Cartier dual} of~$M$ is the $\calG$-module
$M^\vee\colonequals \Hom_\Z(M,k_s^\times)$.
(This is compatible with the notion for finite commutative group schemes.)
\end{definition}

\begin{remark}
\label{R:exact}
For fixed~$\fes$, the functor $M \mapsto M^\fes$ is exact.
\end{remark}

\begin{remark}
\label{R:Cartier dual}
For a finite $\calG$-module~$M$ of size not divisible by~$\Char k$,
and a finite $\calG$-set~$\fes$,
we have $(M^\fes)^\vee \isom (M^\vee)^{\fes}$.
\end{remark}

Each finite \'etale $k$-scheme~$\fes$ is
$\Spec L$ for some \'etale $k$-algebra $L$.
Define the \'etale $k_s$-algebra $L_s \colonequals  L \tensor_k k_s$.
Thus $\G_a^\fes(k)=L$ and $\G_a^\fes(k_s)=L_s$.
Assume $\Char k \nmid n$.
Then $\mu_n^\fes(k_s) = \mu_n(L_s)$.
The group $H^1(\G_m^\fes)=H^1(\calG,L_s^\times)$ is trivial
by a generalization of Hilbert's theorem~90
\cite{SerreLocalFields1979}*{p.~152, Exercise~2},
so $H^1(\mu_n^\fes)=L^\times/L^{\times n}$.

\pagebreak[2]
\begin{samepage}
\section{Weil pairings}
\label{S:Weil pairings}

Let $k$ be a field.
Let $n$ be a positive integer with $\Char k \nmid n$.

\subsection{The Albanese-Albanese definition}
\end{samepage}

In this section we review Lang's construction of the Weil pairing
between Albanese varieties.
Let $V$ and~$W$ be nice $k$-varieties
and let $\fD\in\Div(V\times W)$.
The divisor $\fD$ induces partial maps
\[
	\fD \colon \calZ^0(V_s)\to \Div^0 W_s \textup{ and } 
	\fD^t \colon \calZ^0(W_s)\to \Div^0 V_s.
\]
Summation of $0$-cycles in $\Alb_V(k_s)$ gives rise to the exact sequence
\[
	0\to\calY^0(V_s)\to \calZ^0(V_s)\to \Alb_V(k_s) \to 0.
\]
By \cite{LangAbelianVarieties}*{III,~Theorem~4,~Corollary~2},
$\fD(\calY^0(V_s))\subset \Princ W_s$.
In particular, if $v \in \calZ^0(V_s)$ maps to 
an $n$-torsion point $[v] \in \Alb_V[n](k_s)$,
then $nv \in \calY^0(V_s)$, and $\fD(nv) = \divisor(f_{nv})$ for some
$f_{nv} \in k_s(W)^\times$.
Define $f_{nw} \in k_s(V)^\times$ symmetrically.
Define
\[\begin{array}{cccccc}
e_{\fD,n} \colon & \Alb_V[n](k_s)&\times&\Alb_W[n](k_s)&\to&\mu_n(k_s)\\
&[v]&,&[w]&\mapsto&\dfrac{f_{nw}(v)}{f_{nv}(w)}
\end{array}\]
where $v \in \calZ^0(V_s)$ mapping to~$[v]$
and $w \in \calZ^0(W_s)$ mapping to~$[w]$
are chosen so that the evaluations make sense.
See \cite{LangAbelianVarieties}*{VI,~\S4} 
for the proof that $e_{\fD,n}$ 
is well-defined, bilinear, and Galois-equivariant.

\begin{remark}\label{R:pairingduality}
Let $A$ be an abelian variety, let $\widehat{A}$ be the dual abelian variety,
and continue to suppose that $\Char k \nmid n$.
Take $V=A$ and $W=\widehat{A}$, and let $\fD$ be a Poincar\'e divisor.
Since $\Alb_A=A$ and $\Alb_{\widehat{A}}=\widehat{A}$, 
we obtain a pairing $e_n \colon A[n]\times \widehat{A}[n]\to \mu_n$.
It is nondegenerate (see \cite{LangAbelianVarieties}*{VI, VII}), 
so we obtain an identification of $\widehat{A}[n]$ with $A[n]^\vee$.
\end{remark}

\subsection{The Albanese-Picard definition}
\label{S:Albanese-Picard}

Let $X$ be a nice $k$-variety.
Let $J \colonequals \Alb_X$.
Let $\fP$ be a Poincar\'e divisor in $\Div(J\times \widehat{J})$.
Fix a base point in~$X_s$ to obtain a map $\iota \colon X_s \to J_s$.
The functoriality of taking Albanese varieties yields 
$\iota(\calY^0(X_s))\subset\calY^0(J_s)$.
Define $\fP_0\colonequals (\iota\times \id_{\widehat{J}})^*\fP\in\Div(X_s\times \widehat{J}_s)$.
If $y \in \calY^0(X_s)$ and $z \in \calZ^0(\widehat{J}_s)$, 
then $\fP_0(y)$ (if defined) is the divisor of some
rational function on~$\widehat{J}_s$, which can be evaluated on~$z$
(if the supports are disjoint) to obtain a value $\fP_0(y,z) \in k_s^\times$.
Given $D \in \Div^0(X_s)$, we may find $z \in \calZ^0(\widehat{J}_s)$
summing to $[D] \in \Pic^0(X_s) = \widehat{J}_s(k_s)$,
which means that $D-\fP_0^t(z)=\divisor(g_{D,z})$ 
for some $g_{D,z}\in k_s(X)^\times$,
and we obtain a partially-defined pairing
\[
\begin{array}{cccccc}
[\;,\;]\colon &\calY^0(X_s)&\times&\Div^0(X_s)&\to&k_s^\times\\
&y&,&D&\mapsto&g_{D,z}(y) \fP_0(y,z)
  \end{array}\\
\]
See \cite{Poonen-Stoll1999}*{Section~3.2} 
for a proof that this pairing is independent of the choices made, 
using Lang reciprocity and the seesaw principle.
If $g \in k_s(X)^\times$, then $[y,\divisor(g)] = g(y)$
since the choice $z=0$ yields $g_{D,z}=g$.

Finally, define
\[
\begin{array}{cccc}
e_{X,n} \colon &J[n] \times \widehat{J}[n] &\to&\mu_n(k_s)\\
&([x],[D])&\mapsto&\dfrac{ f_{nD}(x)}{[nx,D]},
\end{array}
\]
where $x \in \calZ^0(X_s)$ represents an element of~$J[n](k_s)$,
and $D \in \Div^0 X_s$ represents an element of~$\widehat{J}[n](k_s)$,
and $f_{nD} \in k_s(X)^\times$ has divisor~$nD$,
and all these are chosen so that everything
is defined.

Let us check that $e_{X,n}$ is well-defined,
i.e., independent of the choices of~$x$ and~$D$.
Changing~$x$ means adding some $y \in \calY^0(X_s)$ to it,
and we have
\[
	e_{X,n}([y],D)
	= \dfrac{f_{nD}(y)}{[ny,D]}
	= \dfrac{f_{nD}(y)}{[y,nD]}
	= \dfrac{f_{nD}(y)}{f_{nD}(y)}
	= 1.
\]
Changing $D$ means adding $\divisor(f)$ to it for some $f \in k_s(X)^\times$,
and we have
\[
	e_{X,n}(x,\divisor(f))
	= \dfrac{f^n(x)}{[nx,\divisor(f)]}
	= \dfrac{f(x)^n}{f(x)^n}
	= 1.
\]

\subsection{Functoriality}

Let $\iota \colon X\to X'$ be a morphism of nice varieties.
Let $x\in \calY^0(X_s)$ and $D\in \Div^0(X'_s)$. 
It is straightforward to check that
\[
	[x,\iota^*D]=[\iota(x),D]
\]
whenever both sides are defined.
It follows that for $[x]\in \Alb_{X_s}[n]$ and $[D]\in(\Pic X_s)[n]$ we have
\[
	e_{X,n}([x],\iota^* [D])=e_{X',n}([\iota(x)],[D]).
\]

\subsection{Equality of the pairings}

Let $J \colonequals \Alb_X$. 
We now have three pairings
\[
	e_{X,n}, e_{J,n}, e_n \colon J[n]\times \widehat{J}[n]\to \mu_n,
\]
where the first two are from Section~\ref{S:Albanese-Picard}
and the third is from Remark~\ref{R:pairingduality}.

\begin{proposition}
The pairings $e_{X,n}$, $e_{J,n}$, and $e_n$ are equal.
\end{proposition}

\begin{proof}
Functoriality with respect to an Albanese embedding $X_s \to \Alb_{X_s}$
shows that $e_{X,n} = e_{J,n}$.

Now we prove that $e_{J,n}=e_n$.
Suppose that $a\in J[n](k_s)$ and $a' \in \widehat{J}[n](k_s)$
are represented by appropriate $z \in\calZ^0(J_s)$ and 
$z' \in \calZ^0(\widehat{J}_s)$.
Let $D = \fP^t z' \in\Div^0(J_s)$. 
By definition, $f_{nz'}=f_{nD}$ and
\[
	e_{J,n}(a,a') 
	= \frac{f_{nD}(z)}{[nz,D]}
	= \frac{f_{nz'}(z)}{\fP_0(nz,z')}
	=\frac{f_{n z'}(z)}{f_{nz}(n z')}
	= e_n(a,a').\qedhere
\]
\end{proof}

\section{Odd theta characteristics}
\label{S:odd theta characteristics}

Assume $\Char k \ne 2$.
Let $X$ be a nice curve of genus~$g$ over~$k$.
Let $\omega$ be its canonical bundle, and let $J \colonequals \Jac X$.

\subsection{Theta characteristics}
\label{S:theta characteristics}

A \defi{theta characteristic} on~$X$
is a line bundle~$\vartheta$ on~$X$
such that $\vartheta^{\tensor 2} \isom \omega$.
A theta characteristic~$\vartheta$ is called \defi{odd}
if the nonnegative integer $h^0(\vartheta)\colonequals \dim H^0(X,\vartheta)$ is odd.
The isomorphism classes of theta characteristics on~$X_s$
form a set~$\TT$ of size~$2^{2g}$,
and the odd ones form a subset~$\TT_{\odd}$ of size~$2^{g-1}(2^g-1)$
(cf.~\cite{Mumford1971}).

\subsection{Theta characteristics and quadratic forms}
\label{S:thetas and quadratic forms}

Given a symplectic pairing~$e$ on an $\F_2$-vector space $V$,
a \defi{quadratic form associated to~$e$}
is a map of sets $q \colon V \to \F_2$ 
such that $q(x+y)-q(x)-q(y)=e(x,y)$ for all $x,y \in V$.

\begin{theorem}[Riemann-Mumford]
\label{T:Riemann-Mumford}
Suppose that $k$ is separably closed and $\Char k\ne 2$.
View the Weil pairing~$e_2$ as a symplectic pairing on~$J[2]$ with values
in $\F_2 \isom \{\pm 1\}$.
\begin{enumerate}[\upshape (a)]
\item 
For each theta characteristic $\vartheta$,
\begin{align*}
	q_\vartheta \colon J[2] &\to \F_2 \\
	\LL &\mapsto (h^0(\vartheta \tensor \LL) + h^0(\vartheta)) \bmod 2
\end{align*}
is a quadratic form associated to~$e_2$.
\item 
The map
\begin{align*}
        \TT &\To \{\textup{quadratic forms on $J[2]$ associated to $e_2$}\} \\
	\vartheta &\longmapsto q_\vartheta,
\end{align*}
is a bijection.
\item 
A theta characteristic~$\vartheta$ is odd
if and only if 
the Arf invariant of~$q_\vartheta$ is~$1$.
\end{enumerate}
\end{theorem}

\begin{proof}
See \cite{Mumford1971} and \cite{Gross-Harris2004}*{\S4}.
\end{proof}

Combining Theorem~\ref{T:Riemann-Mumford}
with the following lemma will help us understand 
the relations between the odd theta characteristics 
in the $\F_2$-vector space $\frac{\Pic X}{\langle \omega \rangle}[2]$: 
see Corollary~\ref{C:sums of thetas}.

\begin{lemma}
\label{L:polynomial}
Let $f$ be in the space $\F_2[x_1,\ldots,x_n]_{\le 2}$ 
of polynomials of total degree at most~$2$.
Let $f_2$ be its homogeneous part of degree~$2$.
  \begin{enumerate}[\upshape (i)]
  \item \label{I:f_2 is square}
If $f$ vanishes on $\F_2^n$, then $f_2$ is a square.
\item \label{I:dim V}
If $n$ is even, and $f_2$ is nondegenerate (as a quadratic form),
and $V \subseteq \F_2^n$ is a coset of a linear subspace of dimension
at least~$n/2+1$,
then $f(v)=0$ for some $v \in V$.
\item \label{I:x,x+v}
If $n$ is even and $n \ge 4$ and $v \in \F_2^n$, 
and $f_2$ is nondegenerate,
then there exists $x \in \F_2^n$ with $f(x)=f(x+v)=0$.
\item \label{I:x,x+v_1,x+v_2}
If $n$ is even and $n \ge 6$ and $v_1,v_2 \in \F_2^n$, 
and $f_2$ is nondegenerate,
then there exists $x \in \F_2^n$ with $f(x)=f(x+v_1)=f(x+v_2)=0$.
  \end{enumerate}
\end{lemma}

\begin{proof}\hfill
  \begin{enumerate}[\upshape (i)]
  \item 
If $f_2$ is not a square, it contains a monomial $x_i x_j$ with $i \ne j$,
and then restricting to the span of the $x_i$- and $x_j$-axes
lets us reduce to the case~$n=2$, which is easy.
\item 
Replacing $f(x)$ by~$f(x+v)$ lets us assume that $V$ is a subspace.
Let $e$ be the symmetric bilinear pairing associated to~$f_2$.
Since $e$ is nondegenerate, 
$e|_V$ has kernel of size at most $n-\dim V < \dim V$,
so $e|_V \ne 0$.
Hence $f_2|_V$ is not a square.
Apply \eqref{I:f_2 is square} to $f+1$.
\item 
If $v \ne 0$, let $V \subseteq \F_2^n$ be 
the codimension-$1$ coset defined by $f(x+v)-f(x)=0$;
if $v=0$, let $V=\F_2^n$.
Apply \eqref{I:dim V}.
\item
By \eqref{I:x,x+v}, we may assume that $v_1,v_2$ are distinct and nonzero.
Let $V$ be the codimension-$2$ coset 
defined by $f(x+v_1)-f(x)=0$ and $f(x+v_2)-f(x)=0$.
Apply \eqref{I:dim V}.\qedhere
  \end{enumerate}
\end{proof}

In the rest of Section~\ref{S:thetas and quadratic forms} 
except in Corollary~\ref{C:Galois action on theta characteristics},
we assume that $k$ is a separably closed field of characteristic not~$2$,
and that $X$ is a nice curve of genus $g \ge 2$ over~$k$.

\pagebreak[2]
\begin{samepage}
\begin{corollary}
\label{C:sums of thetas}
\strut
  \begin{enumerate}[\upshape (a)]
  \item \label{I:sum of 2 thetas}
If $g \ge 2$,
then every class in $J[2] \subset \frac{\Pic X}{\langle \omega \rangle}[2]$
has the form $\vartheta_1 + \vartheta_2$ (modulo $\omega$)
with $\vartheta_1,\vartheta_2 \in \TT_{\odd}$.
\item  \label{I:sum of 6 thetas}
If $g \ge 3$, and $\vartheta_1,\ldots,\vartheta_6 \in \TT_{\odd}$
sum to $0$ in $\frac{\Pic X}{\langle \omega \rangle}[2]$,
then there exist $\vartheta_{12},\vartheta_{34},\vartheta_{56} \in \TT_{\odd}$
such that 
\begin{align*}
  \vartheta_1 + \vartheta_2 + \vartheta_{34} + \vartheta_{56} &= 0 \\
  \vartheta_3 + \vartheta_4 + \vartheta_{12} + \vartheta_{56} &= 0 \\
  \vartheta_5 + \vartheta_6 + \vartheta_{12} + \vartheta_{34} &= 0
\end{align*}
in $\frac{\Pic X}{\langle \omega \rangle}[2]$.
  \end{enumerate}
\end{corollary}
\end{samepage}

\begin{proof}
Fix $\vartheta \in \TT_{\odd}$.
By Theorem~\ref{T:Riemann-Mumford}, there is an identification
$\{\textup{zeros of $q_\vartheta$ in $J[2]$}\} \to \TT_{\odd}$
given by $x \mapsto \vartheta+x$.
Identify $J[2]$ with~$\F_2^{2g}$, and $q_\vartheta$ with a polynomial~$f$.
  \begin{enumerate}[\upshape (a)]
  \item Apply Lemma~\ref{L:polynomial}\eqref{I:x,x+v} 
with $v$ the class in~$J[2]$.
Then $x$ corresponds to the desired $\vartheta_1$,
and $x+v$ to~$\vartheta_2$.
  \item Apply Lemma~\ref{L:polynomial}\eqref{I:x,x+v_1,x+v_2} 
with $v_1=\vartheta_1+\vartheta_2$ and $v_2=\vartheta_3+\vartheta_4$
to get~$x$.
Then $x$ corresponds to the desired $\vartheta_{56}$,
$x+v_1$ corresponds to the desired $\vartheta_{34}$,
and $x+v_2$ corresponds to the desired $\vartheta_{12}$.
The first two relations are then satisfied,
and the third is the sum of the first two.\qedhere
  \end{enumerate}
\end{proof}

By an \defi{incidence structure},
we will mean a pair~$(\fes,\Sigma)$,
where $\fes$ is a set and $\Sigma$ is a collection of subsets of~$\fes$.
An \defi{isomorphism} $(\fes,\Sigma) \to (\fes',\Sigma')$
is a bijection $\fes \to \fes'$ under which $\Sigma$ and $\Sigma'$ 
correspond.

Let $\Sigma$ be the set of $4$-element subsets of $\TT_{\odd}$
that sum to~$0$ in $\frac{\Pic X}{\langle \omega \rangle}[2]$.

\begin{proposition}
\label{P:structures}
The following structures have the same automorphism group $\Sp_{2g}(\F_2)$:
\begin{enumerate}[\upshape (1)]
\item The $\F_2$-vector space $\frac{\Pic X}{\langle \omega \rangle}[2]$
equipped with $\TT_{\odd}$ (viewed as a subset).
\item The $\F_2$-vector space $J[2]$ with the Weil pairing $e_2$.
\item The incidence structure $(\TT_{\odd},\Sigma)$.
\end{enumerate}
\end{proposition}

\begin{proof}
The automorphism group of~(2) is $\Sp_{2g}(\F_2)$,
so it suffices to show how to build each structure 
canonically in terms of the others.

(1)$\to$(2): 
By Corollary~\ref{C:sums of thetas}\eqref{I:sum of 2 thetas},
we can recover~$J[2]$ as the subgroup of
$\frac{\Pic X}{\langle \omega \rangle}[2]$ 
generated by~$\vartheta_1+\vartheta_2$
with $\vartheta_1,\vartheta_2 \in \TT_{\odd}$.
The subset $\TT_{\odd}$ determines the function
$h^0 \bmod 2$ on the nontrivial coset $\TT$ of $J[2]$ 
in $\frac{\Pic X}{\langle \omega \rangle}[2]$.
Choose $\vartheta \in \TT_{\odd}$;
from $h^0 \bmod 2$, we can recover $q=q_\vartheta$ 
and hence $e_2(x,y)=q(x+y)-q(x)-q(y)$.

(2)$\to$(1): Take the space of functions $q \colon J[2] \to \F_2$
such that $(x,y) \mapsto q(x+y)-q(x)-q(y)$ is a multiple of~$e_2$.
Inside this we have the subset of $q$ for which the multiple is~$e_2$ itself
and for which the Arf invariant is~$1$.

(1)$\to$(3): Clear.

(3)$\to$(1): 
If $g=2$, take the $\F_2$-vector space $P$ with generator set $\TT_{\odd}$
and with one relation saying that the sum of the generators is~$0$.
If $g \ge 3$,
take the $\F_2$-vector space $P$ with generator set~$\TT_{\odd}$
and with relations given by the elements of~$\Sigma$.

To show that the map 
$\epsilon \colon P \to \frac{\Pic X}{\langle \omega \rangle}[2]$
sending each basis element to the corresponding~$\vartheta$
is an isomorphism, it suffices to show that its restriction 
$\epsilon_0 \colon P_0 \to J[2]$
is an isomorphism, where $P_0$ is the codimension-$1$ subspace of~$P$
spanned by pairs.
Corollary~\ref{C:sums of thetas}\eqref{I:sum of 2 thetas}
shows that $\epsilon_0$ induces a bijection 
$\binom{\TT_{\odd}}{2}/\!\!\sim \; \To J[2] - \{0\}$,
where $\{\vartheta_1,\vartheta_2\} \sim \{\vartheta_3,\vartheta_4\}$
means $\sum_{i=1}^4 \vartheta_i = 0$ 
in $\frac{\Pic X}{\langle \omega \rangle}[2]$.
In particular, $\epsilon_0$ is surjective.
If $g=2$, then $\epsilon_0$ is injective too, 
because $\dim P_0 = 4 = \dim J[2]$.
To prove injectivity for $g \ge 3$,
it suffices to prove that every relation
\[
	(\vartheta_1+\vartheta_2) + (\vartheta_3+\vartheta_4) 
	= (\vartheta_5+\vartheta_6) 
\]
in $\frac{\Pic X}{\langle \omega \rangle}[2]$
is a consequence of $4$-term relations.
But that is true, 
by Corollary~\ref{C:sums of thetas}\eqref{I:sum of 6 thetas}.
\end{proof}

\begin{corollary}
\label{C:Galois action on theta characteristics}
Let $k$ be any field of characteristic not~$2$.
The action of $\calG$ on $J[2]$, on the set of theta characteristics of~$X_s$
or on the set of odd theta characteristics of~$X_s$,
factors through the standard action of $\Sp_{2g}(\F_2)$ on~$J[2]$.
\end{corollary}

\begin{remark}
\label{R:connectedness of moduli space}
It follows from the connectedness of the moduli space of genus-$g$ curves
\cite{Deligne-Mumford1969} 
that the isomorphism type of each structure in Proposition~\ref{P:structures}
depends only on~$g$, and not on $k$ or~$X$.
\end{remark}

\begin{proposition}
\label{P:315}
For $g \ge 2$, 
we have $\#\Sigma=\frac{2^{g-3}}{3}(2^{2g}-1)(2^{2g-2}-1)(2^{g-2}-1)$.
\end{proposition}

\begin{proof}
Since $\Sp_{2g}(\F_2)$ acts transitively on $J[2]-\{0\}$,
all fibers of the summing map 
$\binom{\TT_{\odd}}{2} \to J[2]-\{0\}
\subseteq \frac{\Pic X}{\langle \omega \rangle}[2]$
have the same size,
namely $\binom{2^{g-1}(2^g-1)}{2} / (2^{2g}-1) = 2^{2g-3}-2^{g-2}$.
Hence the number of {\em pairs} of pairs such that the two pairs
have the same image in $J[2]-\{0\}$
is $(2^{2g}-1)\binom{2^{2g-3}-2^{g-2}}{2}$.
Each such pair of pairs consists of disjoint pairs,
since $x+y=x+z$ would imply $y=z$.
Thus each pair of pairs corresponds to a $4$-element subset of $\TT_{\odd}$
summing to $0$ with a partition into two pairs.
Each $4$-element subset can be partitioned in $3$~ways,
so 
\[
	\#\Sigma = \frac13 (2^{2g}-1)\binom{2^{2g-3}-2^{g-2}}{2} 
	= \frac{2^{g-3}}{3}(2^{2g}-1)(2^{2g-2}-1)(2^{g-2}-1).\qedhere
\]
\end{proof}

\subsection{Representing odd theta characteristics by divisors over the ground field}

\begin{proposition}
\label{P:rational divisor for theta characteristic}
An odd theta characteristic~$\vartheta$ on~$X_s$
whose class lies in $(\Pic X_s)^\calG$
is represented by an element of~$\Div X$.
\end{proposition}

\begin{proof}
The $k$-scheme parametrizing effective divisors 
whose class equals~$\vartheta$
is a Brauer-Severi variety,
a twisted form of $\PP^{h^0(\vartheta)-1}$,
corresponding to a central simple algebra of 
dimension $h^0(\vartheta)^2$ and hence of index dividing 
$h^0(\vartheta)$ \cite{Gille-Szamuely2006}*{Theorems 5.2.1 and~2.4.3},
so the associated Brauer class $\tau \in \Br k$
is killed by the odd integer 
$h^0(\vartheta)$ \cite{Gille-Szamuely2006}*{Proposition~4.5.13(1)}.
On the other hand, the Hochschild-Serre spectral sequence
yields an exact sequence
\[
	\Pic X \to (\Pic X_s)^\calG \to \Br k
\]
under which the second map sends $\vartheta$ to~$\tau$
and $\omega$ to~$0$,
so $2\tau=0$.
Thus $\tau=0$.
So $\vartheta$ comes from an element of~$\Pic X$,
or, equivalently, from an element of~$\Div X$.
\end{proof}

\pagebreak[2]
\begin{samepage}
\section{Generalized explicit descent}
\label{S:generalized explicit descent}

\subsection{The setting}
\label{S:the setting}
Suppose that $X$ is a nice variety over~$k$.
Let $J \colonequals \Alb_X$.

\subsection{True descent}\label{S:true descent}

To motivate the definition of a true descent setup, 
we recall the following:
\end{samepage}

\begin{example}
\label{E:odd hyperelliptic}
Assume $\Char k \ne 2$.
Let $F(x) \in k[x]$ be a non-constant separable polynomial of odd
degree~$2g+1$.
Let $X$ be the smooth projective model of the curve $y^2=F(x)$,
so $X$ is a hyperelliptic curve of genus~$g$.
Let $\infty$ be the unique point at infinity on~$X$,
and let $\fes \subset X$
be the $0$-dimensional $k$-scheme 
such that $\fes(k_s)$ consists of
the $2g+1$ Weierstrass points not equal to~$\infty$.
Then there is a surjection $(\Z/2\Z)^\fes \to \widehat{J}[2]$
sending each basis element $P \in \fes(k_s)$ 
to the divisor class~$[P-\infty]$
(see \cite{MumfordTheta2}*{Lemma~2.4 and Corollary~2.11}).
The family of divisors $P-\infty$ indexed by $P \in \fes(k_s)$
may be viewed as a single divisor~$\beta$ on~$X \times \fes$.
\end{example}

\begin{definition}
\label{D:true setup}
A \defi{true descent setup} for~$X$ consists of a triple $(n,\fes,\LL)$,
where $n$ is a positive integer not divisible by~$\Char k$,
$\fes=\Spec L$ is a finite \'etale $k$-scheme,
and $\LL$ is a line bundle on~$X \times \fes$
such that $\LL^{\tensor n} \isom \OO$.
In more concrete terms, if we choose $\beta \in \Div(X \times \fes)$
representing~$\LL$, the condition on~$\beta$ is that
$n\beta$ is principal; 
i.e., there is a function $f\in k(X\times \fes)^\times$ 
such that $\divisor(f)=n\beta$.
\end{definition}

\begin{remark}
Given $\beta$, for each $P \in \fes(k_s)$, we have the fiber 
$\beta_P \in \Div X_s$.
In fact, we may think of~$\beta$ as the family of divisors~$\beta_P$
depending $\calG$-equivariantly on~$P$.
Similarly, $f$ may be thought of as a family of 
functions $f_P \in k(X_s)^\times$.
\end{remark}

Given $(n,\fes,\LL)$,
we will define a homomorphism
\[
	C \colon J(k) \to \frac{L^\times}{L^{\times n}}
\]
using cohomology, and relate it to a more explicit homomorphism.

\subsubsection{Cohomological definition}
\label{S:true cohomological definition}

Take cohomology of 
\[
	0 \to J[n] \to J \stackrel{n}\to J \to 0
\]
to obtain $J(k) \to H^1(J[n])$.
The condition on~$\LL$ (or~$\beta$)
implies that it induces a map $\fes \to \widehat{J}[n]$,
and hence a homomorphism $(\Z/n\Z)^\fes \to \widehat{J}[n]$.
Taking the Cartier dual and
applying Remark~\ref{R:Cartier dual} yields
$\alpha\colon J[n]\to\mu_n^\fes$.
Composition yields a homomorphism
\begin{equation}
\label{E:true cohomological homomorphism}
	C \colon 
	J(k) \to H^1(J[n]) \stackrel{\alpha}{\longrightarrow}
        H^1(\mu_n^\fes)\isom\frac{L^\times}{L^{\times n}}.
\end{equation}

\subsubsection{Explicit definition}
\label{S:true explicit definition}

Choose $\beta$ and~$f$ as in Definition~\ref{D:true setup}.
Let $X^\good$ be the largest open subscheme of~$X$ such that $f$ is an
invertible regular function on~$X^\good \times \fes$.
Then $f$ defines a morphism $X^\good \times \fes \to \G_m$,
and hence a morphism $X^\good \to \G_m^\fes$.
Evaluating on closed points and taking norms (this is necessary
if the closed point is of degree greater than~$1$),
we obtain a homomorphism $\calZ(X^\good) \to L^\times$.
This induces $\calZ^0(X^\good) \to L^\times/L^{\times n}$.
If we change $(\beta,f)$ to~$(\beta',f')$,
then $\beta'-\beta$ is the divisor of some~$g \in k(X)^\times$,
and $f' = c g^n f$ for some~$c \in k^\times$.
If we evaluate on any $z \in \calZ^0(X)$ that is good for 
both $f$ and~$f'$,
then the value of the homomorphism in $L^\times/L^{\times n}$
is unchanged (the value of~$g$ gives an $n^{\tH}$~power,
and the value of~$c$ is $c^{\deg z}=1$).
Given any $z \in \calZ^0(X)$, we can move $\beta$ to avoid~$z$,
so the compatible homomorphisms glue to give a homomorphism
\begin{equation}
\label{E:true explicit homomorphism}
	\widetilde{C}\colon \calZ^0(X) \to \frac{L^\times}{L^{\times n}}.
\end{equation}

\subsubsection{Compatibility of the two definitions}
\label{S:true equivalence}

\begin{proposition}
\label{P:true equivalence}
The maps $C$ and~$\widetilde{C}$ are compatible
in the sense that the following diagram commutes:
\[\xymatrix{
\calZ^0(X)\ar@/^3.5ex/[drrr]^{\widetilde{C}}\ar[d]\\
J(k)\ar[r]&H^1(J[n])\ar[r]^-{\alpha}&
H^1(\mu_n^\fes)&
\dfrac{L^{\times}}{L^{\times n}}\ar[l]_-{\sim}.
}\]
\end{proposition}

\begin{proof}
We may fix $(\beta,f)$ and consider $\calZ^0(X^\good)$ instead of~$\calZ^0(X)$.
Let $\calZ^0=\calZ^0(X_s^\good)$ and let $\calY^0$ be the Albanese kernel 
in the exact sequence of $\calG$-modules
\[
	0 \to \calY^0 \to \calZ^0 \to J(k_s) \to 0.
\]
We will construct the following commutative diagram of $\calG$-modules
with exact rows
(we write $J$ for~$J(k_s)$, and so on):
\begin{equation}
\label{E:true nz+y diagram}
\begin{split}
\xymatrix{
0\ar[r]&J[n]\ar[r]\ar@/_12ex/[dd]_{\alpha}&J\ar[r]^{n}&J\ar[r]&0\\
0\ar[r]&
    \widetilde{J[n]}\ar[r]\ar[u]\ar[d]&
    \calZ^0\times\calY^0\ar[r]^-{nz+y}\ar[u]\ar[d]&
    \calZ^0\ar[r]\ar[u]\ar[d]^{\widetilde{C}}&
    0\\
0\ar[r]&\mu_n^\fes\ar[r]&\G_m^\fes \ar[r]^n&\G_m^\fes \ar[r]&0
}
\end{split}
\end{equation}
The top and bottom rows are familiar. 
Surjectivity of the map $\calZ^0\times\calY^0\to \calZ^0$ 
given by $(z,y) \mapsto nz+y$ 
follows from surjectivity of multiplication by~$n$ on $J(k_s)=\calZ^0/\calY^0$;
let $\widetilde{J[n]}$ be the kernel of the $nz+y$ map.
We have
\[
	\widetilde{J[n]}=\{(z,-nz)\in \calZ^0\times \calY^0: [z]\in J[n](k_s)\}.
\]
All the upward maps are induced by $\calZ^0\to J(k_s)$, 
so commutativity of the top part of the diagram is straightforward.

The middle and rightmost downward maps in \eqref{E:true nz+y diagram} 
are given by
\[
 \begin{array}{cccccc}
  \calZ^0 &\times& \calY^0&\To&\G_m^\fes & \\
  z&,&y&\longmapsto&  f(z) [y,\beta] &\colonequals  \left(f_P(z)\cdot[y,\beta_P]\right)_{P\in \fes(k_s)}\\[2ex]
& \widetilde{C}\colon & \calZ^0 &\To& \G_m^\fes\\
&& z&\longmapsto& f(z)
\end{array}
\]
The bottom right square commutes since
\[
	(f(z)\cdot [y,\beta])^n=f(nz)\cdot[y,n\beta]=f(nz)f(y)=\widetilde{C}(nz+y).
\]
The leftmost downward vertical map is obtained by restricting the middle one.

Finally, we check that $\alpha$ and the vertical maps in the first column
form a commutative triangle.
Since $\alpha^\vee$ sends $P$ to~$[\beta_P]$,
Remark~\ref{R:pairingduality} and the Albanese-Picard definition of~$e_n$
show that for any $[z] \in J[n]$,
\[
	\alpha([z]) 
	= e_n([z],[\beta]) 
	= \frac{ f(z) }{[nz,\beta]} 
	= f(z)\cdot [-nz,\beta],
\]
so the images of an element $(z,-nz) \in \widetilde{J[n]}$
under the two paths to~$\mu_n^\fes$ are equal.
This completes the construction of~\eqref{E:true nz+y diagram}.

Taking cohomology of \eqref{E:true nz+y diagram} yields
 \[\xymatrix{
J(k)\ar[r]^n&J(k)\ar[r]&H^1(J[n])\ar@/^10ex/[dd]^{\alpha}\\
&\calZ^0(X^\good)\ar[u]\ar[r]\ar[d]^{\widetilde{C}}&H^1(\widetilde{J[n]})\ar[u]\ar[d]\\
L^\times\ar[r]^n&L^\times\ar[r]& H^1(\mu_n^\fes). \\
}\]
Equality of the compositions from~$\calZ^0(X^\good)$
to~$H^1(\mu_n^\fes)$ is the desired result.
\end{proof}

\begin{remark}
\label{R:same image}
If $U$ is a dense open subscheme of~$X$, then by a moving lemma,
$\calZ^0(U)$ and $\calZ^0(X)$ 
have the same image in~$J(k)$.
(\emph{Proof:} It suffices to prove that any closed point~$x$ of~$X$
is rationally equivalent to a $0$-cycle supported on~$U$.
Pick a closed point $u \in U$.
By a classical Bertini theorem, 
or \cite{Poonen-bertini2004}*{Corollary~3.4} if $k$ is finite,
we can find a nice curve $C \subseteq X$ through $x$ and~$u$,
and hence reduce to the case in which $X$ is a curve.
In this case, it suffices to apply weak approximation 
to find a rational function on~$X$
with prescribed valuations at the points of~$X-U$.)
\end{remark}

\begin{corollary}
The explicit homomorphism $\widetilde{C}$ in~\eqref{E:true explicit homomorphism} 
induces a homomorphism 
\[
	\widetilde{C}\colon \left( \frac{J(k)}{nJ(k)} \right)_\circ 
	\to \frac{L^\times}{L^{\times n}}.
\]
\end{corollary}

\subsection{Fake descent}
\label{S:fake descent}

For computations using true descent setups to be practical,
the degrees of the components of~$\fes$ (i.e., the degrees of the field
extensions~$L_i$ whose product is~$L$)
should be not too large.
A fake descent setup will be a variant of the true descent setup, 
a variant that may allow a simpler~$\fes$ to be used,
at the expense of giving less direct information about Selmer groups.
In Section~\ref{S:true fake comparison} 
we show that the true descent setup can be viewed
as a special case of the fake descent setup.

\begin{definition}\label{D:fake setup}
A \defi{fake descent setup} for~$X$ consists of
a triple $(n,\fes,\LL)$,
where $n$ is a positive integer not divisible by~$\Char k$,
$\fes=\Spec L$ is a nonempty finite \'etale $k$-scheme, 
and $\LL$ is a line bundle on $X \times \fes$
such that $\LL^{\tensor n}$ is the pullback of a line bundle on~$X$.
Equivalently, 
in terms of a divisor $\beta\in\Div(X\times \fes)$ representing~$\LL$,
the condition is that there exists $D\in\Div X$ 
is such that $n\beta - (D\times\fes)$ is principal;
i.e., there is a function $f\in k(X\times \fes)^\times$ 
such that $\divisor(f) = n\beta-(D\times \fes)$.
\end{definition}

\begin{remark}
\label{R:difference is n-torsion}
On the fibers, the last condition says that $\divisor(f_P) = n \beta_P - D$
for all $P \in \fes(k_s)$.
Subtracting shows that 
$[\beta_P-\beta_Q] \in \widehat{J}[n]$ 
for all $P,Q \in \fes(k_s)$.
\end{remark}

\begin{examples}\label{Ex:fake setup examples}
Definition~\ref{D:fake setup} is motivated by
the following examples of fake descent setups for curves.
\begin{enumerate}[\upshape (i)]
\item
\label{E:hyperelliptic example}
Assume $\Char k \ne 2$.
Let $\pi \colon X \to \PP^1_k$ be a (ramified) degree-$2$ cover,
so $X$ is a hyperelliptic curve.
Let $n=2$, 
let $\fes \subset X$ be the ramification locus of~$\pi$,
and let $\beta$ be the diagonal copy of~$\fes$ in~$X \times \fes$.
Then we can take $D = \pi^*(y)$ for some $y \in \PP^1(k)$ 
(cf.~\cite{Flynn-Poonen-Schaefer1997}).
See Example~\ref{Ex:hyperelliptic descent} for more about this case.
\item
The previous example generalizes to other 
geometrically generically cyclic covers $\pi\colon X \to \PP^1_k$ 
(cf.~\cite{Poonen-Schaefer1997}).
Let $n \ge 2$.
Let $k$ be a field with $\Char k \nmid n$.
Suppose that $F(x) \in k[x]$ factors completely in~$k_s[x]$
and is not a $p^\tH$~power in~$k_s[x]$ for any $p \mid n$.
Let $X$ be the smooth projective model of $y^n=F(x)$.
Let $\pi\colon X \to \PP^1$ be the $x$-coordinate map.
Let $\fes$ be the ramification locus of~$\pi$.
Let $\beta$ be the diagonal copy of~$\fes$ in~$X \times \fes$.
Then we can take $D = \pi^*(t)$ for some $t \in \PP^1(k)$.
\item
Assume $\Char k \notin\{2,3,7\}$.
Let $X$ be a twist of the Klein quartic curve $x^3 y + y^3 z + z^3 x = 0$
in~$\PP^2$.
Let $n=2$,
let $\fes$ correspond to the $\calG$-set of $8$~triangles,
as defined in~\cite{Poonen-Schaefer-Stoll2007}*{\S 11.1},
and let $\beta \in \Div(X \times \fes)$
be the divisor of relative degree~$3$ over~$\fes$
such that each $\beta_P$ is the sum of the 3~points
in the corresponding triangle.
This gives a fake descent setup, provided that $D$ can be found
(as turned out to be the case for the curves of interest in
\cite{Poonen-Schaefer-Stoll2007}).
\item \label{E:odd theta}
Assume $\Char k \ne 2$.
Let $X$ be any curve of genus $g \ge 2$ over~$k$.
Let $n=2$.
Let $\fes$ correspond to the $\calG$-set
of odd theta characteristics on~$X_s$.
Proposition~\ref{P:rational divisor for theta characteristic}
applied over the residue fields of each point of~$\fes$
shows that one can find $\beta \in \Div(X \times \fes)$
such that each $\beta_P$ is the corresponding odd theta characteristic.
Then one can take $D$ to be a canonical divisor.
This gives a fake descent setup that in principle can be used 
to perform a $2$-descent on the Jacobian of~$X$.
When $g=2$, this specializes to Example~\eqref{E:hyperelliptic example}.
\end{enumerate}
\end{examples}

A fake descent setup $(n,\fes,\LL)$
will give rise to a cohomologically defined homomorphism
\begin{equation}
\label{E:fake cohomological homomorphism}
	C \colon J(k) \to H^1\left( \frac{\mu_n^\fes}{\mu_n} \right)
\end{equation}
that we will relate to an explicit homomorphism
\begin{equation}
\label{E:fake explicit homomorphism}
	\widetilde{C}\colon \calZ^0(X) \to \frac{L^\times}{L^{\times n} k^\times}.
\end{equation}

\subsubsection{Cohomological definition}\label{S:fake cohom def}

By the final statement in Remark~\ref{R:difference is n-torsion}, 
$\LL$ (or $\beta$) induces a homomorphism
\[
	\alpha^\vee\colon (\Z/n\Z)^\fes_{\deg 0}\to \widehat{J}[n].
\]
Dualizing
\[
	0\to(\Z/n\Z)^\fes_{\deg 0}\to(\Z/n\Z)^\fes\to\Z/n\Z\to 0
\]
yields $((\Z/n\Z)^\fes_{\deg 0})^\vee \isom \mu_n^\fes/\mu_n$,
so the dual of $\alpha^\vee$ is a homomorphism
\[
	\alpha \colon J[n] \to \frac{\mu_n^\fes}{\mu_n}.
\]
The composition
\[
	J(k) \longrightarrow 
	H^1(J[n]) \stackrel{\alpha}\longrightarrow 
	H^1\left( \frac{\mu_n^\fes}{\mu_n} \right)
\]
is the desired homomorphism~$C$.

\subsubsection{Explicit definition}
\label{S:fake explicit definition}

As in Section~\ref{S:true explicit definition},
$f$ gives rise to a homomorphism $\calZ(X^\good) \to L^\times$.
If we change $(\beta,D,f)$ to $(\beta',D',f')$,
then $\beta'-\beta$ is the divisor of some $g \in k(X \times \fes)^\times$,
$D'-D$ is the divisor of some $h \in k(X)^\times$,
and $f' = (c g^n / h) f$ for some $c \in k(\fes)^{\times} = L^\times$.
If we evaluate on any $z \in \calZ^0(X)$ that is good for 
both $f$ and~$f'$,
then the value of the homomorphism in $L^\times/L^{\times n} k^\times$
is unchanged (the value of $g^n/h$ gives an element of $L^{\times n} k^\times$,
and the value of $c$ is $c^{\deg z} = c^0 = 1$).
Thus we obtain~\eqref{E:fake explicit homomorphism}.

\subsubsection{Compatibility of the two definitions}
\label{S:fake equivalence}

Taking cohomology of
\[
	0 \to \mu_n \to \mu_n^\fes \to \frac{\mu_n^\fes}{\mu_n} \to 0
\]
yields
\begin{equation}
  \label{E:Br}  
\frac{k^\times}{k^{\times n}} \to \frac{L^\times}{L^{\times n}} \to
H^1\left(\frac{\mu_n^\fes}{\mu_n}\right) \to \Br k,
\end{equation}
so we may identify $L^\times/L^{\times n} k^\times$
with a subgroup of $H^1(\mu_n^\fes/\mu_n)$.

\begin{proposition}
\label{P:fake equivalence}
The two maps $C$ and~$\widetilde{C}$ are compatible
in the sense that the following diagram commutes:
\[\xymatrix{
\calZ^0(X)\ar@/^3.5ex/[drrr]^{\widetilde{C}}\ar[d]\\
J(k) \ar[r] & H^1(J[n]) \ar[r]^-{\alpha}&
H^1\left(\dfrac{\mu_n^\fes}{\mu_n}\right)&
\dfrac{L^{\times}}{L^{\times n}k^\times}\ar@{_(->}[l]
}\]
\end{proposition}

\begin{proof}
We replace the pairing $[y,\beta]$ used in the true case
with
\begin{align}
\label{E:fake bracket}
   \calY^0 &\To \frac{\G_m^{\fes}(k_s)}{\mu_n(k_s)} \\
\nonumber
	y &\longmapsto [y,\beta]_D\colonequals  \frac{[y,\beta-H]}{h(y)^{1/n}}
\end{align}
where $H \in \Div(X_s)$ and $h \in k_s(X)^\times$ 
are chosen so that $\divisor(h)=D-nH$, and $h(y)^{1/n}$ 
is a chosen $n^\tH$~root of~$h(y)$
(and then $H$ and~$h(y)^{1/n}$ are pulled back to~$(X \times \fes)_s$);
such $H$ and~$h$ exist
because one possibility is to take $H=\beta_P$ for some $P \in \fes(k_s)$.
If we change $H$ to another choice~$H'$,
then $h$ is changed to~$hj$ where $\divisor(j)=nH-nH'$,
so the value of~\eqref{E:fake bracket}
is multiplied by $[y,H-H']/j(y)^{1/n}$,
which lies in~$\mu_n(k_s)$ since its $n^\tH$~power equals
$[y,\divisor(j)]/j(y)=1$.
Thus \eqref{E:fake bracket} is well-defined and Galois-equivariant.
Also,
\begin{equation}
\label{E:nth power}
	[y,\beta]_D^n 
	= \frac{[y,n\beta-nH]}{h(y)} 
	= \frac{[y,\divisor(fh)]}{h(y)} 
	= \frac{(fh)(y)}{h(y)}
	= f(y).
\end{equation}
As in the proof of Proposition~\ref{P:true equivalence},
we construct a commutative diagram with exact rows:
\begin{equation}
\label{E:fake nz+y diagram}
\begin{split}
\xymatrix{
0\ar[r]&J[n]\ar[r]\ar@/_14ex/[dd]_{\alpha}&J\ar[r]^{n}&J\ar[r]&0\\
0\ar[r]&
    \widetilde{J[n]}\ar[r]\ar[u]\ar[d]&
    \calZ^0\times\calY^0\ar[r]^-{nz+y}\ar[u]\ar[d]&
    \calZ^0\ar[r]\ar[u]\ar[d]^{\widetilde{C}}&
    0\\
0\ar[r]&\dfrac{\mu_n^\fes}{\mu_n}\ar[r]&\dfrac{\G_m^\fes}{\mu_n} \ar[r]^n&\G_m^\fes \ar[r]&0.
}
\end{split}
\end{equation}
The first two rows and the vertical maps between them are the same
as in~\eqref{E:true nz+y diagram}.
The bottom row is a pushout of the bottom row of~\eqref{E:true nz+y diagram}.

The middle downward map is
\[
 \begin{array}{ccccc}
  \calZ^0 &\times& \calY^0 &\To& \dfrac{\G_m^\fes}{\mu_n}\\[1ex]
  z&,&y&\longmapsto& f(z) [y,\beta]_D \\
 \end{array}
\]
The rightmost downward map is
\[
 \begin{array}{cccc}
 \widetilde{C}\colon&\calZ^0(X^\good_s)&\To& \G_m^\fes\\[1ex]
 & z & \longmapsto & f(z)
 \end{array}
\]
The bottom right square commutes since \eqref{E:nth power} implies
\[
	\left( f(z) [y,\beta]_D \right)^n
	=f(z)^n f(y)
	=f(nz+y).
\]
The leftmost downward vertical map is obtained by restricting the middle one.

Finally, we check that $\alpha$ and the vertical maps in the first column
form a commutative triangle.
Any $(z,-nz) \in \widetilde{J[n]}$ maps up to~$[z] \in J[n]$,
and maps right and down to 
$f(z)\cdot [-nz,\beta]_D \in \frac{\G_m^\fes}{\mu_n}$.
Remark~\ref{R:pairingduality} and the Albanese-Picard definition of~$e_n$
show that
\[
	\alpha([z]) 
	= e_n([z],[\beta-H]) 
	= \frac{f(z)h(z)}{[nz,\beta-H]}
	= f(z)\cdot [-nz,\beta]_D \quad \in \frac{\G_m^\fes}{\mu_n}.
\]
So the triangle commutes.
This completes the construction of~\eqref{E:fake nz+y diagram}.

Taking cohomology of \eqref{E:fake nz+y diagram} yields
\[\xymatrix{
J(k)\ar[r]^n&J(k)\ar[r]&H^1(J[n])\ar@/^10ex/[dd]^{\alpha}\\
&\calZ^0(X^\good)\ar[u]\ar[r]\ar[d]^{\widetilde{C}}&H^1(\widetilde{J[n]})\ar[u]\ar[d]\\
&L^\times\ar[r]& H^1\left( \dfrac{\mu_n^\fes}{\mu_n} \right). \\
}\]
Equality of the compositions from $\calZ^0(X^\good)$ to 
$H^1(\mu_n^\fes/\mu_n)$ is the desired result.
\end{proof}

\begin{corollary}
The explicit homomorphism $\widetilde{C}$ in~\eqref{E:fake explicit homomorphism} 
induces a homomorphism 
\[
	\widetilde{C}\colon \left(\frac{J(k)}{nJ(k)} \right)_\circ 
	\to \frac{L^\times}{L^{\times n} k^\times}.
\]
\end{corollary}

\begin{corollary}
\label{C:independence of explicit homomorphism}
The homomorphism $\widetilde{C}$, which a priori depends on $\beta$ and~$D$,
in fact is independent of~$D$ and depends only on the linear equivalence
class of~$\beta$.
\end{corollary}

\begin{proof}
The linear equivalence class of~$\beta$ is all that is needed to define~$C$.
\end{proof}

\begin{remark}
\label{R:moving}
By Corollary~\ref{C:independence of explicit homomorphism}, 
we may move $\beta$ and~$D$ within their linear equivalence
classes and change~$f$ accordingly
in order to make $\divisor(f)$ avoid any particular $0$-cycle.
Hence we may view $\widetilde{C}$ as being defined on all of~$\calZ^0(X)$.
\end{remark}

\subsection{Isogenies associated to descent setups}
\label{S:descent isogeny}

Given a true or fake descent setup, we obtain a homomorphism
$\alpha^\vee \colon E\to \widehat{J}[n]$, 
where $E\colonequals (\Z/n\Z)^\fes$ for a true descent setup 
and $E\colonequals (\Z/n\Z)_{\deg 0}^\fes$ for a fake descent setup.

In either case we obtain 
an isogeny $\widehat{\phi}\colon\widehat{J}\to \widehat{A}\colonequals \widehat{J}/\alpha^\vee(E)$,
and $\widehat{J}[\widehat{\phi}] = \alpha^\vee(E)$.
Let $A$ be the dual abelian variety of~$\widehat{A}$,
and let $\phi\colon A\to J$ be the dual of~$\widehat{\phi}$.
(For an early instance of relating the functions in a true descent setup
to an isogeny between different abelian varieties, for the purpose
of computing Selmer groups, see \cite{Schaefer1998}.)

\begin{proposition}
\label{P:factors through isogeny quotient}
The homomorphism $C \colon J(k) \to H^1(E^\vee)$
factors through the quotient $J(k)/\phi A(k)$.
\end{proposition}

\begin{proof}
By \cite{MumfordAV1970}*{\S III.15,~Theorem~1}, 
$A[\phi]^\vee\isom\widehat{J}[\widehat{\phi}]$.
Now $\alpha^\vee \colon E\to \widehat{J}[n]$ 
factors through $\alpha^\vee(E)=\widehat{J}[\widehat{\phi}]$,
and dualizing shows that $\alpha \colon J[n] \to E^\vee$
factors through~$A[\phi]$.
This explains the right triangle in the commutative diagram
\begin{equation}
\label{E:isogeny cohomology}
\begin{split}
\xymatrix{
\dfrac{J(k)}{nJ(k)} \ar[r] \ar[d] & H^1(J[n]) \ar[r] \ar[d] & H^1(E^\vee) \\
\dfrac{J(k)}{\phi A(k)} \ar[r] & H^1(A[\phi]), \ar[ru] \\
}
\end{split}
\end{equation}
while the square comes from functoriality of connecting homomorphisms.
Since the top row gives~$C$, the result follows.
\end{proof}

\begin{corollary}
The explicit homomorphism $\widetilde{C}$ defined on $\calZ^0(X)$
factors through not only $J(k)_\circ$ or $(J(k)/n J(k))_\circ$,
but also $(J(k)/\phi A(k))_\circ$.
\end{corollary}

\begin{proof}
Combine Proposition~\ref{P:factors through isogeny quotient}
with Proposition~\ref{P:true equivalence}
or Proposition~\ref{P:fake equivalence}.
\end{proof}

Let $R\colonequals \ker \alpha^\vee$, so we have an exact sequence
\begin{equation}
  \label{E:definition of R}
	0\longrightarrow 
	R\longrightarrow 
	E\stackrel{\alpha^\vee}{\longrightarrow} 
	\widehat{J}[\widehat{\phi}]\longrightarrow 
	0.
\end{equation}
Dualizing yields an exact sequence
\begin{equation}
  \label{E:alpha and q}  
 0\longrightarrow A[\phi]\stackrel{\;\alpha}{\longrightarrow} E^\vee \stackrel{q}{\longrightarrow} R^\vee\longrightarrow 0.
\end{equation}
Take cohomology:
for a true descent setup we obtain
\begin{equation}
\label{E:true main sequence}  
\begin{split}
\xymatrix{
&&&&\left(\dfrac{J(k)}{\phi A(k)}\right)_\circ \ar@{^(->}[d] \ar[r]^-{\widetilde{C}} \ar[dr]^-{C} & \dfrac{L^\times}{L^{\times n}} \ar@{=}[d] \\
0\ar[r]&A[\phi](k)\ar[r]^-{\alpha} & E^\vee(k)\ar[r]^q &R^\vee(k)\ar[r]&
H^1(A[\phi]) \ar[r]^-{\alpha} & H^1(E^\vee)\ar[r]^q \ar[r]^q &H^1(R^\vee)\\
}
\end{split}
\end{equation}
and for a fake descent setup we obtain
\begin{equation}
\label{E:fake main sequence}  
\begin{split}
\xymatrix{
&&&&\left(\dfrac{J(k)}{\phi A(k)}\right)_\circ\ar@{^(->}[d]\ar[r]^-{\widetilde{C}} \ar[dr]^-{C} &
\dfrac{L^\times}{L^{\times n}k^\times}. \ar@{^(->}[d] \\
0\ar[r]&A[\phi](k)\ar[r]^-{\alpha} &E^\vee(k)\ar[r]^q &R^\vee(k)\ar[r]&
H^1(A[\phi]) \ar[r]^-{\alpha} & H^1(E^\vee)\ar[r]^q &H^1(R^\vee)\\
}
\end{split}
\end{equation}
The commutativity follows from
Proposition~\ref{P:true equivalence}
or Proposition~\ref{P:fake equivalence}.

\subsection{Comparison of true and fake descent setups} 
\label{S:true fake comparison}

A fake descent setup $(n,\fes,\LL)$ in which
there exists $P \in \fes(k)$
gives rise to a true descent setup $(n,\fes',\LL')$ as follows:
Let $\fes' = \fes - \{P\}$ and define $\beta'_Q=\beta_Q-\beta_P$
for each $Q \in \fes'$.
The \'etale algebras $L,L'$ corresponding to~$\fes,\fes'$
satisfy $L \isom L' \times k$.
One can check that the fake diagram~\eqref{E:fake nz+y diagram}
maps to the true diagram~\eqref{E:true nz+y diagram}
as follows: 
the first two rows map via the identity,
and the third row maps via the homomorphisms induced
by $(c_Q)_{Q \in \fes} \mapsto (c_Q/c_P)_{Q \in \fes'}$.
Finally, the explicit true and fake homomorphisms are compatible in
the sense that 
\[
\xymatrix{
\left( \dfrac{J(k)}{\phi A(k)} \right)_\circ \ar[r]^-{\widetilde{C}_{\text{true}}} \ar[rd]_-{\widetilde{C}_{\text{fake}}} & \dfrac{(L')^\times}{(L')^{\times n}} \ar@{=}[d] \\
& \dfrac{L^\times}{L^{\times n} k^\times} \\
}
\]
commutes.

Moreover, every true descent setup can be obtained from some
fake descent setup in this way, at least if $k$ is infinite.
Thus fake descent is a generalization of true descent.

\subsection{Notation}

For our intended application, we are primarily interested in fake descent.
To avoid having to state things twice, 
we let $\widetilde{L^\times/L^{\times n} k^\times}$
denote $L^\times/L^{\times n}$ in the context of a true descent setup,
or $L^\times/L^{\times n} k^\times$ in the context of a fake descent setup.
If $v \in \Omega_k$, let $L_v \colonequals  L \tensor_k k_v$
and define $\widetilde{L_v^\times/L_v^{\times n} k_v^\times}$
similarly.
Decompose $L$ as~$\prod L_i$ where each~$L_i$ is a field.

\subsection{Restrictions on the image of $\widetilde{C}$}

Recall that $R \subseteq E \subseteq (\Z/n\Z)^{\fes}$.

\begin{lemma}
\label{L:kernel of norm}
Suppose that the image of the diagonal $\Z/n\Z \to (\Z/n\Z)^{\fes}$
is contained in~$R$.
Then the image of 
$\widetilde{C} \colon \calZ^0(X) \to \widetilde{L^\times/L^{\times n} k^\times}$
is contained in the kernel of the homomorphism
$N \colon \widetilde{L^\times/L^{\times n} k^\times} \to k^\times/k^{\times n}$
induced by the norm map $L^\times \to k^\times$.
\end{lemma}

\begin{proof}
Dualizing the complex 
\[
	\Z/n\Z \hookrightarrow E \to \widehat{J}[\widehat{\phi}]
\]
and applying $H^1$ yields a complex
\[
	H^1(A[\phi]) \to H^1(E^\vee) 
	\stackrel{\calN}\to \frac{k^\times}{k^{\times n}}.
\]
The restriction of~$\calN$ to the subgroup
$\widetilde{L^\times/L^{\times n} k^\times}$
equals~$N$,
because it is induced by dualizing $\Z/n\Z \to (\Z/n\Z)^\fes$
and taking cohomology.
The result now follows from the commutative squares 
in \eqref{E:true main sequence} and~\eqref{E:fake main sequence}.
\end{proof}

\section{Unramified classes}
\label{S:unramified classes}

\subsection{Local and global unramified classes}
\label{S:local and global unramified classes}

Let $k_v$ be a non-archimedean local field.
If $M$ is any $\calG_{k_v}$-module, define the 
\defi{subgroup of unramified classes}
\[ 
	H^1(k_v, M)_\unr \colonequals  \ker\left(H^1(k_v, M) \to H^1(k_{v,u}, M)\right)
	\isom H^1(\Gal(k_{v,u}/k_v),M(k_{v,u})),
\]
where the isomorphism follows from the inflation-restriction sequence.

Now let $k$ be a global field,
and let $\calS$ be a subset of~$\Omega_k$ containing the archimedean places.
If $M$ is any $\calG_k$-module, define
the \defi{subgroup of classes unramified outside $\calS$},
$H^1(k,M)_\calS$,
as the set of~$\xi \in H^1(k,M)$ whose restriction in~$H^1(k_v,M)$
lies in~$H^1(k_v,M)_\unr$ for all $v\notin \calS$.

\subsection{A Tamagawa number criterion}

Let $k_v$ be a non-archimedean local field.
Let $\phi \colon A \to J$ be a separable isogeny of
abelian varieties over~$k_v$.
The short exact sequence
\[
	0 \to A[\phi] \to A \stackrel{\phi}\to J \to 0
\]
gives rise to a connecting homomorphism 
$\gamma_v \colon J(k_v) \to H^1(k_v,A[\phi])$.
Let $\calJ$ be the N\'eron model of~$J$ over~$\calO_v$. 
Let $\calJ^0$ be the connected component of the identity in~$\calJ$.
Let $\Phi$ be the group of connected components
of the special fiber $\calJ \times \F_v$,
so $\Phi$ is a finite \'etale commutative $\F_v$-group scheme.
The number $c_v(J) \colonequals  \#\Phi(\F_v)$ 
is called the \defi{Tamagawa number of~$J$ at~$v$}.
We use analogous notation for objects related to~$A$.

\begin{lemma} \label{L:Tamagawa-unramified}
Let $\phi \colon A \to J$ be as above.
Let $n$ be a positive integer such that $n A[\phi]=0$
(one possibility is $n=\deg \phi$).
If the residue characteristic $\Char \F_v$ does not divide~$n$, 
and if $c_v(J)$ and $c_v(A)$ are coprime to~$n$, then
$\gamma_v(J(k_v)) = H^1(k_v, A[\phi])_\unr$.
\end{lemma}

\begin{proof}
This is a straightforward generalization of the proof
of~\cite{Schaefer-Stoll2004}*{Lemma~3.1~and~Proposition~3.2},
which uses results from~\cite{Schaefer1996}*{\S3}.
\end{proof}

\subsection{Unramified elements in the target of the descent map}
\label{S:unramified in target of descent}

Now suppose that $k$ is a global field 
and that we have a true or fake descent setup.
{}From \eqref{E:true main sequence} or~\eqref{E:fake main sequence}  
applied to~$k_v$ we obtain
$\widetilde{L_v^\times/L_v^{\times n} k_v^\times} \injects H^1(k_v,E^\vee)$.
Call an element of $\widetilde{L_v^\times/L_v^{\times n} k^\times}$
\defi{unramified}
if its image in~$H^1(k_v,E^\vee)$ is unramified.

Let $\calS$ be a subset of~$\Omega_k$ containing the archimedean places.
Say that an element of $\widetilde{L^\times/L^{\times n} k^\times}$ 
is \defi{unramified outside $\calS$}
if its image in $\widetilde{L_v^\times/L_v^{\times n} k_v^\times}$
is unramified for each $v \notin \calS$.
Let $\widetilde{(L^\times/L^{\times n} k^\times)}_\calS$ 
be the subgroup of such elements.
Proposition~\ref{P:unramified classes} will provide an explicit description
of $\widetilde{(L^\times/L^{\times n} k^\times)}_\calS$.

Given $\ell \in L = \prod L_i$, let $\ell_i$ be its image in~$L_i$.
Let $L(n,\calS)$ 
be the subgroup of $L^\times/L^{\times n}$ 
consisting of elements represented by~$\ell \in L^\times$
such that the prime-to-$\calS$ part of 
the fractional ideal~$(\ell_i)$ of~$L_i$
is an $n^{\tH}$~power for all~$i$.
In the true case, let $\widetilde{L(n,\calS)}=L(n,\calS)$;
in the fake case, let $\widetilde{L(n,\calS)}$
be the subgroup of $L^\times/L^\times k^\times$
consisting of elements represented by~$\ell \in L^\times$
for which there exists a fractional ideal~$\mathfrak{a}$ of~$k$
such that the prime-to-$\calS$ part of~$\mathfrak{a} \cdot (\ell_i)$
is an $n^{\tH}$~power for all~$i$.

\begin{proposition}
\label{P:unramified classes}
Suppose that $\calS \subseteq \Omega_k$ contains all archimedean places
and all places of residue characteristic dividing~$n$.
Then $\widetilde{(L^\times/L^{\times n} k^\times)}_\calS = \widetilde{L(n,\calS)}$.
\end{proposition}

\begin{proof}
The statement and proof are the same 
as for~\cite{Poonen-Schaefer1997}*{Proposition~12.5}.
\end{proof}

\begin{proposition}
\label{P:L(n,S)}
The group $\widetilde{L(n,\calS)}$ is finite and computable.
\end{proposition}

\begin{proof}
Let $\calO_\calS$ and $\calO_{L,\calS}$ be 
the rings of $\calS$-integers in $k$ and~$L$, respectively;
then there are exact sequences
\begin{equation}
\label{E:L(n,S)}
	0 \to \frac{\calO_{L,\calS}^\times}{\calO_{L,\calS}^{\times n}}
	\to L(n,\calS)
	\to \Cl(\calO_{L,\calS})[n]
	\to 0
\end{equation}
and (in the fake case)
\begin{equation}
\label{E:widetilde{L(n,S)}}
	k(n,\calS) \to L(n,\calS) \to \widetilde{L(n,\calS)} 
	\to \frac{\Cl(\calO_\calS)}{n \Cl(\calO_\calS)},
\end{equation}
as in~\cite{Poonen-Schaefer1997}*{Propositions~12.6 and 12.8},
where these are constructed for prime~$n$.
So the finiteness follows from the Dirichlet $\calS$-unit theorem
and finiteness of the class groups,
and the computability follows too since these are effective.
\end{proof}

\begin{remark}
The computation of the unit groups and class groups
typically dominates the running time in performing explicit descent.
\end{remark}

\begin{proposition}
\label{P:finiteness of unramified classes}
Let $M$ be a finite $\calG$-module with $\Char k \nmid \#M$.
Let $\calS$ be a finite set of places of~$k$ containing the
archimedean places.
Then $H^1(k,M)_\calS$ is finite.
\end{proposition}

\begin{proof}
Using the inflation-restriction sequence lets us enlarge $k$
to assume that $\calG$ acts trivially on~$M$ and on the roots of unity
of order dividing~$\#M$.
Decomposing $M$ as a product of cyclic groups 
lets us reduce to the case $M=\mu_n$ where $\Char k \nmid n$.
Now $H^1(k,\mu_n)_\calS \isom (k^\times/k^{\times n})_\calS$,
which, as in Proposition~\ref{P:L(n,S)}, is finite.
\end{proof}

\section{Finite Galois modules and $\Sha^1$}\label{S:sha1}
\label{S:Sha1}

Let $k$ be a global field, and let $M$ be a finite $\calG_k$-module. 
If $v$ is a place of~$k$, write $H^1(k,M)\to H^1(k_v,M)$
for the restriction to a decomposition group associated to~$v$;
if $\xi \in H^1(k,M)$, let $\xi_v \in H^1(k_v,M)$ be its restriction.
As usual, define
\[
	\Sha^1(k,M) \colonequals  \ker \left( H^1(k,M)\to \prod_{v \in \Omega_k} H^1(k_v,M) \right).
\]

\begin{lemma}[cf.~\cite{MilneADT2006}*{Example~I.4.11(i)}]\label{lemma:splitsha}
If $\calG_k$ acts trivially on~$M$, then $\Sha^1(k,M)=0$.
\end{lemma}

\begin{proof}
The hypothesis implies that $H^1(k,M)=\Hom(\calG_k,M)$, 
where $\Hom$ denotes the group of continuous homomorphisms.
Chebotarev's density theorem implies that 
the union of the decomposition groups~$\calG_{k_v}$ is dense in~$\calG_k$,
so the map
\[
	\Hom(\calG_k,M) \to \prod_v \Hom(\calG_{k_v},M)
\]
is injective.  This implies the result.
\end{proof}

\begin{lemma}
\label{L:Sha of mu_p}
If $p$ is a prime not equal to~$\Char k$, and $\fes$ is any $\calG$-set,
then $\Sha^1(k,\mu_p^{\fes})=0$.
\end{lemma}

\begin{proof}
By Shapiro's lemma, one reduces to proving that $\Sha^1(k,\mu_p)=0$.
The latter is the well-known fact that an element of~$k^\times$ that is
a $p^{\tH}$~power in every~$k_v$ is a $p^{\tH}$~power in~$k$.
\end{proof}

The following result allows us to compute $\Sha^1(k,M)$ 
using finite group cohomology.

\begin{proposition}\label{P:sha1 decomp}
Let $k$ be a global field, 
let $M$ be a finite $\calG_k$-module, 
and let $K$ be a Galois splitting field of~$M$.
Let $G=\Gal(K/k)$. 
For $v \in \Omega_k$, let $D_v\subset G$ denote a decomposition group.
Then
\[
	\Sha^1(k,M) \isom \ker\Big(H^1(G,M)\to \prod_v H^1(D_v,M)\Big).
\]
In particular,
\[
	\Sha^1(k,M) \subseteq
	\Intersection_{\textup{cyclic $H \le G$}} \ker(H^1(G,M)\to H^1(H,M)).
\]
\end{proposition}
\begin{proof}
For each nontrivial place $v$ of $k_s$,
let $k_v$ and $K_v$ denote the corresponding completions of $k$ and $K$,
and let $D_v$ be the corresponding decomposition group inside $G$.
The products in the diagram below will be taken over all such $v$,
instead of just the underlying places of $k$, 
but this does not affect the definitions of the Shafarevich--Tate groups,
since in general, 
the kernel of a restriction map $H^1(G,M) \to H^1(H,M)$ is unchanged 
if $H \le G$ is replaced by a conjugate subgroup.
Inflation-restriction with respect to
\[
	1\to \calG_K \to \calG_k \to G\to 1
\]
and its local analogues yield the exactness of the last two rows 
in the commutative diagram
\[\xymatrix{
0\ar[r]&\Sha^1(G,M)\ar[r]\ar[d]&\Sha^1(k,M)\ar[r]\ar[d]& \Sha^1(K,M) \ar[d]\\
0\ar[r]&H^1(G,M)\ar[r]\ar[d]&H^1(k,M)\ar[r]\ar[d]&H^1(K,M)\ar[d]\\
0\ar[r]&\prod_v H^1(D_v,M)\ar[r]&\prod_v H^1(k_v,M)\ar[r]&\prod_v H^1(K_v,M),
}\]
and the groups in the top row are defined as the kernels
of the vertical maps connecting the second and third rows.
Lemma~\ref{lemma:splitsha} yields $\Sha^1(K,M)=0$, 
so $\Sha^1(k,M) \isom \Sha^1(G,M)$,
which is the first statement of the lemma.
By the Chebotarev density theorem, 
every cyclic subgroup of $G$ occurs as a $D_v$;
this implies the second statement.
\end{proof}

\section{Definition of Selmer groups and sets}
\label{S:Selmer groups and sets}

\subsection{Selmer groups associated to isogenies}

\begin{definition}
Let $\phi\colon A\to J$ be an isogeny
of abelian varieties over a global field~$k$
such that $\Char k \nmid \deg \phi$.
For each~$v$, we obtain a connecting homomorphism
$\gamma_v \colon J(k_v) \to H^1(k_v,A[\phi])$.
Define
\[
	\Sel{\phi}(J)\colonequals  \left\{\delta \in H^1(k,A[\phi]) : \delta_v \in \im \gamma_v \textup{ for all $v$} \right\}.
\]
\end{definition}

\begin{proposition}
\label{P:Sel is unramified}
Let $\calS$ be a finite set of places of~$k$ containing
\begin{itemize}
 \item all archimedean places,
 \item all places where the residue characteristic divides~$n$, and
 \item all places where the Tamagawa number of $J$ or~$A$
       is not coprime to~$n$
	(these form a subset of the places of bad reduction).
\end{itemize}
Then 
\begin{enumerate}[\upshape (a)]
\item \label{I:description of Selmer} 
$\Sel{\phi}(J) = \left\{ \delta \in H^1(k,A[\phi])_\calS : \delta_v \in \im \gamma_v \textup{ for all $v \in \calS$} \right\}$.
\item \label{I:finiteness of H^1_S}
The group $H^1(k,A[\phi])_\calS$ is finite.
\item \label{I:Sel is finite}
The group $\Sel{\phi}(J)$ is finite.
\end{enumerate}
\end{proposition}

\begin{proof}\hfill
  \begin{enumerate}[\upshape (a)]
  \item Apply Lemma~\ref{L:Tamagawa-unramified} to~$k_v$
   for each $v \notin \calS$.
 \item Finiteness of $H^1(k,A[\phi])_\calS$ is a special case
of Proposition~\ref{P:finiteness of unramified classes}.
\item This follows from \eqref{I:description of Selmer} 
	and~\eqref{I:finiteness of H^1_S}.\qedhere
  \end{enumerate}
\end{proof}

The image of $J(k)/\phi A(k) \injects H^1(k,A[\phi])$
is contained in~$\Sel{\phi}(J)$,
by definition of the latter.

\subsection{Selmer groups associated to descent setups}

We define true and fake Selmer groups
by replacing $H^1(k,A[\phi])$ by its explicit analogue
$\widetilde{L^\times/L^{\times n} k^\times}$,
and by replacing $\gamma_v \colon J(k_v) \to H^1(k_v,A[\phi])$ 
by~$C_v=\widetilde{C}_v$, which is a homomorphism
$J(k_v)^\circ \to \widetilde{L_v^\times/L_v^{\times n} k_v^\times}$.

\begin{definition}
\label{D:true and fake Sel}
Given a true descent setup $(n,\fes,\beta)$,
define the \defi{true Selmer group}
\[
	\Seltrue{\alpha}(J) \colonequals  
        \left\{\delta\in \frac{L^\times}{L^{\times n}} : 
	\delta_v \in \im(C_v)\textup{ for all places $v$ of $k$}\right\}.
\]
Given a fake descent setup $(n,\fes,\beta)$,
define the \defi{fake Selmer group}
\[
	\Selfake{\alpha}(J) \colonequals  
	\left\{\delta\in \frac{L^\times}{L^{\times n}k^\times}: 
	\delta_v \in \im(C_v)\textup{ for all places $v$ of $k$}\right\}.
\]
To avoid having to state results twice, 
we use $\Seltf{\alpha}(J)$ to denote either $\Seltrue{\alpha}(J)$
or~$\Selfake{\alpha}(J)$, depending on whether we are considering a true or a
fake descent setup. 
\end{definition}

\subsection{True and fake Selmer sets for $X$}
\label{S:Sel for X}

Suppose that we have a true or fake descent setup for~$X$,
and that we fix $f \in k(X\times \fes)^\times$ 
as in Definition~\ref{D:true setup} or~\ref{D:fake setup},
and define $X^\good$ accordingly.
Evaluation on closed points as in Section~\ref{S:true explicit definition}
or \ref{S:fake explicit definition}
defines a homomorphism
$\widetilde{C}_f \colon \calZ(X^{\good}) 
\to \widetilde{L^\times/L^{\times n} k^\times}$.
Since $\widetilde{C}_f$ and $\widetilde{C}$ agree on $\calZ^0(X^{\good})$,
they have a common extension to the sum of their domains,
$\calZ(X^{\good}) + \calZ^0(X) = \calZ(X)$.
Then we may restrict to~$X(k)$ to obtain a map of sets
\[
	C_f \colon X(k) 
	\to \widetilde{L^\times/L^{\times n} k^\times}.
\]
Similarly, for each $v \in \Omega_k$, we obtain
\[
	C_{f,v} \colon X(k_v) 
	\to \widetilde{L_v^\times/L_v^{\times n} k_v^\times}.
\]
As the notation suggests, these maps depend on the choice of~$f$.

\begin{definition}
\label{D:Sel for X}
Given a true or fake descent setup for $X$ and~$f$,
define
\[
	\Seltf{f}(X) \colonequals  
	\left\{ \delta\in \widetilde{\frac{L^\times}{L^{\times n}k^\times}} : 
		\delta_v \in \im(C_{f,v}) \textup{ for all places $v$ of $k$}
	\right\}.
\]
\end{definition}

\begin{lemma}
\label{L:coset of kernel of norm}
Suppose that the image of the diagonal $\Z/n\Z \to (\Z/n\Z)^{\fes}$
is contained in~$R$.
Then there exists $c \in k^\times$ such that 
$N(\widetilde{C}_f(z)) = c^{\deg z}$ in $k^\times/k^{\times n}$
for all $z \in \calZ(X)$.
\end{lemma}

\begin{proof}
In the true case,
the hypothesis on~$R$ implies that
$\sum_{P \in \fes} \beta_P = \divisor(r)$ for some $r \in k(X)^\times$.
Then $\divisor(N_{L/k}(f))=n \sum_{P \in \fes} \beta_P = \divisor(r^n)$,
so $N_{L/k}(f) = cr^n$ for some $c \in k^\times$.
Evaluating on any $z \in \calZ(X^{\good})$ yields
the result for such~$z$.
On the other hand, Lemma~\ref{L:kernel of norm} yields the result
for $z \in \calZ^0(X)$.
Together, these prove the result 
for any $z \in \calZ(X^{\good}) + \calZ^0(X) = \calZ(X)$.

In the fake case, 
the hypothesis on~$R$ implies that
$\#\fes = n m$ for some $m \in \Z_{>0}$ 
and that $\left( \sum_{P \in \fes} \beta_P \right) - mD = \divisor(r)$
for some $r \in k(X)^\times$.
Then $\divisor(N_{L/k}(f))=n \left(\sum_{P \in \fes} \beta_P\right) - nmD 
= \divisor(r^n)$,
so $N_{L/k}(f) = cr^n$ for some $c \in k^\times$.
The rest of the proof is as in the true case.
\end{proof}


\section{Relations between various Selmer groups}\label{S:relations between Selmer groups}

In Sections~\ref{S:local considerations}--\ref{S:selmer group comparison},
we will assume the following:
\begin{hypothesis}
\label{H:circ}
The maps $J(k)_\circ \to J(k)/\phi A(k)$
and $J(k_v)_\circ \to J(k_v)/\phi A(k_v)$ are surjective 
for all places~$v$ of~$k$.
\end{hypothesis}

There are some common situations 
in which Hypothesis~\ref{H:circ} is justified:
\begin{lemma}\label{L:circ}
Suppose that $X$ is a nice curve such that
\begin{enumerate}[\upshape (i)]
\item\label{I:circ k} 
$X$ has a $k$-point 
(or more generally $H^0(k,\Pic X_s) \stackrel{\deg}\to \Z$ 
is surjective), or
\item\label{I:circ k_v} 
$k$ is a global field and $X$ has a $k_v$-point 
(or more generally $H^0(k_v,\Pic X_{k_{v,s}}) \stackrel{\deg}\to \Z$ 
is surjective)
for every $v \in \Omega_k$, or
\item\label{I:circ hyperelliptic}
$k$ is a global field of characteristic not~$2$
and $X$ is the smooth projective model of a hyperelliptic curve $y^2=f(x)$
of {\em even genus}.
\end{enumerate}
Then $J(k)_\circ = J(k)$.
If moreover $k$ is a global field,
then $J(k_v)_\circ = J(k_v)$ for every~$v$,
so Hypothesis~\ref{H:circ} holds.
\end{lemma}

\begin{proof}
For \eqref{I:circ k} and~\eqref{I:circ k_v},
see~\cite{Poonen-Schaefer1997}*{Propositions~3.2~and~3.3}.
For \eqref{I:circ hyperelliptic} (generalized to superelliptic
curves), see~\cite{Poonen-Schaefer1997}*{end~of~\S4}.
\end{proof}

Hypothesis~\ref{H:circ} will let us define
a map from the usual $\phi$-Selmer group into the Selmer group associated
to our descent setup; 
without this map, not much could be said about how the groups relate.
In Section~\ref{S:local considerations} 
we study a local group~$W_v$
that in a sense measures the difference between
the local descent map~$\gamma_v$ and our approximation~$C_v$.
In Section~\ref{S:selmer group comparison}
we define a global group~$\calK$
with a homomorphism 
$\kappa \colon \calK \to \prod_v W_v$ whose kernel and cokernel control the
difference between the actual $\phi$-Selmer group 
and our explicit Selmer group~$\Seltf{\alpha}(J)$.


\subsection{Local considerations}\label{S:local considerations}

We fix a place~$v$ of~$k$. 
Diagram \eqref{E:true main sequence} or~\eqref{E:fake main sequence}  
applied to~$k_v$ yields
\begin{equation}
  \label{E:local descent diagram}
\begin{split}
\xymatrix{
       & & \dfrac{J(k_v)}{\phi A(k_v)} \ar@{^{(}->}[d]_-{\gamma_v} \ar[dr]^{C_v} \\
              E^\vee(k_v) \ar[r]^-{q} & R^\vee(k_v) \ar[r]
                                 & H^1(k_v, A[\phi]) \ar[r]^-{\alpha_v}
                                 & H^1(k_v, E^\vee) \ar[r]
                                 & H^1(k_v, R^\vee) \\
            }
\end{split}
\end{equation}
in which the main row is exact.
The group $\ker \alpha_v$ is hence isomorphic
to $R^\vee(k_v)/q E^\vee(k_v)$,
which is finite, and computable in terms of the actions
of the decomposition group of~$v$ on $E^\vee$ and~$R^\vee$.

Since $C_v = \alpha_v \gamma_v$ in \eqref{E:local descent diagram},
we have an exact sequence
\begin{equation}
  \label{E:kernel-cokernel}
	\ker C_v \to \ker \alpha_v \to \coker \gamma_v \to \coker C_v.
\end{equation}

\begin{definition}
\label{D:W_v}
Let $W_v$ be either of the following naturally isomorphic groups 
obtained from~\eqref{E:kernel-cokernel}:
\begin{enumerate}[\upshape (i)]
\item \label{I:W_v is cokernel}
the cokernel of the first map $\ker C_v \to \ker \alpha_v$ 
\item \label{I:W_v is image} 
the image of the second map $\ker \alpha_v \to \coker \gamma_v$, or
\item \label{I:W_v is kernel} 
the kernel of the third map $\coker \gamma_v \to \coker C_v$.
\end{enumerate}
\end{definition}

\begin{lemma}
\label{L:computing W_v}
For any $v \in \Omega_k$, 
the group~$W_v$ is finite, and
$\#W_v = \dfrac{\# \coker q \cdot \# \im C_v}{\# \im \gamma_v}$.
\end{lemma}

\begin{proof}
The group $\ker \alpha_v \isom \coker q$ is finite.
Separability of~$\phi$ implies that $\phi A(k_v)$ is an open subgroup
of the compact group~$J(k_v)$,
so the group $\im \gamma_v \isom J(k_v)/\phi A(k_v)$ is finite too.
By \eqref{E:local descent diagram} 
and Definition~\ref{D:W_v}\eqref{I:W_v is image}, 
we have exact sequences
\begin{align*}
	0 \To \im \gamma_v \intersect \ker \alpha_v 
		&\To \ker \alpha_v \To W_v \To 0 \\
	0 \To \im \gamma_v \intersect \ker \alpha_v 
		&\To \im \gamma_v \To \im C_v \To 0,
\end{align*}
which let us compute $\# W_v$.
\end{proof}

The following lemma will let us understand $W_v$ for most~$v$.

\begin{lemma} \label{L:Kvtrivial}
  Let $v$ be a non-archimedean place of~$k$ such that
  \begin{enumerate}[\upshape (i)]
    \item\label{I:residue char} 
	the residue characteristic of~$v$ does not divide~$n$, and
    \item\label{I:Tamagawa}
	the Tamagawa numbers $c_v(J)$ and $c_v(A)$ are coprime to~$n$.
  \end{enumerate}
Then 
\begin{enumerate}[\upshape (a)]
\item\label{I:im C_v}
$\im(C_v) \subseteq H^1(k_v,E^\vee)_\unr$, and
\item\label{I:W_v formula} 
$W_v = \im \left( \dfrac{R^\vee(k_v)}{q E^\vee(k_v)} \to \dfrac{R^\vee(k_{v,u})}{q E^\vee(k_{v,u})} \right)$.
\end{enumerate}
\end{lemma}

\begin{proof}
The commutative diagram
\begin{equation}
\label{E:R/qE}
\begin{split}
\xymatrix{
 & & \dfrac{J(k_v)}{\phi A(k_v)}\ar[d]_-{\gamma_v}\ar[dr]^-{C_v}\\
0 \ar[r] & \dfrac{R^\vee(k_v)}{q E^\vee(k_v)} \ar[r] \ar[d] & H^1(k_v,A[\phi])\ar[r]^{\alpha_v}\ar[d] & H^1(k_v,E^\vee)\ar[d]\\
0\ar[r]& \dfrac{R^\vee(k_{v,u})}{q E^\vee(k_{v,u})}\ar[r] & H^1(k_{v,u},A[\phi])\ar[r] & H^1(k_{v,u},E^\vee)\\
}
\end{split}
\end{equation}
has exact rows (cf.~\eqref{E:local descent diagram}).
The central column is exact too,
by \eqref{I:residue char}, \eqref{I:Tamagawa}, 
and Lemma~\ref{L:Tamagawa-unramified}.
\begin{enumerate}[\upshape (a)]
\item 
In~\eqref{E:R/qE}, $J(k_v)/\phi A(k_v)$ maps to~$0$ in $H^1(k_{v,u},E^\vee)$,
so $\im(C_v)\subset H^1(k_v,E^\vee)_\unr$.
\item 
By Definition~\ref{D:W_v}\eqref{I:W_v is image},
\begin{align*}
	W_v &= \im\left(\ker \alpha_v \to \coker \gamma_v\right) \\
	&= \im \left( \dfrac{R^\vee(k_v)}{q E^\vee(k_v)} \to H^1(k_{v,u},A[\phi]) \right) \\
	&= \im \left( \dfrac{R^\vee(k_v)}{q E^\vee(k_v)} \to \dfrac{R^\vee(k_{v,u})}{q E^\vee(k_{v,u})} \right).\qedhere
\end{align*}
\end{enumerate}
\end{proof}

\begin{corollary}
\label{C:W_v=0 usually}
For all but finitely many~$v$, we have $W_v=0$.
\end{corollary}

\begin{proof}
For all but finitely many~$v$,
the residue characteristic of~$v$ does not divide~$n$,
and $J$ and~$A$ have good reduction at~$v$, so $c_v(J)=c_v(A)=1$.
If we also discard the finitely many~$v$ at which $E^\vee$ is ramified,
then for the remaining~$v$ the surjection $E^\vee \to R^\vee$
of finite \'etale group schemes 
induces a surjection 
$E^\vee(k_{v,u}) = E^\vee(k_{v,s}) 
\stackrel{q}\to R^\vee(k_{v,s}) = R^\vee(k_{v,u})$,
so Lemma~\ref{L:Kvtrivial}\eqref{I:W_v formula}
implies that $W_v=0$.
\end{proof}

\begin{corollary}\label{C:ProdKv-finite}
The product $\prod_{v \in \Omega_k} W_v$ is a finite group.
\end{corollary}

\begin{proof}
Combine Lemma~\ref{L:computing W_v} and Corollary~\ref{C:W_v=0 usually}.
\end{proof}

\begin{lemma} \label{L:KvInfinite}
If $v$ is complex, or if $v$ is real and $n$ is odd, then $W_v = 0$.
\end{lemma}

\begin{proof}
If $v$ is complex, then $\Gal(k_{v,s}/k_v)=\Gal(\C/\C)=\{1\}$.
If $v$ is real then $\Gal(k_{v,s}/k_v)=\Gal(\C/\R)=\Z/2\Z$. 
Our assumptions assure that in either case, 
$\Gal(k_{v,s}/k_v)$ is annihilated by a unit modulo~$n$,
so $H^1(k_{v},A[\phi])=H^1(k_{v},E^\vee)=0$. 
It follows that $W_v=0$.
\end{proof}


\subsection{Finite description of true and fake Selmer groups}
\label{S:finite description}

Let $\calS$ be as in Proposition~\ref{P:Sel is unramified}.
Given another finite set of places~$\calT$, define
\[
	S_\calT \colonequals  \left\{\delta\in \widetilde{L(n,\calS)} : 
	\delta\in \im(C_v)\textup{ for all }v\in\calT \right\}.
\]

\begin{theorem}
\label{T:finite description of Selmer}
Let $\calS$ and $\calT$ be as above.
\begin{enumerate}[\upshape (a)]
\item \label{I:upper bound for fake Sel}
We have $\Seltf{\alpha}(J) \subseteq S_\calT$.
\item \label{I:good T}
For each $\delta \in \widetilde{L(n,\calS)}$,
fix a finite Galois extension $k_\delta/k$ splitting~$R^\vee$
such that $\delta$ maps to~$0$ in~$H^1(k_\delta,R^\vee)$.
If $\calT$ contains all places in~$\calS$, all places at which $E^\vee$ ramifies,
and enough places so that for each $\delta \in \widetilde{L(n,\calS)}$
the Frobenius elements~$\Frob_v$ for~$v \in \calT$ unramified in~$k_\delta/k$
cover all conjugacy classes in~$\Gal(k_\delta/k)$,
then $\Seltf{\alpha}(J) = S_\calT$.
\end{enumerate}
\end{theorem}

\begin{proof}\hfill
  \begin{enumerate}[\upshape (a)]
  \item For $v \notin \calS$, 
Lemma~\ref{L:Kvtrivial}\eqref{I:im C_v} proves that $\Seltf{\alpha}(J)$
consists of elements unramified at~$v$.
Proposition~\ref{P:unramified classes} now shows that
$\Seltf{\alpha}(J) \subseteq \widetilde{L(n,\calS)}$.
The condition $\delta \in \im C_v$ in the definition of~$S_\calT$
is also in the definition of~$\Seltf{\alpha}(J)$.
\item 
Suppose that $\delta \in S_\calT$.
Given $v \notin \calT$, we must show that $\delta \in \im C_v$.
Choose $w \in \calT$ unramified in~$k_\delta/k$ such that the conjugacy classes
$\Frob_w$ and~$\Frob_v$ in~$\Gal(k_\delta/k)$ match.
By definition of~$S_\calT$, we have $\delta_w \in \im C_w$.
Then~\eqref{E:local descent diagram} for~$k_w$
shows that $\delta$ maps to~$0$ in~$H^1(k_w,R^\vee)$.
In other words, $\delta \in H^1(\Gal(k_\delta/k),R^\vee)$ restricts to~$0$
at~$w$.
But the decomposition groups of $v$ and~$w$ in $\Gal(k_\delta/k)$
are conjugate,
so $\delta$ maps to~$0$ in~$H^1(k_v,R^\vee)$ too.
Thus the element $\delta_v \in H^1(k_v,E^\vee)$
is the image of some $\xi_v \in H^1(k_v,A[\phi])$.
Since $v \notin \calS$, the element~$\delta_v$ is unramified;
in particular, $\xi_v$ maps to~$0$ in~$H^1(k_{v,u},E^\vee)$.
By hypothesis, $E^\vee$ is unramified at~$v$,
so $E^\vee(k_{v,u}) \stackrel{q}\to R^\vee(k_{v,u})$ is surjective,
so the bottom row of~\eqref{E:R/qE}
shows that $H^1(k_{v,u},A[\phi]) \to H^1(k_{v,u},E^\vee)$ is injective.
Thus $\xi_v$ maps to~$0$ already in~$H^1(k_{v,u},A[\phi])$;
i.e., $\xi_v \in H^1(k_v,A[\phi])_\unr$.
By Lemma~\ref{L:Tamagawa-unramified},
$\xi_v$ is in the image of~$J(k_v)$ under~$\gamma_v$.
Thus $\delta_v$ is in the image of~$J(k_v)$ under~$C_v$.\qedhere
  \end{enumerate}
\end{proof}

\begin{remark}
\label{R:DSS}
The idea to use an enlarged set $\calT$ including places whose
Frobenius elements cover the conjugacy classes was first used
in \cite{Djabri-Schaefer-Smart2000}*{Corollary~12}.
\end{remark}

\begin{remark}
\label{R:T is computable}
Here we show how to compute a finite set~$\calT$ as in
Theorem~\ref{T:finite description of Selmer}\eqref{I:good T}.
A finite splitting field $k_\fes$ of~$\fes$ will split~$R^\vee$.
Given $\delta$, represented by $\ell = (\ell_i) \in L^\times$, say,
adjoining all $n^{\tH}$~roots of the~$\ell_i$ to~$k_\fes$
yields a candidate for~$k_\delta$.
The Chebotarev density theorem guarantees that we can find enough~$v$
unramified in~$k_\delta/k$ to cover the conjugacy classes.
\end{remark}

\begin{remark}
\label{R:smaller T}
In practice, we may choose a smaller~$\calT$,
one that does not satisfy the hypotheses in 
Theorem~\ref{T:finite description of Selmer}\eqref{I:good T}.
Then we have only inclusions
\[
	C(J(k)) \subseteq \Seltf{\alpha}(J) \subseteq S_\calT.
\]
But if we find enough points in~$J(k)$ to show that $C(J(k))=S_\calT$,
then we obtain $\Seltf{\alpha}(J) = S_\calT$ nevertheless.
Often $\calT=\calS$ suffices.
\end{remark}


\subsection{Comparison of the Selmer group associated to an isogeny
            with the Selmer group associated to a descent setup}
\label{S:selmer group comparison}

In the following, we will study the relation between the Selmer group
associated to a true or fake descent setup (which is an object we can
hope to compute) and the Selmer group associated to the isogeny~$\phi$
determined by the descent setup (which is the object we would like to
compute).
Theorem~\ref{T:Selseq} will show that $\alpha$~induces a
homomorphism from the latter to the former.

\begin{lemma}
\label{L:Seltf}
We have an exact sequence 
\[
	0 \to \Seltf{\alpha}(J) \to H^1(k,E^\vee) \to \prod_v \frac{H^1(k_v, E^\vee)}{\im C_v}.
\]
\end{lemma}

\begin{proof}
In the true case, $L^\times/L^{\times n} = H^1(k,E^\vee)$, so this is
just the definition of~$\Seltrue{\alpha}(J)$.
In the fake case, \eqref{E:Br} for~$k$ and its completions yields a
commutative diagram 
\[
\xymatrix{
0 \ar[r] & \dfrac{L^\times}{L^{\times n} k^\times} \ar[r] \ar[d] & H^1(k,E^\vee) \ar[r] \ar[d] & \Br k \ar[d] \\
0 \ar[r] & \displaystyle \prod_v \dfrac{L_v^\times}{L_v^{\times n} k_v^\times} \ar[r] & \rule[-1em]{0pt}{2.5em}\displaystyle \prod_v H^1(k_v,E^\vee) \ar[r] &\rule[-1em]{0pt}{2.5em}\displaystyle \prod_v \Br k_v \\
}
\]
with exact rows.
An element $\delta \in H^1(k,E^\vee)$ mapping into 
$\im C_v \subseteq L_v^\times/L_v^{\times n} k_v^\times$ for all~$v$
maps to~$0$ in~$\Br k_v$ for all~$v$, so by the local-global property
of the Brauer group, it also maps to~$0$ in~$\Br k$,
so $\delta \in L^{\times}/L^{\times n}k^\times$.
Thus
\[
	\ker\left( H^1(k,E^\vee) \to \prod_v \frac{H^1(k_v, E^\vee)}{\im C_v} \right)
	= \ker\left( \frac{L^\times}{L^{\times n} k^\times} \to \prod_v \frac{H^1(k_v, E^\vee)}{\im C_v} \right)
	= \Selfake{\alpha}(J).\qedhere
\]
\end{proof}

Let $\calK$ be the kernel of the global map 
$\alpha \colon H^1(k, A[\phi]) \to H^1(k, E^\vee)$,
which by~\eqref{E:fake main sequence} 
equals the (computable) cokernel of $q \colon E^\vee(k) \to R^\vee(k)$.
Let $\kappa$ be the composition
\begin{equation}
\label{E:kappa}
	\calK = \ker \alpha \to \prod_v \ker \alpha_v \surjects
	\prod_v \im(\ker \alpha_v \to \coker \gamma_v) = \prod_v W_v.
\end{equation}
The following proposition gives a homomorphism 
$\Sel{\phi}(J) \to \Seltf{\alpha}(J)$
and provides information on its failure to be an isomorphism.

\begin{theorem} \label{T:Selseq}
  We have an exact sequence
  \[ 0 \To \ker \kappa \To \Sel{\phi}(J)
     \stackrel{\alpha}{\To} \Seltf{\alpha}(J) \intersect \alpha\bigl(H^1(k, A[\phi])\bigr)
     \To \coker \kappa.
  \]
\end{theorem}

\begin{proof}
By definition of~$\kappa$, the diagram
\begin{equation}
\label{E:big snake}
\begin{split}
 \xymatrix{    
                0 \ar[r] & \calK \ar[r] \ar[d]_-{\kappa}
                         & H^1(k, A[\phi]) \ar[r]^-{\alpha} \ar[d]
                         & \alpha\bigl(H^1(k, A[\phi])\bigr) \ar[r] \ar[d]
                         & 0 \\
                0 \ar[r] & \rule[-1em]{0pt}{2.5em}\displaystyle\prod_v W_v \ar[r]
                         & \displaystyle\prod_v \dfrac{H^1(k_v, A[\phi])}{\im \gamma_v} \ar[r]
                         & \displaystyle\prod_v \dfrac{H^1(k_v, E^\vee)}{\im C_v}.\\
              }
\end{split}
\end{equation}
commutes.
By definition of~$\calK$ and Definition~\ref{D:W_v}\eqref{I:W_v is kernel}
of~$W_v$, the rows are exact.
Apply the snake lemma and take the first four terms of the snake;
Lemma~\ref{L:Seltf} identifies the kernel of the third vertical map.
\end{proof}

\begin{corollary}
\label{C:bounding Sel to Selfake}
We have 
$\#\ker\left(\Sel{\phi}(J) \to \Seltf{\alpha}(J) \right) \; \le \;
\# R^\vee(k)/q E^\vee(k)$.
\end{corollary}

\begin{proof}
By Theorem~\ref{T:Selseq},
the kernel is isomorphic to
$\ker \kappa \subseteq \calK \isom R^\vee(k)/q E^\vee(k)$.
\end{proof}

The following summarizes the best possible situation.

\begin{corollary} \label{C:Nicecase}
  Assume that $\Seltf{\alpha}(J) \subseteq \alpha\bigl(H^1(k, A[\phi])\bigr)$
  and that \hbox{$W_v = 0$} for all places~$v$ of~$k$. 
Then we have an exact sequence
  \[ 0 \To \calK \To \Sel{\phi}(J) \To \Seltf{\alpha}(J) \To 0 . \]
  In particular, $\#\Sel{\phi}(J) = \# \calK \cdot \#\Seltf{\alpha}(J)$.
\end{corollary}

\begin{proof}
In Theorem~\ref{T:Selseq}, 
the homomorphism $\kappa \colon \calK \to \prod_v W_v$ is~$0$.
\end{proof}

We need a criterion that tells us when $\Seltf{\alpha}(J)$ is already
contained in $\alpha\bigl(H^1(k, A[\phi])\bigr)$. Recall the discussion
of~$\Sha^1$ in Section~\ref{S:sha1} and the exact sequence
\[ 0 \To A[\phi] \To E^\vee \stackrel{q}{\To} R^\vee \To 0 . \]

\begin{lemma} \label{L:Sha1R} \hfill
  \begin{enumerate}[\upshape (a)]
  \item\label{I:Sha1R}
  There is an exact sequence
  \[ 0 \To \Seltf{\alpha}(J) \intersect \alpha\bigl(H^1(k, A[\phi])\bigr)
       \To \Seltf{\alpha}(J)
       \stackrel{q}{\To} \Sha^1(k, R^\vee).
  \]
\item\label{I:if Sha=0}
In particular, if $\Sha^1(k, R^\vee) = 0$, then 
$\Seltf{\alpha}(J) \subseteq \alpha\bigl(H^1(k, A[\phi])\bigr)$.
  \end{enumerate}
\end{lemma}

\begin{proof}
The last three terms in \eqref{E:true main sequence} 
or~\eqref{E:fake main sequence} 
for~$k$ and the~$k_v$ give rise to a commutative diagram
with exact rows:
\[
\xymatrix{
0 \ar[r] & \alpha\left( H^1(k,A[\phi]) \right) \ar[r] \ar[d] & H^1(k,E^\vee) \ar[r]^q \ar[d] & H^1(k,R^\vee) \ar[d] \\
0 \ar[r] & \displaystyle\prod_v \frac{\alpha_v\left( H^1(k_v,A[\phi]) \right)}{\im C_v} \ar[r] & \displaystyle\prod_v \frac{H^1(k_v,E^\vee)}{\im C_v} \ar[r] & \displaystyle\prod_v H^1(k,R^\vee) \\
}
\]
Take the kernels of the three vertical maps and apply 
Lemma~\ref{L:Seltf}.
\end{proof}

\begin{remark}
Results of Section~\ref{S:sha1} allow us to bound~$\Sha^1(k, R^\vee)$
and to prove that it is~$0$ in many cases.
\end{remark}

\begin{example}
\label{Ex:hyperelliptic descent}
We revisit the fake descent setup in 
Example~\ref{Ex:fake setup examples}\eqref{E:hyperelliptic example}
used for 2-descent 
on the Jacobian~$J$ of a hyperelliptic curve~$X$ with an even degree model
$y^2=f(x)$.
Assume that $X$ has even genus or that $X(k_v) \ne \emptyset$
for all~$v$; then Hypothesis~\ref{H:circ} is satisfied,
by Lemma~\ref{L:circ}.
The $\calG_k$-set~$\fes$ is the set of Weierstrass points.
The group~$\Z/2\Z$ injects into $E = (\Z/2\Z)^\fes_{\deg 0}$,
and $E^\vee = \mu_2^\fes/\mu_2$.
Since $\widehat{J} \isom J$, 
exact sequence~\eqref{E:definition of R} is
\[ 
	0 \To \frac{\Z}{2\Z} \To E \To J[2] \To 0, 
\]
and its dual~\eqref{E:alpha and q} is
\[ 
	0 \To J[2] \stackrel{\alpha}\To E^\vee \stackrel{N}{\To} \mu_2 \To 0
\]
where $N$ is the norm. 
Cohomology (see~\eqref{E:fake main sequence}) gives
  \[ 0 \To \calK \To H^1(k, J[2]) \stackrel{\alpha}\To H^1(k, E^\vee),\]
where $\calK \isom \mu_2(k)/N(E^\vee(k))$.
One can show (see~\cite{Poonen-Schaefer1997}*{Theorem~11.3})
that the element of~$\calK$ represented by~$-1 \in \mu_2(k)$
corresponds to the class~$\xi$
of the $2$-covering $\PIC^1_{X/k} \to J$ in~$H^1(k,J[2])$,
where $\PIC^1_{X/k}$ is the Picard scheme component
parametrizing line bundles of degree~$1$ on~$X$.
For each~$v$,
our assumption implies that $\PIC^1_{X/k}$ has a $k_v$-point,
so $\xi_v$ maps to~$0$ in~$H^1(k_v,J)$,
so $\xi_v \in \im\left(C_v \colon J(k_v)/2J(k_v) \to H^1(k_v,J[2]) \right)$.
Since $\xi_v$ generates $\ker \alpha_v$,
we have $W_v=0$ by Definition~\ref{D:W_v}\eqref{I:W_v is image}.
Also, $\Sha^1(k,\mu_2)=0$ by Lemma~\ref{lemma:splitsha},
so Lemma~\ref{L:Sha1R}\eqref{I:if Sha=0} applies,
and Corollary~\ref{C:Nicecase} yields an exact sequence
\[ 
	0 \To \calK \To \Sel{2}(J) \To \Seltf{\alpha}(J) \To 0.
\]
Finally, since $\calK \isom \mu_2(k)/N(E^\vee(k))$,
we have that $\calK=0$ if 
$\fes$ has a Galois-stable unordered partition into
two sets of odd cardinality (a computable condition),
and $\calK \isom \mu_2(k)$ otherwise 
(cf.~\cite{Poonen-Schaefer1997}*{Theorem~13.2}).
\end{example}


\section{Computing true and fake Selmer groups} \label{S:compute fake}

Choose $\calS$ and $\calT$ as in Section~\ref{S:finite description},
and compute~$\widetilde{L(n,\calS)}$
as in Section~\ref{S:unramified in target of descent}.
Because of Theorem~\ref{T:finite description of Selmer},
it remains to find an algorithm to test whether a given element
of~$\widetilde{L(n,\calS)}$ is in the local image~$\im C_v$ for a given~$v$.
It is clear that this can be done in principle,
but our goal here will be to describe a practical method,
under an additional hypothesis:
\begin{hypothesis}
\label{H:curve with a k_v-point}
The variety~$X$ is a genus-$g$ curve with a $k_v$-point~$x_0$
(or more generally, a degree~$1$ divisor~$x_0$ over~$k_v$),
and either $\Char k_v = 0$ or $\Char k_v > g$.
\end{hypothesis}
We consider the non-archimedean and archimedean cases separately.

\subsection{Computing the local image at a non-archimedean place}
\label{S:non-arch local image}

Fix a non-archimedean place~$v$ of~$k$.
Let $K$ be a finite extension of~$k_v$, say of degree~$d$,
with valuation ring $\calO$, uniformizer $\pi$, and maximal ideal~$\mm$.
Let $L_K\colonequals L \tensor_k K$.
We have a map $X(K) \to \calZ^0(X_K)$ sending $x$ to the class
of the $0$-cycle $x-x_0$,
and following this with~$\widetilde{C}_K$ defines
a continuous map 
\[
	\cc_K \colon X(K) \to 
		\widetilde{\frac{L_K^\times}{L_K^{\times n} K^\times}}.
\]
To define $\cc_K$ on all of~$X(K)$ requires using several
functions $f_1,\ldots,f_r$ with disjoint support,
as in Remark~\ref{R:moving}.
By perturbing $x_0$ in the smooth space~$X(k_v)$
(if $x_0$ is a point)
or perturbing each component of~$x_0$ in the space of points
over its field of definition (if $x_0$ is a divisor),
we may assume that the~$f_i$ can be evaluated at~$x_0$.

We now explain how to compute a finite description of~$\cc_K$.
Choose a proper $\calO$-scheme~$\calX$ with $\calX_K \isom X_K$.
By the valuative criterion for properness, 
$X(K) \isom \calX(\calO) \isom \varprojlim_{m \ge 0} \calX(\calO/\pi^m)$.
By Hensel's lemma, $L_K^{\times n}$ has finite index in~$L_K^\times$,
so $\widetilde{L_K^\times/L_K^{\times n} K^\times}$
is a finite discrete set,
so $\cc_K$ is locally constant.
Proceed as follows:
\begin{itemize}
\item 
Start with $m=1$.
\item
On each remaining residue disk modulo $\pi^m$ in~$X(K)$,
check whether any~$f_i$ is such that $f_i(x)/f_i(x_0)$
is constant in~$\widetilde{\frac{L_K^\times}{L_K^{\times n} K^\times}}$;
if not, break the residue disk into residue disks modulo~$\pi^{m+1}$,
and apply recursion.
\end{itemize}
Because the divisors of the~$f_i$ are disjoint,
eventually this algorithm will terminate,
with a partition of~$X(K)$ into residue disks on which $\cc_K$
takes a known constant value.

Next, the diagram
\[\xymatrix@C3em{
X(K)\ar[r]^{x \mapsto x-x_0} \ar[rd]_{x \mapsto \tr_{K/k_v} x  - dx_0} & \calZ^0(X_K)\ar[r] \ar[d]^{\tr_{K/k_v}} & \widetilde{\dfrac{L_K^\times}{L_K^{\times n} K^\times}} \ar[d]^{N_{K/k_v}}\\
&\calZ^0(X_{k_v})\ar[r] & \widetilde{\dfrac{L_v^\times}{L_v^{\times n} k_v^\times}} \\
}\]
commutes, so we can compute $\widetilde{C}$ on elements of~$\calZ^0(k_v)$
of the form $\tr_{K/k_v} x  - dx_0$.

\begin{lemma}
\label{L:generating J(K)}
\hfill
\begin{enumerate}[\upshape (a)]
\item\label{I:finitely many extensions}
There are only finitely many extensions~$K/k_v$ with $[K:k_v] \le g$
up to isomorphism, and we can list them all.
\item As $K$ ranges through these field extensions, and $x$ ranges over $X(K)$,
the images of $\tr_{K/k_v} x  - dx_0$ generate~$J(k_v)$.
\end{enumerate}
\end{lemma}

\begin{proof}
\hfill
\begin{enumerate}[\upshape (a)]
\item 
If $\Char k_v=0$, see~\cite{Pauli-Roblot2001}.
If $\Char k_v>0$, 
the characteristic assumption in Hypothesis~\ref{H:curve with a k_v-point}
guarantees that each~$K$ is tamely ramified over~$k_v$,
so $K$ is obtained by adjoining a single $m^{\tH}$~root of a uniformizer
to a finite unramified extension of~$k_v$.
We can compute representatives for the set of possible uniformizers
up to $m^{\tH}$~powers.
\item 
The Riemann-Roch theorem shows that every degree-$0$ divisor on~$X$
is linearly equivalent to $E-g x_0$ for some effective $E \in \Div^g X$.
Writing $E$ as a sum of closed points lets one express $E-g x_0$
as a sum of $0$-cycles of the form $\tr_{K/k_v} x - dx_0$
for various~$K$ of degree at most~$g$.
\qedhere
\end{enumerate}
\end{proof}

We can now compute $\im C_v$ as follows:
\begin{enumerate}[\upshape (i)]
\item Compute all extensions $K$ as in 
Lemma~\ref{L:generating J(K)}\eqref{I:finitely many extensions}.
\item For each, compute the image of~$\cc_K$, and apply
$N_{K/k_v}$ to find the image of~$C_v$ on $0$-cycles
of the form $\tr_{K/k_v} x - dx_0$.
\item Compute the subgroup generated by all such images.
\end{enumerate}
By Lemma~\ref{L:generating J(K)}, the result equals $\im C_v$.

\subsection{Computing the local image at an archimedean place}\label{S:arch local image}

For archimedean~$v$ a similar method works,
but instead of partitioning $X(K)$ into residue classes,
we partition it into connected components,
because a continuous map to the discrete set 
$\widetilde{L_v^\times/L_v^{\times n} k_v^\times}$
is constant on connected components.
If $K=\C$, then every point in~$X(K)$ has the same image under
$x \mapsto \tr_{K/k_v} x - dx_0$ as~$x_0$, which is~$0$.
If $K=\R$, then we compute $f(x)/f(x_0)$ for one point~$x$
in each connected component of~$X(K)$,
perturbing~$x$ as necessary to ensure that $f$ is regular and nonvanishing
at~$x$.

\subsection{Further comments on the computation of local images}

\begin{remark}
To find different $f$'s that can be used to compute~$\cc_K$,
we need to move~$\beta$ (and/or~$D$ in the fake case) 
within their linear equivalence classes.
In the fake case, it is often more convenient to move~$D$ in practice,
since in many applications $\beta$ is 
the only effective divisor in its class.
\end{remark}

\begin{remark}
\label{R:no moving necessary}
Sometimes we can compute $\cc_K$ without moving~$\beta$,
as we now explain.
Suppose that
\begin{itemize}
\item we have a true or fake descent setup, 
with a choice of $\beta$, $D$, and~$f$;
\item $f$ can be evaluated at~$x_0$;
where $x_0 \in X(K)$ or $x_0$ is a degree~$1$ divisor on~$X$ over~$K$,
for some local field~$K$;
\item the supports of~$\beta_P$ for~$P \in \fes$ are pairwise disjoint;
and
\item the image of the diagonal $\Z/n\Z \to (\Z/n\Z)^{\fes}$
is contained in~$R$.
\end{itemize}
Given $x \in X(K)$,
the hypothesis on the~$\beta_P$ implies that there is at most one~$P$
such that $f_P$ has a zero or pole at~$x$.
We can evaluate $f_Q(x-x_0)$ for all~$Q \ne P$,
and the missing component of~$\cc_K(x)$ can be recovered
from the fact (Lemma~\ref{L:kernel of norm})
that $\cc_K(x)$ lies in the kernel of~$N$.
\end{remark}

\begin{remark}
\label{R:no moving necessary for curve}
Under the same hypotheses on $\beta,D,f,R$ as in 
Remark~\ref{R:no moving necessary},
we can evaluate $C_f$ at arbitrary~$x \in X(k)$,
using Lemma~\ref{L:coset of kernel of norm}
instead of Lemma~\ref{L:kernel of norm}.
Similarly, we can evaluate~$C_{f,v}$.
\end{remark}

\begin{remark}
\label{R:image of random points}
Here we describe an alternative method that, when it succeeds, 
computes~$\im C_v$ much more quickly in practice.
The method consists of two parts:
computing an upper bound on~$\#\im C_v$ (or its exact value),
and computing lower bounds on the group~$\im C_v$,
and hoping that the sizes match.
Suppose that $A=J$ and $\phi$ is multiplication-by-$n$ on~$J$.
To compute an upper bound on~$\#\im C_v$:
\begin{itemize}
\item First calculate the action of 
the decomposition group $\Gal(k_{v,s}/k_v)$ on~$\fes$.
\item 
Using this, compute the exact sequence
\begin{equation}
  \label{E:local sequence computed}
	0\to J[n](k_v)\to E^\vee(k_v) \stackrel{q}\to R^\vee(k_v)
\end{equation}
of finite groups.
\item From this, compute $\# J[n](k_v)$.
\item
Substitute this into the formula
\begin{equation}
  \label{E:J(k_v) mod n}
	\#\frac{J(k_v)}{nJ(k_v)}
	= \|n\|_v^{-\dim J} \#J[n](k_v),
\end{equation}
where $\|\;\|_v$ is the $v$-adic absolute value normalized so that
$\|a\|_v=(\calO_v:a\calO_v)^{-1}$ for $a \in \calO_v$ 
(this is a variant of \cite{Poonen-Schaefer1997}*{Proposition~12.10}).
\item
We obtain the bound $\#\im C_v \le \#J(k_v)/n J(k_v)$,
since $C_v$ factors through~$J(k_v)/n J(k_v)$.
\item
If $q \colon E^\vee(k_v) \to R^\vee(k_v)$ is surjective,
a condition that can be checked using~\eqref{E:local sequence computed},
then 
\eqref{E:local descent diagram} shows that $\alpha_v$ is injective and
$\#\im C_v = \#J(k_v)/n J(k_v)$.
\end{itemize}
To compute lower bounds on~$\im C_v$:
\begin{itemize}
\item
Randomly select $x \in X(k_v) \injects J(k_v)$ 
and compute their images under~$C_v$
hoping that they generate a group of the right size.
(Calculate $N_{K/k_v}(\cc_K(x))$ for $x \in X(K)$ 
also for finite extensions~$K$ of~$k_v$
if it seems that the $k_v$-points are insufficient.)
\end{itemize}
\end{remark}


\section{Genus-3 curves}
\label{S:genus 3}

In this section we specialize to the case of a fake descent setup
given by the bitangents of a smooth plane quartic.
This is the non-hyperelliptic $g=3$ case of 
Example~\ref{Ex:fake setup examples}\eqref{E:odd theta}. 
Many of the results about smooth plane quartics we require 
were established already before 1900. 
See \cite{Salmon1879}*{p.~223 onwards} for a survey.

\subsection{Bitangents of smooth plane quartics}\label{S:bitangents}

Let $k$ be a field of characteristic not~$2$.
Let $X$ be a non-hyperelliptic genus-$3$ curve over~$k$.
The canonical map embeds $X$ as a smooth quartic curve $g(x,y,z)=0$
in~$\PP^2$.
Conversely, every smooth quartic curve in~$\PP^2$ arises in this way.
Let $\PPdual^2$ be the dual projective space, 
with homogeneous coordinates $u,v,w$;
its points correspond to lines in~$\PP^2$.

\begin{definition}
\label{D:bitangent}
A \defi{bitangent} to~$X_s$ is a line $l \subset \PP^2_{k_s}$ such that
the intersection $l.X$ is~$2\beta_l$ for some $\beta_l \in \Div X_{k_s}$.
\end{definition}
Given a bitangent~$l$,
the line bundle~$\LL_l$ associated to~$\beta_l$
is an odd theta characteristic on~$X_s$.
The bitangents to~$X_s$ form a $28$-element $\calG$-set~$\fes$
in bijection with the $\calG$-set~$\TT_{\odd}$ of
Section~\ref{S:theta characteristics} \cite{Gross-Harris2004}*{p.~289}.
Alternatively, we may view $\fes$ as a finite \'etale subscheme of~$\PPdual^2$,
and the collection~$(\LL_l)$ as a line bundle~$\LL$ on~$X \times \fes$.
Then $(2,\fes,\LL)$ is isomorphic to the fake descent setup
in Example~\ref{Ex:fake setup examples}\eqref{E:odd theta}
with $g=3$.
The isogeny $\phi \colon A \to J$ of Section~\ref{S:descent isogeny}
is $[2] \colon J \to J$.

\subsection{Syzygetic quadruples}
\label{S:syzygetic quadruples}

\begin{lemma}\label{L:bitang pair classes}
Let $l_1,\ldots,l_4$ be bitangents.
Then the following are equivalent:
\begin{enumerate}[\upshape (i)]
\item The corresponding four odd theta characteristics sum to~$0$
in $\frac{\Pic X_s}{\langle \omega \rangle}[2]$.
\item The divisor $\beta_{l_1}+\cdots+\beta_{l_4}$ 
is linearly equivalent to $2$~times a canonical divisor.
\item The divisor $\beta_{l_1}+\cdots+\beta_{l_4}$ is an intersection $Q.X$
for some (not necessarily irreducible) conic $Q \subset \PP^2$.
\end{enumerate}
\end{lemma}

\begin{proof}\hfill

(i)$\iff$(ii): 
Considering degrees shows that the odd theta characteristics sum to~$0$
in $\frac{\Pic X_s}{\langle \omega \rangle}[2]$
if and only if they sum to $\omega^{\tensor 2}$ in $\Pic X_s$.

(iii)$\implies$(ii): Let $l$ be any line in $\PP^2$.
Then $Q \sim 2l$ on $\PP^2$, so $Q.X \sim 2l.X$ on $X$,
and $l.X$ is a canonical divisor.

(ii)$\implies$(iii):
Equivalently, we must show that $\Gamma(\PP^2,\OO(2)) \to \Gamma(X,\OO_X(2))$
is surjective.
This is true, since the cokernel is contained in $H^1(\PP^2,\OO(-4+2))=0$.
\end{proof}

A $4$-element set
$\{l_1,\ldots,l_4\}$ 
satisfying the conditions of Lemma~\ref{L:bitang pair classes}
is called a \defi{syzygetic quadruple}.
Let $\Sigma \subset \binom{\fes}{4}$ be the set of syzygetic quadruples.
This incidence structure~$(\fes,\Sigma)$ constructed above from bitangents
is the same as that in Section~\ref{S:thetas and quadratic forms} for $g=3$.
In particular, its isomorphism type
is independent of~$X$ (see Remark~\ref{R:connectedness of moduli space}).
By Proposition~\ref{P:315}, $\#\Sigma=315$.

\subsection{$\Sp_6(\F_2)$-modules}
\label{S:modules}

If we chose a bijection between $\fes$ and $\{1,\ldots,28\}$,
then the image~$G$ of $\calG \to \Aut \fes$ would be identified with a subgroup
of the symmetric group~$S_{28}$, and changing the bijection
would change the subgroup up to~$S_{28}$-conjugacy.

Instead we will choose an isomorphism between 
$(\fes,\Sigma)$ and a fixed incidence structure
$(\mathbf{\fes},\mathbf{\Sigma})$,
so that $G$ is identified with a subgroup of 
$\Aut(\mathbf{\fes},\mathbf{\Sigma})$,
which by Proposition~\ref{P:structures}
is a specific copy of~$\Sp_6(\F_2)$ in~$S_{28}$.
Changing {\em this} isomorphism changes~$G$ only up to $\Sp_6(\F_2)$-conjugacy.
This more refined information will be needed to deduce the action
of~$G$ on $J[2]$, $E^\vee$, and~$R^\vee$.

The following lemma constructs our fixed $(\mathbf{\fes},\mathbf{\Sigma})$ 
directly from the group~$\Sp_6(\F_2)$:
\begin{lemma}\label{L:incidence structure from group}
\hfill
\begin{enumerate}[\upshape (a)]
\item\label{I:index 28}
The group $\Sp_6(\F_2)$ has a unique index-$28$ subgroup~$H$,
up to conjugacy.  Let $\mathbf{\fes}$ be the $\calG$-set $\calG/H$.
Identify $\mathbf{\fes}$ with~$\{1,\ldots,28\}$.
\item There is a unique $\Sp_6(\F_2)$-invariant subset 
of $\binom{\mathbf{\fes}}{4}$ of size~$315$;
call it~$\mathbf{\Sigma}$.
\end{enumerate}
\end{lemma}

\begin{proof}
Direct computation using {\tt Magma}~\cite{Magma}.
\end{proof}

Let $\mathbf{\Gamma}$ be the image of the map
\begin{align*}
  \binom{\mathbf{\fes}}{2} &\to \binom{\mathbf{\fes}}{12} \\
	\pi &\mapsto \bigcup \{\sigma\in \Sigma: \pi\subset\sigma\}.
\end{align*}
In~$\F_2^{\mathbf{\fes}}$,
let $\mathbf{1}$, $\widetilde{\mathbf{J}}$, $\mathbf{R}$, $\mathbf{E}$
be the $\F_2$-spans of $\{\sum_{l \in \mathbf{\fes}} l\}$,
$\mathbf{\Gamma}$, $\mathbf{\Sigma}$, $\binom{\mathbf{\fes}}{2}$,
respectively,
where we identify subsets of~$\mathbf{\fes}$ with the sums of their elements
in~$\F_2^{\mathbf{\fes}}$.

\begin{lemma}
\label{L:submodules}
The $\Sp_6(\F_2)$-submodules of~$\F_2^{\mathbf{\fes}}$ are
\[
	0\subset \mathbf{1} \subset \widetilde{\mathbf{J}}
	\subset \mathbf{R} \subset \mathbf{E} \subset \F_2^{\mathbf{\fes}},
\]
which have dimensions $0,1,7,21,27,28$, respectively.
\end{lemma}

\begin{proof}
The given modules obviously are submodules.
Direct computation using {\tt Magma} establishes that there are no others
and that the dimensions are as stated.
\end{proof}

\begin{corollary}
\label{C:specialization}
Let $X$ be a smooth plane quartic.
Let $\fes$ be its set of bitangents.
Let $\Sigma$ be its set of syzygetic quadruples.
Construct $E$ and~$R$ from the fake descent setup
as in Section~\ref{S:fake descent}.
Then there is an isomorphism 
$(\fes,\Sigma) \isom (\mathbf{\fes},\mathbf{\Sigma})$
of incidence structures,
and any such isomorphism induces a homomorphism
$\rho \colon \calG \to \Sp_6(\F_2)$
and $\calG$-equivariant isomorphisms
$E \isom \mathbf{E}$, $R \isom \mathbf{R}$,
and $J[2] \isom \ker(\mathbf{E}^\vee \to \mathbf{R}^\vee)$,
where $\calG$ acts on the $\Sp_6(\F_2)$-modules via~$\rho$.
\end{corollary}

\begin{proof}
The isomorphism $(\fes,\Sigma) \isom (\mathbf{\fes},\mathbf{\Sigma})$
exists because of the uniqueness in Lemma~\ref{L:incidence structure from group}.
The $\calG$-action on~$\fes$ must respect~$\Sigma$,
so we get
\[
	\calG \to \Aut(\fes,\Sigma) \isom \Aut(\mathbf{\fes},\mathbf{\Sigma}) 
	= \Sp_6(\F_2).
\]
The induced isomorphism $\F_2^{\fes} \to \F_2^{\mathbf{\fes}}$
sends $E$ to~$\mathbf{E}$ and $R$ to~$\mathbf{R}$
because by Lemma~\ref{L:submodules}
there is at most one submodule of each dimension.
By \eqref{E:alpha and q}, we have
\[
	J[2] \isom \ker\left(E^\vee \to R^\vee\right) \isom 
	\ker(\mathbf{E}^\vee \to \mathbf{R}^\vee).\qedhere
\]
\end{proof}

\subsection{Computing bitangents and syzygetic quadruples}

\begin{lemma}\label{L:computing bitangents and quadruples}
Let $k$ be a field for which arithmetic operations can be computed.
Suppose that $g(x,y,z) \in k[x,y,z]$
is a degree-$4$ homogeneous polynomial 
defining a smooth curve~$X$ in~$\PP^2$.
\begin{enumerate}[\upshape (a)]
\item\label{I:computing Delta}
We can compute $\fes$ as a subscheme in~$\PPdual^2_k$.
\item\label{I:computing Sigma}
In the remaining parts, suppose that $F$ is 
an explicit finite Galois extension of $k$
over which $\fes$ splits completely,
and write $\fes$ for $\fes(F)$.
If we choose a bijection $\fes \isomto \{1,\ldots,28\}$,
then $\Sigma$ can be determined as an explicit subset
of~$\binom{\{1,\ldots,28\}}{4}$.
\item\label{I:syzygetic isomorphism}
With notation as in~\eqref{I:computing Sigma},
we can compute an isomorphism 
$(\fes,\Sigma) \to (\mathbf{\fes},\mathbf{\Sigma})$.
\item\label{I:Galois group}
If we have an explicit description of~$\Gal(F/k)$ acting on~$F$,
then the image of the homomorphism $\rho \colon \calG_k \to \Sp_6(\F_2)$
of Corollary~\ref{C:specialization} can be computed.
\end{enumerate}
\end{lemma}

\begin{proof}\hfill
  \begin{enumerate}[\upshape (a)]
  \item 
We describe $\fes$ as a subscheme of~$\PPdual^2$
by describing its intersection with each standard affine patch of~$\PPdual^2$.
For example, to compute the part of~$\fes$ in the patch
whose points correspond to lines $l \colon ux+vy+wz=0$ with $w=1$,
substitute the parametrization $(s:t)\mapsto (s:t:-us-vt)$ of the line
into $g(x,y,z)$, set the result equal to $(a_0 s^2 + a_1 st + a_2 t^2)^2$
for indeterminate $a_0,a_1,a_2$,
and eliminate $a_0,a_1,a_2$ to find the conditions on $u,v,w$
for $l$ to be a bitangent.
\item 
We use criterion~(iii) in Lemma~\ref{L:bitang pair classes}.
Enumerate the $3$-element subsets $\{l_1,l_2,l_3\}$ of~$\fes$.
For each, use linear algebra over $F$ to determine if there is a
conic~$Q$ in~$\PP^2$ such that
$Q.X \ge \beta_{l_1} + \beta_{l_2} + \beta_{l_3}$;
if so, compute $Q.X$ to check whether it equals
$\beta_{l_1} + \beta_{l_2} + \beta_{l_3} + \beta_{l_4}$
for some $l_4 \in \Delta$.
Record all such $\{l_1,\ldots,l_4\}$.
\item 
Given \eqref{I:computing Sigma},
this is a matter of matching combinatorial data.
A bijection $\fes \to \fes'$ can be built one value at a time;
we try all possibilities, checking as we go along that the
distinguished $4$-element subsets match so far. 

\item 
We compute equations for the bitangents,
and hence compute the action of $\Gal(F/k)$ on $\fes$.
Using~\eqref{I:syzygetic isomorphism}, we translate this
into an action of~$\Gal(F/k)$ on~$\mathbf{\fes}$.
The image of $\Gal(F/k) \to \Aut \mathbf{\fes}$
is~$\im \rho$.
\qedhere
  \end{enumerate}
\end{proof}

\subsection{Computing the discriminant of a ternary quartic form}
\label{S:computing discriminant}

The following is well-known.

\begin{lemma}
Fix positive integers $n$ and~$d$.
Let $g(x_0,\ldots,x_n) \in \Z[x_0,\ldots,x_n]$
be a degree-$d$ homogeneous polynomial with indeterminate 
coefficients $c_1,\ldots,c_N$,
so $N = \binom{n+d}{d}$.
Then there is a polynomial $D(c_1,\ldots,c_N) \in \Z[c_1,\ldots,c_N]$,
called the \defi{discriminant},
such that whenever $c_1,\ldots,c_N$ are specialized to elements
of a field~$k$,
the polynomial~$D$ vanishes if and only if
the hypersurface $g=0$ in~$\PP^n_k$ fails to be smooth of dimension~$n-1$.
Moreover, $D$ is unique up to a sign.
\end{lemma}

\begin{proof}[Sketch of proof]
Let $B=\Aff^N_\Z$, and let $X \subset \PP^n \times B$ be the closed 
subscheme defined by $g=0$, so $X \to B$ is the universal family
of degree-$d$ hypersurfaces in~$\PP^n$.
Let $S$ be the locus where $X \to B$ fails to be smooth of dimension~$n-1$.
Then $S \to B$ is proper, 
so its image is a closed subscheme $V \subseteq \Aff^N_\Z$.
On the other hand, analyzing the fibers of the other projection $S \to \PP^n$
lets us show that $S$ and~$V$ are integral (or empty)
and lets us compute their dimensions.
This eventually shows that $V$ is a hypersurface (possibly empty),
so $V$ is cut out by a single equation, defined up to a unit of~$\Z$.
\end{proof}

We now restrict to the case $n=2$ and~$d=4$,
in which case $D$ is a polynomial of degree~$27$,
denoted~$I_{27}(g)$ in the notation of~\cite{Dixmier1987}.
We fix the sign by decreeing that $I_{27}(x^4+y^4+z^4)>0$. 

Over a ring in which $\deg g$ is invertible,
the discriminant of~$g$ agrees, up to a unit,
with the resultant~$R(g)$ of
$\frac{\del g}{\del x}, \frac{\del g}{\del y}, \frac{\del g}{\del z}$ 
(cf.~\cite{Gelfand-Kapranov-Zelevinsky2008}*{Chapter~13, \S1}),
for which an algorithm is given in~\cite{Salmon1876}*{\S91}.
Thus $I_{27}(g) = c R(g)$ for some $c \in \Z[1/2]^\times$ independent of~$g$.

To determine the power of~$2$ in~$c$,
we compute~$R(g)$ for a single $g \in \Z_2[x,y,z]$
for which $g=0$ is smooth over~$\Z_2$
(e.g., take $g = x^4 + y^3 z + z^3 y$);
for this~$g$ we must have $c R(g) \in \Z_2^\times$.
The sign of~$c$ can be determined as the sign of~$R(x^4+y^4+z^4)$.
It turns out that $c=2^{-14}$ 
(cf.~\cite{Gelfand-Kapranov-Zelevinsky2008}*{Chapter~13, Proposition~1.7} 
and~\cite{Demazure2012}*{\S5, D\'efinition~4 and Exemple~3}).

This algorithm was implemented by Christophe Ritzenthaler, 
and included by David Kohel 
in his {\tt Magma} package {\tt Echidna}~\cite{Kohel-Echidna}.
We used this implementation
(but negated the output to make the sign agree with our convention).

\begin{remark}
Given a smooth plane quartic curve~$X$ over~$\Q$ defined by $g=0$,
we can choose $g \in \Z[x,y,z]$ to minimize~$|I_{27}(g)|$.
But $I_{27}(-g)=-I_{27}(g)$, 
so the {\em sign} of~$I_{27}(g)$ for such a minimizer $g$
is not uniquely determined by the curve $X$ in $\PP^2_\Q$.
\end{remark}

\subsection{Computing fake Selmer groups of smooth plane quartics}
\label{S:g3 compute}

We apply the procedure outlined in Section~\ref{S:compute fake} 
to compute~$\Selfake{\alpha}(J)$ for the fake descent setup given
in Section~\ref{S:bitangents},
under the assumption that $X_{k_v}$ has a divisor of degree~$1$
for all $v \in \Omega_k$.
For simplicity and practicality, we assume that $k=\Q$.
By multiplying the polynomial $g(x,y,z)$ defining~$X$ in~$\PP^2$
by a positive integer, we may assume that $g$ has coefficients in~$\Z$.
We proceed with the following steps:
\begin{enumerate}
 \item Determine $L$.
 \item Determine the ring of integers $\calO_L$ of~$L$.
 \item Determine a finite set $\calS\subset\Omega_k$ 
	containing the archimedean places, the
	places of residue characteristic~$2$,
	and the places~$v$ where the Tamagawa number
	$c_v(J)$ is even.
 \item Determine $\widetilde{L(2,\calS)}$.
 \item Determine the image $G$ of $\calG \to \Aut \fes$
       as a subgroup of~$\Sp_6(\F_2)$ up to $\Sp_6(\F_2)$-conjugacy.
 \item Choose a finite set $\calT\subset\Omega_k$. For each~$v\in\calT$,
 \begin{enumerate}
   \item Determine the decomposition group~$D_v$ as a subgroup of~$G$
         up to $G$-conjugacy.
   \item Determine $\im C_v$.
 \end{enumerate}
 \item Compute $S_\calT'=\{\delta\in \widetilde{L(2,\calS)}: 
	N_{L/k}(\delta)=1 \in \frac{k^\times}{k^{\times 2}} 
	\textup{ and } 
	\Res_v\delta \in \im C_v \textup{ for all } v\in\calT\}$.
\end{enumerate}
Sections~\ref{S:determine L}--\ref{S:local computations} 
elaborate on the implementation of the corresponding steps.

Lemma~\ref{L:kernel of norm} and 
Theorem~\ref{T:finite description of Selmer}\eqref{I:upper bound for fake Sel}
imply $\Selfake{\alpha}(J) \subseteq S_\calT'$.
So the result of the calculation is an upper bound~$S_\calT'$
for~$\Selfake{\alpha}(J)$,
which by Theorem~\ref{T:finite description of Selmer}\eqref{I:good T}
can be guaranteed to equal~$\Selfake{\alpha}(J)$ by choosing $\calT$
appropriately large.
(But as explained in Remark~\ref{R:smaller T}, 
in practice we will often not choose $\calT$ so large.)

\subsubsection{Determine $L$}\label{S:determine L}

Our goal is to find a squarefree $h(t) \in k[t]$
of degree~$28$ such that $\fes \isom \Spec L$ with $L\colonequals k[t]/h(t)$.
First compute~$\fes$ as a subscheme of~$\PPdual^2$ as in
Lemma~\ref{L:computing bitangents and quadruples}\eqref{I:computing Delta}.
Next, repeatedly choose a projection $\PPdual^2\to\PP^1$
(i.e., perform a change of variables, and then eliminate a variable)
until one is found that maps $\fes$ to an \'etale scheme~$\fes'$
of degree~$28$ with $\infty \notin \fes'$.
Then $\fes'$ is the zero locus in~$\Aff^1$ of the desired $h(t) \in k[t]$.

Let $\theta$ be the image of~$t$ in~$L$.

\subsubsection{Determine the ring of integers of $L$}

Here our task is to find a $\Z$-basis of~$\calO_L$,
with each basis element expressed as a polynomial in~$\theta$
of degree less than~$28$.
This is a standard operation in algebraic number theory packages 
if $L$ is a field. 
If $L$ splits nontrivially as a product of fields, 
we just compute the integer rings of each constituent separately 
and piece the results together using the Chinese remainder theorem.

{}From the denominators of the coefficients of~$h(x)$
we find $m\in\Z$ such that $m\theta$ is an algebraic integer. 
The standard approach for determining~$\calO_L$ would then be
to compute the integral closure of $\calO\colonequals \Z[m\theta]$ in~$L$.
But since $\theta$ was obtained
by projecting a degree~$28$ subscheme of~$\PP^2$ onto~$\PP^1$,
there are probably primes introduced into~$\Disc \calO$
that do not divide $\Disc \calO_L$.  In practice these primes can easily
be of size~$10^{100}$, so even finding them in~$\Disc \calO$ may involve
a challenging factorization. 

Here are two ways to circumvent this problem:
\begin{enumerate}[\upshape (1)]
\item 
Compute the \defi{discriminant}~$I_{27}(g)$ of the quartic form
as in Section~\ref{S:computing discriminant},
and factor it.
If $p$ is a prime not dividing the integer~$I_{27}(g)$,
then the $28$~bitangents of~$X$ reduce to the $28$~distinct bitangents
of the smooth plane quartic curve defined by $g(x,y,z) \bmod p$,
so $p \nmid \Disc \calO_L$.
In other words, the set of prime factors of~$I_{27}(g)$
is an upper bound for the set of primes dividing~$\Disc \calO_L$.
{\tt Magma}'s routine {\tt MaximalOrder} accepts this upper bound as part
of the input to help it remove extraneous factors via a lazy factorization.
\item
Use a different projection to find a second order~$\calO'$.
Then the order generated by $\calO$ and~$\calO'$
is likely of small index in~$\calO_L$,
in which case computing its integral closure in~$L$ is much easier.
\end{enumerate}

\begin{remark}
It is advisable to compute an LLL-reduced basis for~$\calO_L$ and use that
to represent elements of $\calO_L$ and~$L$.
\end{remark}

\subsubsection{Determine the set $\calS\subset\Omega_k$}

Compute the integer~$I_{27}(g)$ and factor it.
If $p$ is a prime such that $p \nmid I_{27}(g)$, 
then $X$ has good reduction at~$p$,
from which it follows that $J$ has good reduction at~$p$
and $c_p(J)=1$.
Therefore we may take the set~$\calS$
of Proposition~\ref{P:Sel is unramified}
to be the set of prime divisors of~$I_{27}(g)$,
together with $2$ and~$\infty$.

\begin{remark}
\label{R:smaller S}
Often we can use a smaller~$\calS$ by computing a proper
regular model of~$X$ over~$\Z_p$ and using
\cite{Bosch-Lutkebohmert-Raynaud1990}*{\S9.6,~Theorem~1} 
to compute~$c_p(J)$.
{\tt Magma} has an implementation by Steve Donnelly 
that can compute regular models of plane curves in many cases; 
there is also a function that extracts the component group~$\Phi$;
then $c_p(J)=\#\Phi(\F_p)$ is a divisor of~$\#\Phi$,
so this may let us prove that $c_p(J)$ is odd, 
in which case $p$ can be excluded from~$\calS$.
\end{remark}

Here is a common special case in which no extra computation is required
to exclude a prime:
\begin{remark}
\label{R:primes of multiplicity 1}
Suppose that an odd prime $p$ appears with exponent~$1$ in
the factorization of~$I_{27}(g)$.
Then $\Proj \Z_p[x,y,z]/(g)$ is regular and its special fiber is
an irreducible curve with multiplicity~$1$, with a single node.
We deduce $c_p(J)=1$, so we may exclude $p$ from~$\calS$.
\end{remark}

\subsubsection{Determine $\widetilde{L(2,\calS)}$}

Our goal here is to compute elements of~$L^\times$
whose images in~$L^\times/L^{\times 2} k^\times$
form a basis for~$\widetilde{L(2,\calS)}$.

We first compute~$L(2,\calS)$.
If we compute $L(2,\calS')$ for some $\calS' \supset \calS$,
we can recover its subgroup~$L(2,\calS)$
by performing linear algebra with the valuations in~$\calS'-\calS$;
thus we are free to enlarge~$\calS$.
If $\calS$ is enlarged enough that 
we can verify that $\Cl(\calO_{L,\calS})[2]=0$,
then \eqref{E:L(n,S)} implies that
$L(2,\calS)=\calO_{L,\calS}^\times/\calO_{L,\calS}^{\times 2}$.

Next, since $k=\Q$, we have $\Cl(\calO_\calS)=0$,
so \eqref{E:widetilde{L(n,S)}} implies
$\widetilde{L(2,\calS)} = \coker\left( \calO_S^\times/\calO_S^{\times 2} 
	\to \calO_{L,\calS}^\times/\calO_{L,\calS}^{\times 2}\right)$.

Thus we need only solve the following standard problems
of algebraic number theory:
\begin{enumerate}[\upshape (A)]
\item\label{I:find S} 
find $\calS$ such that $\Cl(\calO_{L,\calS})[2]=0$, and
\item\label{I:compute units} 
compute bases for $\calO_S^\times/\calO_S^{\times 2}$ and 
$\calO_{L,\calS}^\times/\calO_{L,\calS}^{\times 2}$.
\end{enumerate}

Current methods for solving problem~\eqref{I:find S}
involve computing the whole class group $\Cl(\calO_{L,\calS})$.
Doing this unconditionally requires computation up to the Minkowski
constants of the fields~$L_i$ with product~$L$,
which requires running time polynomial in~$\Disc \calO_{L_i}$;
in the case where $L$ is a degree~$28$ field, 
it is rare that $\Disc \calO_L$ is small enough to make this practical.
On the other hand, if we are willing to assume the Generalized Riemann
Hypothesis for the~$L_i$ of large degree,
we can handle many more instances.

Problem~\eqref{I:compute units} is less serious, 
once problem~\eqref{I:find S} has been solved.
One can check whether an element of~$\calO_{L,\calS}^\times$ is a square,
by constructing the square root numerically to prove a positive answer
and by reducing modulo primes to prove a negative answer.
The Dirichlet $\calS$-unit theorem 
yields $\dim_{\F_2} \calO_{L,\calS}^\times/\calO_{L,\calS}^{\times 2} = \#\calS$
in advance, so if generators are constructed conditionally,
they can be verified without difficulty.

\begin{remark}
{\tt Magma} has a command to compute $L(2,\calS)$ when $k=\Q$.
\end{remark}

\subsubsection{Choice of $\calT$}
\label{S:choice of T}

Choose a finite set of places~$\calT$,
while keeping Remark~\ref{R:smaller T} in mind.
If it is necessary to get a better upper bound~$S_\calT'$ on
$\Seltf{\alpha}(J)$, we can enlarge $\calT$ later.

\subsubsection{Determine the global and local Galois actions on $\fes$}
\label{S:g3 galois}

We have two goals:
\begin{enumerate}[\upshape (A)]
\item\label{I:global Galois}
Describe the image $G$ of
the homomorphism $\rho \colon \calG \to \Sp_6(\F_2)$ 
in Corollary~\ref{C:specialization}
as a subgroup of~$\Sp_6(\F_2)$ up to $\Sp_6(\F_2)$-conjugacy.
Since in Lemma~\ref{L:incidence structure from group}\eqref{I:index 28} 
we fixed an embedding $\Sp_6(\F_2) \injects S_{28}$,
the answer can be specified by giving explicit elements of~$S_{28}$
that generate one such representative~$G$ of its conjugacy class.
\item\label{I:local Galois}
For each $v \in \calT$,
describe the decomposition group~$D_v$
as a subgroup of the~$G$ in \eqref{I:global Galois} up to $G$-conjugacy.
Each $D_v$ is to be specified by a list of elements of $G \subseteq S_{28}$
that generate~$D_v$.
\end{enumerate}
If \eqref{I:global Galois} and~\eqref{I:local Galois} are done, 
then $E$, $R$, $J[2]$ with the actions of $\calG$ and~$\calG_{k_v}$
can be computed as submodules and subquotients
of~$\F_2^{\fes}$, as in Section~\ref{S:modules}.

{\tt Magma} shows that there are $1397$~conjugacy classes
of subgroups in~$\Sp_6(\F_2)$ to distinguish.
Here are some methods that can be applied (with varying degree
of success) towards determining the image~$G_K$ of $\calG_K \to \Sp_6(\F_2)$
(up to $\Sp_6(\F_2)$-conjugacy) for a field~$K$:
\begin{itemize}
\item \emph{Splitting field.} 
If a splitting field of $\fes$ over~$K$
is of small enough degree,
then Lemma~\ref{L:computing bitangents and quadruples}\eqref{I:Galois group}
is practical.
But in the worst case, the splitting degree is $\#\Sp_6(\F_2) = 1\,451\,520$.
\item \emph{Orbit sizes.}
Factoring $h(x)$ over~$K$ determines the orbit sizes of~$G_K$;
this already cuts the number of possibilities from $1397$ down to 
at most~$107$ (there are $107$~conjugacy classes with orbit size multiset
$\{4,8,16\}$).
Factoring $h(x)$ over the fields obtained by adjoining 
one root of (one factor of)~$h(x)$
determines the orbit size multisets for the one-point stabilizers;
these constrain $G_K$ further.
\item \emph{Decomposition subgroups.}
If $K$ is a global field~$k$,
and we have computed decomposition subgroups~$D_v$ in~$\Sp_6(\F_2)$
up to $\Sp_6(\F_2)$-conjugacy for several~$v$,
then these constrain the possibilities for~$G_K$.
\end{itemize}

Next we discuss the practicality of these methods for special types
of fields~$K$:
\begin{enumerate}[\upshape 1.]
\item \emph{Finite fields.} 
Here $G_K$ is cyclic.
{\tt Magma} shows that for each cyclic subgroup of~$\Sp_6(\F_2)$,
either the orbit size multiset determines the conjugacy class uniquely,
or the subgroup has size at most~$6$, 
in which case the splitting field method is practical.
So $G_K$ is determined easily.
\item \emph{Archimedean local fields.} 
The splitting field has degree at most~$2$,
so use the splitting field method.
\item \emph{Non-archimedean local fields.} 
If $X$ has good reduction, then we reduce to the case of a finite field.
Otherwise, let $q=p^e$ be the size of the residue field (for $k=\Q$, $e=1$);
then we have normal subgroups $P \triangleleft I$ of~$G_K$
(wild inertia and inertia)
such that $G_K/I$ is cyclic of some order~$f$,
and $I/P$ is cyclic of order dividing~$q^f-1$
and $P$ is a $p$-group.
If $p>7$, then $p \nmid \#\Sp_6(\F_2)$, so $P=\{1\}$,
and $G_K$ is metacyclic;
there are only $214$~conjugacy classes of metacyclic subgroups
of~$\Sp_6(\F_2)$, and the largest such subgroup has order~$60$,
which is about the largest degree 
for which the splitting field method is easily practical.
For $K=\Q_2,\Q_3,\Q_5,\Q_7$,
the maximal size of $G_K$ is $2304,432,120,60$,
respectively;
if the splitting field method is impractical,
we can eliminate possibilities by using the orbit size methods,
and by using that the local Galois group must be contained in
the global Galois group if the latter has been computed already.
\item \emph{The global field $\Q$.}
Determining the Galois group of a polynomial $h(t)\in\Q[t]$ 
is a standard problem in computational number theory: 
see \cites{Stauduhar1973, Geissler-Klueners2000, Geissler2003}.
A recent implementation by Claus Fieker and J\"urgen Kl\"uners in {\tt Magma}
chooses a prime~$p$, labels the roots of~$h$ in a splitting field~$K_p/\Q_p$
by giving approximations to them,
and computes the Galois group over~$\Q$ as a subgroup of
the permutations of the labels; in fact, by choosing an unramified prime~$p$,
the computations can be done in finite extensions of~$\F_p$ instead of~$\Q_p$.
Once this is done for our~$h$ of degree~$28$,
the splitting field method can then compute $G$ in~$\Sp_6(\F_2)$
up to $\Sp_6(\F_2)$-conjugacy.

An alternative to using the general-purpose implementation above
is to use orbit sizes and decomposition subgroups to try to
eliminate all possibilities except one.
When this succeeds, which in practice seems to be almost always the case,
it is often faster.
\end{enumerate}

Summary: In practice we can always achieve goal~\eqref{I:global Galois}.
Goal~\eqref{I:local Galois} may be more difficult to reach,
but the methods presented above often succeed.
{}From now on, we assume that both goals have been reached.

\subsubsection{Determining the local images}
\label{S:local computations}

Given $v \in \calT$, our goal is to compute elements of~$L_v^\times$
whose images in $L_v^\times/L_v^{\times 2} k_v^\times$ form a basis of~$\im C_v$.
We follow the approach in Section~\ref{S:compute fake}.

First let us explain how to compute the map~$\cc_K$
in Section~\ref{S:non-arch local image},
for any finite extension~$K$ of~$k_v$.
Let $(u_\theta:v_\theta:w_\theta)\in\fes(L)$ 
be the generic point of~$\fes$ over~$k$;
we perturb $x_0$ so that $u_\theta x + v_\theta y + w_\theta z$ does not vanish
at (points in the support of)~$x_0$.
Given $P \in X(K)$,
choose a linear form $u_0 x + v_0 y + w_0 z \in K[x,y,z]$ 
not vanishing at~$P$ or (points in the support of)~$x_0$.
If $P \in X(K)$ does not lie on any bitangent,
then evaluating the rational function
\[ 
	\frac{u_\theta x + v_\theta y + w_\theta z}{u_0 x + v_0 y + w_0 z}
\]
at $P$ and $x_0$ give elements of~$L_K^\times$;
their ratio mapped to $L_K^{\times}/L_K^{\times 2} K^\times$ is $\cc_K(P)$.
If $P$ does lie on a bitangent, then use Remark~\ref{R:no moving necessary},
which applies since the~$\beta_l$ are disjoint.

Now, to compute $\im C_v$,
attempt to use Remark~\ref{R:image of random points},
in which for $k_v=\Q_p$, \eqref{E:J(k_v) mod n} becomes
\begin{equation}
\label{E:im gamma_v}
\# \frac{J(k_v)}{2J(k_v)} =
 \begin{cases}
 \#J[2](k_v), &\textup{if $v \ne 2$}\\
 8\#J[2](k_v), &\textup{if $v=2$.}\\
 \end{cases}
\end{equation}
If the approach of Remark~\ref{R:image of random points} fails, 
use the general approach of 
Section \ref{S:non-arch local image} or~\ref{S:arch local image}.

\subsection{Comparing $\Selfake{\alpha}(J)$ with $\Sel{2}(J)$}

Let notation be as in Section~\ref{S:g3 compute},
where we computed a group $S'_\calT$ containing $\Selfake{\alpha}(J)$,
which Section~\ref{S:selmer group comparison} relates to~$\Sel{2}(J)$.

We review here the ways in which $S'_\calT$, $\Selfake{\alpha}(J)$,
and~$\Sel{2}(J)$ can differ:
\begin{enumerate}[\upshape (i)]
 \item $S'_\calT$ may be larger than $\Selfake{\alpha}(J)$
 \item $\Selfake{\alpha}(J)$ may fail to be contained in 
$\alpha(H^1(k,J[2]))$ (both are subgroups of $H^1(k,E^\vee)$).
 \item If $\Selfake{\alpha}(J) \subseteq \alpha(H^1(k,J[2]))$,
then $\alpha$ induces a map $\Sel{2}(J)\to \Selfake{\alpha}(J)$
(see Theorem~\ref{T:Selseq}),
but it may fail to be injective and/or surjective.
\end{enumerate}

And here are ways to detect or avoid these differences:
\begin{enumerate}[\upshape (i)]
\item
Theorem~\ref{T:finite description of Selmer}\eqref{I:good T} 
shows how to choose~$\calT$
to ensure that $\Selfake{\alpha}(J) = S_\calT$,
in which case both equal the group~$S_\calT'$ sandwiched between them.
But the amount of computation involved in implementing that strategy 
is probably prohibitive,
so we may prefer to use the observation of Remark~\ref{R:smaller T}.
\item
Lemma~\ref{L:Sha1R}\eqref{I:if Sha=0} 
shows that the containment 
$\Selfake{\alpha}(J) \subseteq \alpha(H^1(k,J[2]))$
follows if $\Sha^1(k,R^\vee)=0$,
and Proposition~\ref{P:sha1 decomp} gives an upper bound
on~$\Sha^1(k,R^\vee)$ that is often~$0$.
More precisely,
this upper bound is~$0$ for $1103$ of the subgroup classes of~$\Sp_6(\F_2)$.
Even if we are unlucky enough to have $G$ 
in the remaining $266$ subgroup classes,
there is hope since we find that 
there is a $G$-orbit in $\Sigma$ of size at most~$7$,
so the ideas of Appendix~\ref{S:SelDirect} are probably practical; and
$J[2]$ is reducible, 
so there are nontrivial isogenies available for descent computations.
\item 
In order to estimate the kernels and cokernels of 
$\Sel{2}(J)\to \Selfake{\alpha}(J)$,
we use Theorem~\ref{T:Selseq}. 
Assuming that we succeeded in computing global and local Galois groups
as in Section~\ref{S:g3 galois}, 
we have for each~$v\in\calT$ an explicit description of
\[\xymatrix{
0\ar[r]&
   J[2](k)  \ar[r]\ar@{^(->}[d]&
   E^\vee(k)\ar[r]^q\ar@{^(->}[d]&
   R^\vee(k)\ar@{^(->}[d]\\
0\ar[r]&
   J[2](k_v)  \ar[r]&
   E^\vee(k_v)\ar[r]^q&
   R^\vee(k_v).\\
}\]
Let $\kappa_v'$ be the map $\calK = \ker \alpha \to \ker \alpha_v$ 
implicit in~\eqref{E:kappa},
which can be identified with
\[
	\frac{R^\vee(k)}{q E^\vee(k)} \to \frac{R^\vee(k_v)}{q E^\vee(k_v)},
\]
and hence computed.
The map in~\eqref{E:kappa} we are really interested in is the map
\[
	\kappa_v \colon \calK = \ker \alpha \to 
	W_v \colonequals  \frac{\ker \alpha_v}{\ker \alpha_v \intersect \im \gamma_v}.
\]
We computed $\# J(k_v)/2J(k_v)$ and $\# \im C_v$;
dividing gives the size of 
$\ker C_v \isom \gamma_v(\ker C_v) = \ker \alpha_v \intersect\im \gamma_v$,
so in some cases we may have sufficient information to deduce~$\kappa_v$
and piece together $\kappa=\prod_v\kappa_v$ and its kernel and cokernel. 
Even if we can determine $\# \coker \kappa$, however,
it is only an upper bound on 
$\#\coker(\Sel{2}(J)\to\Selfake{\alpha}(J))$
(see Theorem~\ref{T:Selseq}).
\end{enumerate}

\subsection{Descent on the curve}
\label{S:curve descent}

Let $X$ be a smooth plane quartic over a global field~$k$
of characteristic not~$2$.
The fake descent setup described in Section~\ref{S:bitangents}
can be used also to perform a descent on the curve 
as described in Section~\ref{S:Sel for X}. 
Let $u_\theta,v_\theta,w_\theta \in L$ 
be as in Section~\ref{S:local computations},
and choose $f$ to be $(u_\theta x+v_\theta y+w_\theta z)/\ell$
for some linear form~$\ell \in k[x,y,z]$.

To apply Remark~\ref{R:no moving necessary for curve},
we need to find the $c$ and~$r$ of Lemma~\ref{L:coset of kernel of norm}.
By interpolation, find a degree~$14$ form $r_{14} \in k[x,y,z]$
whose zero divisor on~$X$ is $\sum_{P \in \fes} \beta_P$;
such a form is unique modulo~$g$ and up to scalar multiple.
Thus there exists $c \in k^\times$ such that 
\begin{equation}
  \label{E:r_14}
	N_{L[x,y,z]/k[x,y,z]} (u_\theta x+v_\theta y+w_\theta z) 
	\equiv c r_{14}(x,y,z)^2 \pmod{g(x,y,z)}.
\end{equation}
Comparing coefficients lets us compute~$c$.
Then $c$ and $r\colonequals r_{14}/\ell^{14}$ are as in 
Lemma~\ref{L:coset of kernel of norm}.

We can now evaluate $C_{f,v}(Q)$ for any~$Q \in X(k_v)$,
by using~$f$; in fact, it suffices to evaluate
$u_\theta x+v_\theta y+w_\theta z$
at any triple of homogeneous coordinates representing~$Q$,
since the value of $\ell$ on this triple will be in $k_v^\times$.

\begin{lemma} 
\label{L:good reduction for descent on curve}
Assume that the $u_\theta,v_\theta,w_\theta$ above are in $\calO_L$.
Let $\calS\subset\Omega_k$ be a finite set of places 
containing the archimedean places, 
the places of residue characteristic~$2$,
the places of bad reduction of the model $g(x,y,z)=0$ of~$X$,
and places lying under $\calO_L$-primes
dividing the $\calO_L$-ideal $(u_\theta,v_\theta,w_\theta)$.
Then $\im C_f \subset \widetilde{L(2,\calS)}$.
\end{lemma}

\begin{proof}
Let $v\in\Omega_k\setminus\calS$.  
Then $L_v$ is a product of local fields,
and we let $\calO_{L,v}$ be the product of their valuation subrings.
Suppose that $(x:y:z)\in X(k_v)$. 
Without loss of generality, $x,y,z\in\calO_v$ and $(x,y,z)\calO_v=\calO_v$.
The point $(u_\theta:v_\theta:w_\theta) \in \PPdual^2(L)$ 
extends to a point in $\PPdual^2(\calO_{L,v})$.
Given $Q \in \fes(k_{v,u}) = \fes(k_s)$,
we can specialize $u_\theta$ to an element~$u_Q$ in the valuation
ring of~$k_{v,u}$, and define $v_Q$ and~$w_Q$ similarly;
our hypothesis on $(u_\theta,v_\theta,w_\theta)$
implies that $\min(v(u_Q),v(v_Q),v(w_Q))=0$.
On the reduction of~$X$ at~$v$,
the $28$~bitangents are distinct,
so the reduction of $(x:y:z) \in X(k_v)$
lies on at most one bitangent, so
\[
	v(u_Q x+v_Q y+ w_Qz)>0
\]
for at most one~$Q$.
On the other hand,
\eqref{E:r_14} and our hypothesis on $(u_\theta,v_\theta,w_\theta)$
imply that $2 \mid v(c)$,
so
\[
	2\mid v(N_{L_v/k_v}(u_\theta x + v_\theta y + w_\theta z))
	= v \left(\prod_Q (u_Q x + v_Q y + w_Q z) \right).
\]
Combining the previous two sentences shows that $2\mid v(u_Q x + v_Q y + w_Q z)$ 
for all~$Q$, and hence $u_Q x + v_Q y + w_Q z$ is a square in~$L\tensor k_{v,u}$.
\end{proof}

By breaking up $X(k_v)$ as in Section~\ref{S:non-arch local image},
we can compute $C_{f,v}(X(k_v))$ for any~$v$.
(Note that in contrast to Section~\ref{S:non-arch local image}, 
we do not need to work with extensions of~$k_v$.)
For a finite set $\calT\subset \Omega_k$, we compute
\[
S'_\calT\colonequals \{\delta\in \widetilde{L(2,\calS)}: 
	N(\delta)\in ck^{\times2} \textup{ and }
	\delta_v \in C_{f,v}(X(k_v)) \textup{ for all } v\in\calT\},
\]
which contains $\Selfake{f}(X)$.
In particular, if $S'_\calT$ is empty, then $X$ has no $k$-rational points.

On the other hand, the following may be useful in proving the existence
of local points:
\begin{lemma}
\label{L:Weil plus Hensel}
Let $k_v$ be a nonarchimedean local field with valuation ring~$\calO_v$
and residue field~$\F_v$ of size at least~$37$.
If $X_{k_v}$ is a smooth plane quartic,
and $X_{\F_v}$ is geometrically irreducible,
then $X(k_v)$ is nonempty.
\end{lemma}

\begin{proof}
Removing the singularities from $X_{\F_v}$ 
yields a smooth curve of genus~$g$ with at most $6-2g$ punctures,
for some $g \in \{0,1,2,3\}$.
In all four cases, the Hasse-Weil bound implies the
existence of a smooth $\F_v$-point.
By Hensel's lemma, it lifts to an $\calO_v$-point,
which gives a $k_v$-point.
\end{proof}

\subsection{Examples}
\label{S:examples}

Denis Simon 
kindly supplied us with a list of smooth plane quartics with small
integer coefficients and small discriminant.
These serve as test cases for the methods
outlined in Section~\ref{S:g3 compute}.
In the first three examples, we use fake descent to determine 
the structure of~$J(\Q)$, which in the first two examples
lets us determine~$X(\Q)$.
In the fourth example, we use descent on the curve to prove 
that $X(\Q)=\emptyset$.

\subsubsection{A genus $3$ curve with $J(\Q) \isom \Z/51\Z$}
\label{S:disc 4727 example}
Let $X$ be the curve in~$\PP^2_\Q$ defined by
\[
	x^3y - x^2y^2 - x^2z^2 - xy^2z + xz^3 + y^3z=0.
\]
(This is isomorphic to the curve of smallest discriminant in Simon's list.)

\medskip 
\noindent\emph{Steps 1 and 2.}
The algebra $L=\Q[t]/(h(t))$ turns out to be 
a degree~$28$ number field over~$\Q$.
It follows that $\calG_\Q$ acts transitively on the bitangents.
We compute $\calO_L$ and find that $\Disc\calO_L=2^{42}\cdot29^6\cdot163^6$.

\noindent\emph{Step 3.} We find that $I_{27}=4727=29\cdot163$.
By Remark~\ref{R:primes of multiplicity 1},
we may take $\calS=\{2,\infty\}$.

\noindent\emph{Step 4.} 
The class group of~$\calO_L$ is generated by primes of norm below the
Minkowski bound, which is $36\,984\,868$, 
remarkably small for a degree-$28$ number field. 
We can prove unconditionally that the class group of~$\calO_L$ is trivial,
and we can find explicit generators of~$L(2,\calS)$.
We find $L(2,\calS) \isom (\Z/2\Z)^{17}$ 
and $\widetilde{L(2,\calS)} \isom (\Z/2\Z)^{15}$.

\noindent\emph{Step 5.} 
By computing the Frobenius action at $3$ and $5$, 
we find that $G$ is either the full group~$\Sp_6(\F_2)$
or the unique index-$36$ subgroup up to conjugacy. 
The larger group acts doubly transitively on the bitangents, 
whereas the smaller does not. 
If we factor $h(t)$ over~$L$,
we find factors of degrees $1,12,15$, so $G$ is the smaller group. 
With this information we can check that
\begin{equation}\label{E:ex4247 JER}
\xymatrix@R=1em{
0\ar[r]& J[2](\Q)\ar[r]& E^\vee(\Q)\ar[r]& R^\vee(\Q)\\
&\ar@{=}[u] 0 &\ar@{=}[u] 0 &\ar@{=}[u] \F_2
}
\end{equation}

\noindent\emph{Step 6.} 
Taking $\calT=\emptyset$ (i.e., doing no local computations with~$X$)
yields the upper bound $S_\emptyset' \isom (\Z/2\Z)^{13}$ 
on~$\Selfake{\alpha}(J)$.
To obtain a better bound, we next try $\calT=\{2\}$; 
it will turn out that this suffices.

\noindent\emph{Step 6a.} 
We check that $h$ is irreducible over~$\Q_2$,
so $D_2$ acts transitively on~$\fes$.
This, together with the general constraints on a decomposition subgroup
at~$2$, leaves only two candidates for~$D_2$ up to conjugacy:
a group~$G_{56}$ of order~$56$ and a group~$G_{168}$ of order~$168$.
We can distinguish these by the number of points fixed
by their one-point stabilizers:
$G_{56}$ leaves $4$~points fixed and $G_{168}$ leaves $1$~point fixed.
We find that $h$ has $4$~roots in $\Q_2(\theta)$, so $D_2=G_{56}$.
Actually, the distinction is unimportant, because for either group we have
\begin{equation}\label{E:ex4247 JER 2}
\xymatrix@R=1em{
0\ar[r]& J[2](\Q_2)\ar[r]& E^\vee(\Q_2)\ar[r]^{q} & R^\vee(\Q_2)\\
&\ar@{=}[u] 0 &\ar@{=}[u] 0 &\ar@{=}[u] \F_2.
}
\end{equation}

\noindent\emph{Step 6b.} From $J[2](\Q_2)=0$ and~\eqref{E:im gamma_v},
we deduce $\#\im \gamma_2=8$.
Embed $X_{\Q_2}$ in~$J_{\Q_2}$ using $x_0=(0:0:1) \in X(\Q_2)$;
we find points in~$X(\Q_2)$ whose images under~$C_2$
generate a group of size~$8$.
Thus the inequality $\#\im C_2 \le \#\im \gamma_2$
is an equality; i.e., $\#\im C_2=8$.
By Lemma~\ref{L:computing W_v} and its proof, we obtain
\begin{equation}
  \label{E:W_2 in example 1}
	W_2 \isom \ker \alpha_2 \isom \frac{R^\vee(\Q_2)}{q E^\vee(\Q_2)} 
	\isom \Z/2\Z,
\end{equation}
by~\eqref{E:ex4247 JER 2}.

\noindent\emph{Step 7.} 
Given that we have explicit generators for~$\widetilde{L(2,\calS)}$,
we can compute the map
\[
	\widetilde{L(2,\calS)}
	\to 
	\frac{(L\tensor \Q_2)^\times}{(L\tensor \Q_2)^{\times2}\Q_2^\times}.
\]
In Step 6b, we computed $\im C_2$, 
so we can compute $S'_\calT$ for~$\calT=\{2\}$.
We find $S'_\calT=0$.
Thus~$\Selfake{\alpha}(J)=0$.

By~\eqref{E:ex4247 JER} and Corollary~\ref{C:bounding Sel to Selfake},
$\Sel{2}(J)$ is either $0$ or~$\Z/2\Z$.
We can decide which by computing $\ker \kappa$ in Theorem~\ref{T:Selseq}.
By~\eqref{E:W_2 in example 1},
the map $\calK \to W_2$ is 
$R^\vee(\Q)/qE^\vee(\Q) \to R^\vee(\Q_2)/qE^\vee(\Q_2)$,
which is an isomorphism,
so $\kappa$ is injective;
thus $\Sel{2}(J)=0$ by Theorem~\ref{T:Selseq}.

\begin{remark}
The situation can be reinterpreted as follows:
the kernel~$\calK$ of the map $\alpha\colon H^1(\Q,J[2])\to H^1(\Q,E^\vee)$
has a non-trivial element~$\delta$,
but its restriction $\delta_2 \in H^1(\Q_2,J[2])$ 
is not in~$\im \gamma_2$.
(This is because $0 \ne \delta_2 \in \ker \alpha_2$,
while $\im \gamma_2 \intersect \ker \alpha_2 = 0$: see Step 6b above.)
In classical language, $\delta$ corresponds to a $2$-covering of~$J$
that has no $\Q_2$-point, so $\delta \notin \Sel{2}(J)$.
\end{remark}

\begin{proposition}
If $X$ is the curve
\[
	x^3y - x^2y^2 - x^2z^2 - xy^2z + xz^3 + y^3z=0
\]
in $\PP^2_\Q$, then $J(\Q)=\langle[(0:1:0)-(0:0:1)]\rangle\isom \Z/51\Z$ and
\[
	X(\Q)=\{(1 : 1 : 1), (0 : 1 : 0), (0 : 0 : 1), 
		(1 : 0 : 0), (1 : 1 : 0), (1 : 0 : 1)\}.
\]
\end{proposition}

\begin{proof}
Since $\Sel{2}(J)=0$, the group~$J(\Q)$ is finite and of odd order.
Reduction modulo~$3$ injects $J(\Q)_\tors$ into~$J(\F_3)$,
which, according to {\tt Magma}, is of order~$51$.
(We use~$J(\F_3)$ to denote the group of $\F_3$-points on the reduction.)
On the other hand, 
{\tt Magma} shows that $[(0:1:0)-(0:0:1)]$ is of exact order~$51$ in~$J(\Q)$.
Thus $J(\Q) \isom \Z/51\Z$.

Finally, we determine~$X(\Q)$.
The Abel-Jacobi map $X\to J$ given by $P\mapsto [P-(0:0:1)]$ 
injects $X(\Q)$ into~$J(\Q)$.
{\tt Magma} determines which points in~$J(\Q)$
can be represented as~$[P-(0:0:1)]$,
and the result is as stated.
\end{proof}

\subsubsection{A positive rank example}

\label{S:disc 14227 example}
Let $X$ be the curve in~$\PP^2_\Q$ defined by
\[
	x^2y^2 - xy^3 - x^3z - 2x^2z^2 + y^2z^2 - xz^3 + yz^3=0.
\]
The group $J(\Q)_{\tors}$ injects into~$J(\F_3)$,
and its odd part injects into~$J(\F_2)$
(the curve has good reduction at $2$ and~$3$).
We compute $J(\F_2)\isom \Z/71\Z$ and $J(\F_3)\isom \Z/85\Z$,
so $J(\Q)_{\tors}=0$.
On the other hand, the divisor class
\[
	G\colonequals [(0:1:-1)+(0:0:1)-2(0:1:0)]
\]
is nonzero since $X$ is not hyperelliptic.
Thus $\rk J(\Q) \ge 1$.

\medskip 
\noindent\emph{Steps 1 and 2.}
The algebra~$L$ turns out to be a degree~$28$ number field.
We find that $\Disc \calO_L = 2^{42}\cdot41^6\cdot347^6$.

\noindent\emph{Step 3.}
We have $I_{27}=41\cdot347$.
By Remark~\ref{R:primes of multiplicity 1},
we may take $\calS=\{2,\infty\}$.

\noindent\emph{Step 4.} 
The Minkowski bound for $\calO_L$ is $1\,008\,340\,641$.
The truly dedicated enthusiast 
could probably verify unconditionally 
that the class group of~$\calO_L$ is trivial.
We verified this only conditionally 
on the Generalized Riemann Hypothesis for~$L$.
Subject to this, 
we find explicit generators of $L(2,\calS)\isom (\Z/2\Z)^{17}$
and $\widetilde{L(2,\calS)} \isom (\Z/2\Z)^{15}$.

\noindent\emph{Step 5.} 
By computing the Frobenius action at $5$ and~$7$,
we find that $G \isom \Sp_6(\F_2)$.
Thus $J[2](\Q)=E^\vee(\Q)=R^\vee(\Q)=0$.

\noindent\emph{Step 6.} 
It will turn out that taking $\calT=\{2\}$
is enough to obtain the upper bound $S'_{\{2\}} \isom \Z/2\Z$
on~$\Selfake{\alpha}(J)$.

\noindent\emph{Step 6a,b.} 
We check that $h$ is irreducible over~$\Q_2$,
so $D_2$ acts transitively on the bitangents. 
This leaves $6$~possibilities for~$D_2$ up to $\Sp_6(\F_2)$-conjugacy,
but for all of them we have $J[2](\Q_2)=E^\vee(\Q_2)=0$.
Following Remark~\ref{R:image of random points},
we obtain $\#\im C_2 \le \# J(\Q_2)/2J(\Q_2) = 8$ (see~\eqref{E:im gamma_v}),
but we also find enough points in~$X(\Q_2)$ to show that $\#\im C_2 \ge 8$,
so $\#\im C_2=8$.

\noindent\emph{Step 7.} 
We can compute the map
\[
	\widetilde{L(2,\calS)}\to 
	\frac{(L\tensor \Q_2)^\times}{(L\tensor \Q_2)^{\times2}\Q_2^\times}
\]
explicitly, and we have $\im C_2$; from this we compute that $S'_\calT=\Z/2\Z$.

Since $R^\vee(\Q)=0$, Corollary~\ref{C:bounding Sel to Selfake}
implies that $\Sel{2}(J) \to \Selfake{\alpha}(J)$ is injective.
Thus
\[
	\frac{J(\Q)}{2J(\Q)} \hookrightarrow 
	\Sel{2}(J) \hookrightarrow 
	\Selfake{\alpha}(J) \subset 
	S'_\calT = \Z/2\Z,
\]
so $\rk J(\Q) \le 1$.
Combining this with the earlier information yields $J(\Q) \isom \Z$.

\begin{proposition} 
Let $X$ be the curve
\[ 
	x^2y^2 - xy^3 - x^3z - 2x^2z^2 + y^2z^2 - xz^3 + yz^3 = 0
\]
in $\PP^2_\Q$.
Assume the Generalized Riemann Hypothesis.
Then $J(\Q)\isom \Z$ and
\[ 
\begin{split}
X(\Q)=\{
& (1:1:0), (-1 : 0 : 1), (0 : -1 : 1), (0 : 1 : 0),\\
& (1 : 1 : -1), (0 : 0 : 1), (1 : 0 : 0), (1 : 4 : -3)
\}.
\end{split}
\]
\end{proposition}

\begin{proof}
We already proved that $J(\Q) \isom \Z$.
To determine~$X(\Q)$
we use an explicit version of the method of Chabauty and Coleman 
over~$\Q_3$ (see \cite{McCallum-Poonen2012} for a survey).

We check that $\#X(\F_3)=7$,
and that the set of $8$~points in~$X(\Q)$ listed above
surjects onto~$X(\F_3)$,
with $P_0\colonequals (1:1:0)$ and $P_1\colonequals (1:4:-3)$ reducing to the same point.
Since $\rk J(\Q)=1$ and $\dim J = 3$,
there is a $2$-dimensional subspace $V \subset H^0(X_{\Q_3},\Omega^1)$
such that $\int_P^{P'} \omega = 0$ for all $P,P' \in X(\Q)$
and $\omega \in V$.
Since $[P_1-P_0]$ is nonzero in~$J(\Q)$, it is of infinite order,
so $V = \{ \omega : \int_{P_0}^{P_1} \omega = 0 \}$.
We compute this integral (to some $3$-adic precision) 
for~$\omega$ in a basis for~$H^0(X_{\Q_3},\Omega^1)$
by evaluating the integral as a power series
in a uniformizing parameter~$t$ at~$P_0$;
then linear algebra produces a basis for~$V$.
Explicitly, if we identify each $\omega \in H^0(X_{\Q_3},\Omega^1)$
with a linear form $ux+vy+wz$,
a basis for~$V$ is given by
$\omega_1$ corresponding to $(21262+O(3^{10}))x-y$ 
and $\omega_2$ corresponding to $(1302+O(3^{10}))x-z$.

For each point of~$X(\F_3)$, we need to find the rational points
in the corresponding residue class in~$X(\Q_3)$.
For each point $Q \in X(\F_3)$ other than the reduction of~$P_0$ (and~$P_1$),
the mod~$3$ reduction of one of $\omega_1,\omega_2$ is nonvanishing at~$Q$,
so by Proposition~6.3 in~\cite{Stoll2006-chabauty},
there is at most one rational point reducing to~$Q$,
and we already know one.
The rational points~$P$ in the residue class containing $P_0$ and~$P_1$
satisfy $\int_{P_0}^P \omega_1 = \int_{P_0}^P \omega_2 = 0$;
these give two power series equations over~$\Q_3$ to be solved
for $t \in 3\Z_3$.
We calculate that each power series has three zeros in~$3 \Z_3$,
but the intersection is of size at most~$2$,
so $P_0$ and~$P_1$ are the only rational points in this residue class.
\end{proof}

\subsubsection{A modular curve of level $13$}

Let $X_\mathrm{split}(13)$ (resp.\ $X_\mathrm{nonsplit}(13)$)
be the modular curve of level~$13$ over~$\Q$
corresponding to the normalizer of a split (resp.\ nonsplit) Cartan subgroup
of~$\SL_2(\Z/13\Z)$.
These are non-hyperelliptic curves of genus~$3$,
and it turns out~\cites{Baran-preprint,Baran-preprint2} 
that both are isomorphic to the smooth plane quartic curve
\begin{equation}
\label{E:Xsplit13}
	X\colon x^3y + x^3z - 2x^2y^2 - x^2yz + xy^3 - xy^2z 
		+ 2xyz^2 - xz^3 - 2y^2z^2 + 3yz^3 = 0.
\end{equation}
This curve has at least the following $7$~rational points.
\[
\{ (0 : 1 : 0), (0 : 0 : 1), (-1 : 0 : 1), (1 : 0 : 0), (1 : 1 : 0), (0 : 3 : 
2), (1 : 0 : 1) \}\]
We compute $J(\F_3) \isom \Z/91\Z$ and $J(\F_7) \isom \Z/659\Z$.
Since $J(\Q)_{\tors}$ injects into both groups,
$J(\Q)_{\tors}=0$.
This, with the fact that $\#X(\Q) \ge 2$, implies that $\rk J(\Q) > 0$.
One can verify that the divisor classes
\[[ (0:1:0)-(1:0:0) ], [(0:0:1)-(1:0:0)], [(-1:0:1)-(1:0:0)] \in J(\Q)\]
generate a group containing all differences of the points listed above.
Furthermore, the image in $J(\F_3)\times J(\F_5)\times J(\F_{43})$ is isomorphic to
\[\Z/(7\cdot 13\cdot 29\cdot 97)\Z\times \Z/13\Z\times \Z/13\Z,\]
so $\rk J(\Q)\geq 3$.

We apply the procedure in Section~\ref{S:g3 compute} 
to compute an upper bound on~$\rk J(\Q)$.

\medskip 
\noindent\emph{Steps 1 and 2.}
The algebra $L$ is a degree~$28$ number field
with $\Disc \calO_L = 2^{42} \cdot 13^{24}$.

\noindent\emph{Step 3.}
We have $I_{27}=13^6$.
The closed subscheme of~$\PP^2_{\Z_{13}}$ defined by~\eqref{E:Xsplit13}
is regular, and its special fiber is
a geometrically integral curve of genus~$0$, so $c_{13}(J)=1$.
Therefore we can take $\calS = \{2, \infty\}$.

\noindent\emph{Step 4.} 
The Minkowski bound for $\calO_L$ is $8\,158\,071\,456$.
The truly dedicated enthusiast 
could probably verify unconditionally 
that the class group of~$\calO_L$ is trivial.
We verified this only conditionally 
on the Generalized Riemann Hypothesis for~$L$.
Subject to this, 
we find explicit generators of $L(2,\calS) \isom (\Z/2\Z)^{17}$
and $\widetilde{L(2,\calS)} \isom (\Z/2\Z)^{15}$.

\noindent\emph{Step 5.} 
Since $L$ is a field, $G$ acts transitively on the bitangents.
There are $18$~subgroups of~$\Sp_6(\F_2)$
up to conjugacy with that property. 
The Fieker-Kl\"uners implementation in {\tt Magma} for finding Galois groups 
yields $\#G=504$.
This determines~$G$ uniquely up to conjugacy.

This information allows us to compute:
\begin{equation}\label{E:exX13 JER}
\xymatrix@R=1em{
0\ar[r]& J[2](\Q)\ar[r]& E^\vee(\Q)\ar[r]& R^\vee(\Q)\\
&\ar@{=}[u] 0 &\ar@{=}[u] 0 &\ar@{=}[u] \F_2^2.
}
\end{equation}

\noindent\emph{Steps 6 and 7.}
It will turn out that taking $\calT=\{2\}$ is enough to obtain the upper bound
$S'_{\{2\}} \isom \Z/2\Z$ on $\Selfake{\alpha}(J)$.
The fact that $L$ has only one prime above~$2$
shows that $D_2$ acts transitively. 
This together with the constraints on a decomposition group at~$2$
determines $D_2$ uniquely up to conjugacy in~$\Sp_6(\F_2)$.
We obtain $\#D_2=56$ and
\begin{equation}\label{E:exX13 JER 2}
\xymatrix@R=1em{
0\ar[r]& J[2](\Q_2)\ar[r]& E^\vee(\Q_2)\ar[r]& R^\vee(\Q_2)\\
&\ar@{=}[u] 0 &\ar@{=}[u] 0 &\ar@{=}[u] \F_2^2.
}
\end{equation}
Following Remark~\ref{R:image of random points},
we obtain $\#\im C_2 \le \#\im \gamma_2 = \# J(\Q_2)/2J(\Q_2) = 2^3$ 
(see~\eqref{E:im gamma_v}).
Further computation shows that $S'_\emptyset=(\ZZ/2\ZZ)^{13}$ 
and that the homomorphism
\[
	S'_\emptyset \to 
	\frac{(L\tensor \Q_2)^\times}{(L\tensor\Q_2)^{\times 2}\Q_2^\times}
\]
is injective. 
It follows that $\#S'_\calT\le \#\im C_2 \le 2^3$.
Thus $\#\Selfake{\alpha}(J) \le 2^3$.
By Corollary~\ref{C:bounding Sel to Selfake} and \eqref{E:exX13 JER}, 
this implies $\#\Sel{2}(J) \le 2^5$,
so $\rk J(\Q)\leq 5$. 

We now present two approaches to improve this to $\rk J(\Q) \le 3$.
The first is to use the isomorphism
$\End J \isom \Z[\zeta_7+\zeta_7^{-1}]$
of~\cite{Baran-preprint2}*{Proposition~2.4} 
to obtain that $\rk J(\Q)$ is a multiple of~$3$, 
which improves the bound to $\rk J(\Q) \le 3$.
The second is to follow Section~\ref{S:non-arch local image} 
to compute that $\#\im C_2=2$, which can be used as follows.
By the same argument as in the previous paragraph, 
$\#\im C_2=2$ implies $\#\Selfake{\alpha}(J) \le 2^1$
and $\#\Sel{2}(J) \le 2^3$.
On the other hand, $3 \leq \rk J(\Q) \le \dim_{\F_3} \Sel{2}(J)$,
so equality holds everywhere in this paragraph.
In particular, if $\Sha(J)$ is the Shafarevich-Tate group of~$J$,
then $\Sha(J)[2]=0$.
We have proved

\begin{proposition}
Let $X=X_\mathrm{split}(13) \isom X_\mathrm{nonsplit}(13)$ over $\Q$. 
Assume the Generalized Riemann Hypothesis. 
Then $J(\Q) \isom \Z^3$ and $\Sha(J)[2]=0$.
\end{proposition}

\begin{remark}
By \eqref{E:exX13 JER 2} and Lemma~\ref{L:computing W_v},
we have $W_2=0$.
In particular, we are in the situation of Corollary~\ref{C:Nicecase}.
\end{remark}

\subsubsection{A genus $3$ curve violating the local-to-global principle}

The fake descent setup presented in Section~\ref{S:bitangents} allows us
also to compute $\Selfake{f}(X)$ for a smooth plane quartic~$X$.

\begin{proposition} 
Let $X$ be the curve in~$\PP^2_\Q$ defined by
\[
	x^4 + y^4 + x^2 y z + 2 x y z^2 - y^2 z^2 + z^4 = 0.
\]
Then $X(\R)\ne \emptyset$ and $X(\Q_p)\ne \emptyset$ for all~$p$,
but if the Generalized Riemann Hypothesis holds, then
$X(\Q)=\emptyset$.
\end{proposition}

\begin{proof}
We have $I_{27}=-2^8\cdot 5^2\cdot 1361\cdot97103$, 
so Lemma~\ref{L:Weil plus Hensel} implies $X(\Q_p)\ne  \emptyset$ 
for $p \ge 37$.
A further calculation using Hensel's lemma shows that $X(\Q_p) \ne \emptyset$
for $p<37$ too,
and that $X(\R) \ne \emptyset$.

To prove $X(\Q) = \emptyset$, we apply Section~\ref{S:curve descent} with
$\calS=\{\infty,2,5,1361,97103\}$. 
It is straightforward to find $(u_\theta,v_\theta,w_\theta)$ 
satisfying the hypotheses of Lemma~\ref{L:good reduction for descent on curve}
for~$\calS$.
We find that $\Disc(\calO_L)=2^{30}\cdot 5^{10}\cdot 1361^6\cdot 97103^6$
and that the Minkowski bound exceeds~$10^{22}$,
so determining $\Cl(\calO_{L,\calS})$ unconditionally 
is out of the question.
Conditional on the Generalized Riemann Hypothesis for~$L$,
which we assume from now on,
we find that $\Cl(\calO_{L,\calS})=\Z/2\Z$ 
and we also find $\calO_{L,\calS}^\times/\calO_{L,\calS}^{\times 2}$.
This leads to explicit generators of $L(2,\calS) \isom (\Z/2\Z)^{41}$
and $\widetilde{L(2,\calS)} \isom (\Z/2\Z)^{36}$.
We compute $c=-1$.
For $\calT=\{2,1361\}$, we find $S'_\calT=\emptyset$, so $X(\Q)=\emptyset$.
\end{proof}

\appendix
\section{Determining the $\phi$-Selmer group directly}
\label{S:SelDirect}

We can modify our approach so that it computes $\Sel{\phi}(J)$ directly,
instead of computing~$\Seltf{\alpha}(J)$
and hoping to control the difference.

The fundamental idea behind a true descent setup is to replace
the Galois module~$\widehat{J}[\widehat{\phi}]$ whose cohomology we want
by a permutation module~$(\ZZ/n\ZZ)^\fes$ whose cohomology we can compute.
The discrepancies between $\Seltf{\alpha}(J)$ and~$\Sel{\phi}(J)$
come from the non-injectivity of
$(\ZZ/n\ZZ)^\fes\to\widehat{J}[\widehat{\phi}]$. 
In this section we deal with the non-injectivity
by finding another permutation module~$(\ZZ/n\ZZ)^{\fes'}$
that surjects onto the kernel.

\subsection{Correspondences}
\label{S:correspondences}

A \defi{correspondence} $\corr{\TAU}{\fes}{\fes'}$
between finite $\calG$-sets 
is a $\calG$-homo\-mor\-phism $\Z^\fes \to \Z^{\fes'}$.
Given a $\calG$-module~$M$ and $\corr{\TAU}{\fes}{\fes'}$,
we define homomorphisms $\TAU^* \colon M^{\fes'} \to M^\fes$
and $\TAU_* \colon M^\fes \to M^{\fes'}$ as follows.
Restricting the composition pairing
\[
	\Hom_\Z(\Z^{\fes'},M) \times \Hom_\Z(\Z^\fes,\Z^{\fes'}) 
	\to \Hom_\Z(\Z^\fes,M)
\]
by setting the second argument to~$\TAU$
yields a $\calG$-homomorphism 
that with the identifications of Remark~\ref{R:Z^fes}
becomes $\TAU^* \colon M^{\fes'} \to M^\fes$.
Applying this construction to the $\Z$-dual of $\TAU$ yields~$\TAU_*$.

\begin{example}
\label{Ex:what TAU can do}
Suppose that $X$ is a nice $k$-variety.
If we apply the previous remark to the \'etale group scheme~$M$
such that $M(k_s)$ is the $\calG_k$-module $\Div X_s$,
then $M^{\fes}(k) = \Div (X \times \fes)$,
and we obtain 
$\TAU_* \colon \Div(X \times \fes) \to \Div(X \times \fes')$.
Similarly we obtain 
$\TAU_* \colon k(X \times \fes)^\times \to k(X \times \fes')^\times$.
\end{example}

\begin{lemma}
\label{L:correspondences surjecting and injecting}
Let $M$ be a finite $\calG$-module such that $nM=0$.
\begin{enumerate}[\upshape (a)]
\item\label{I:surjecting}
There exists a finite $\calG$-set~$\fes$
with a surjection $(\Z/n\Z)^\fes \surjects M$.
\item\label{I:injecting}
There exists a finite $\calG$-set~$\fes$
with an injection $M \injects \mu_n^\fes$ (here we assume $\Char k \nmid n$).
\end{enumerate}
\end{lemma}

\begin{proof}
For \eqref{I:surjecting}, take $\fes=M$: 
the identity $M \to M$ induces $\Z^M \surjects M$,
which factors through $(\Z/n\Z)^M$ since $nM=0$.
Applying \eqref{I:surjecting} to $M^\vee\colonequals \Hom_\Z(M,k_s^\times)$
and applying $\Hom_\Z(-,k_s^\times)$ yields~\eqref{I:injecting}.
\end{proof}

\begin{lemma} 
\label{L:correspondence divisible by n}
Each of the four homomorphisms
\begin{align*}
\Hom_\Z(\Z^\fes,\Z^{\fes'}) &\xrightarrow{\TAU\mapsto\TAU_*} \Hom_\Z((\Z/n\Z)^\fes,(\Z/n\Z)^{\fes'}) \\
\Hom_\Z(\Z^\fes,\Z^{\fes'}) &\xrightarrow{\TAU\mapsto\TAU^*}
 \Hom_\Z((\Z/n\Z)^{\fes'},(\Z/n\Z)^\fes) \\
\Hom_\Z(\Z^\fes,\Z^{\fes'}) &\xrightarrow{\TAU\mapsto\TAU_*} \Hom_\Z(\mu_n^\fes,\mu_n^{\fes'}) \\
\Hom_\Z(\Z^\fes,\Z^{\fes'}) &\xrightarrow{\TAU\mapsto\TAU^*} \Hom_\Z(\mu_n^{\fes'},\mu_n^\fes)
\end{align*}
is surjective with kernel $n \Hom_\Z(\Z^\fes,\Z^{\fes'})$
(in the last two we assume $\Char k \nmid n$).
\end{lemma}

\begin{proof}
The statements do not involve the $\calG$-action,
so each reduces to an easy statement about abelian groups.
\end{proof}

\subsection{Determining the Selmer group} \label{S:Det Selmer}

Set $\widetilde{\G_m^\fes} \colonequals  \G_m^\fes$
for a true descent setup, and $\widetilde{\G_m^\fes} \colonequals  \G_m^\fes/\mu_n$ for a
fake descent setup. 

Recall that $R= \ker(\alpha^\vee \colon E \surjects \widehat{J}[\widehat{\phi}])$.
Fix a finite \'etale scheme $\fes'=\Spec L'$ 
with a surjection $(\Z/n\Z)^{\fes'} \surjects R$
(one such choice is $\fes'=R$). 
By definition of~$R$, we have an exact sequence
\begin{equation}
  \label{E:tau*}
(\Z/n\Z)^{\fes'} \To E \stackrel{\alpha^\vee}\To \widehat{J}[\widehat{\phi}] \To 0
\end{equation}
and its dual
\begin{equation}
\label{E:dual of tau*}
    0 \To A[\phi] \To E^\vee \To \mu_n^{\fes'}
\end{equation}
in which the image of the last map is~$R^\vee$.

\begin{remark} \label{R:critSha1}
  Define $Q$ and~$q'$ by the exact sequence
  \begin{equation}
    \label{E:Q and q'}
	0 \To R^\vee \To \mu_n^{\fes'} \stackrel{q'}{\To} Q \To 0 . 
  \end{equation}
  If $Q(k) = q'\bigl(\mu_n(L')\bigr)$
  and $\Sha^1(k, \mu_n^{\fes'}) = 0$
  (the latter holds if $n$ is prime, by Lemma~\ref{L:Sha of mu_p}),
  then $\Sha^1(k, R^\vee) = 0$, so
  the first assumption in Corollary~\ref{C:Nicecase} is satisfied.
  This gives a simple way of showing that $\Sha^1(k, R^\vee)$ vanishes,
  but the criterion can be weaker than that coming from Section~\ref{S:sha1}.
\end{remark}

By Lemma~\ref{L:correspondence divisible by n}, the composition 
\begin{equation}
\label{E:def of TAU}
	(\Z/n\Z)^{\fes'} \surjects R \injects E \injects (\Z/n\Z)^\fes
\end{equation}
is $\TAU^*$ for some correspondence
$\corr{\TAU}{\fes}{\fes'}$.
We use~$\TAU_*$ to denote any of several homomorphisms induced by~$\TAU$;
the context will make the meaning clear.
In the fake case,
$\TAU_* \colon \mu_n^\fes \to \mu_n^{\fes'}$
kills the diagonal~$\mu_n$,
because taking duals in~\eqref{E:def of TAU} 
shows that $\TAU_*$ factors through $E^\vee = \mu_n^\fes/\mu_n$.
Thus $\TAU_* \colon \G_m^\fes\to\G_m^{\fes'}$
induces $\TAU_* \colon \widetilde{\G_m^\fes}\to\G_m^{\fes'}$
in both cases.

Let $U$ be the image of 
$(\TAU_*, n) \colon \G_m^\fes \to \G_m^{\fes'} \times \G_m^\fes$.
The last map in~\eqref{E:dual of tau*} is induced by
$\TAU_* \colon \mu_n^\fes \to \mu_n^{\fes'}$,
so we have a commutative diagram with exact rows and columns
\begin{equation}
\begin{split}
\label{E:two rows}  
 \xymatrix{ 
     		       &  & & 0 \ar[d] \\
     		       & 0 \ar[d] & & R^\vee \ar[d] \\
              0 \ar[r] & A[\phi] \ar[d]_{\alpha} \ar[r]
                       & \widetilde{\G_m^\fes} \ar@{=}[d] \ar[r]^{(\TAU_*,n)}
                       & U \ar[d]^{\pr_2} \ar[r]
                       & 0\phantom{,} \\
              0 \ar[r] & E^\vee \ar[r] \ar[d]_q
                       & \widetilde{\G_m^\fes} \ar[r]^{n}
                       & \G_m^\fes \ar[r] \ar[d]
                       & 0, \\
                       & R^\vee \ar[d] & & 0 \\
                       & 0 \\
            }
\end{split}
\end{equation}
in which $\ker \pr_2$ is identified as~$R^\vee$ by the snake lemma.
Taking cohomology, we obtain
\begin{equation}
\begin{split}
 \label{E:UH diagram}
   \xymatrix{ 
		& & R^\vee(k) \ar@{^{(}->}[d] \\
		& (\widetilde{\G_m^\fes})(k) \ar@{=}[d] \ar[r]^-{(\TAU_*,n)}
                & U(k) \ar[d]^{\pr_2} \ar[r]
                & H^1(k, A[\phi]) \ar[d]_{\alpha} \ar[r]
                & H^1(k, \widetilde{\G_m^\fes}) \phantom{.} \ar@{=}[d] \\
             E^\vee(k) \ar[r] & (\widetilde{\G_m^\fes})(k) \ar[r]^-n
                & L^\times \ar[r]
                & H^1(k, E^\vee) \ar[r]
                & H^1(k, \widetilde{\G_m^\fes}) .
            }
\end{split}
\end{equation}

\begin{lemma} \label{L:UasKernel}
  There is a finite \'etale $k$-scheme~$\fes''$ and a correspondence
  $\corr{\TAU'}{\fes'}{\fes''}$ such that
  \[ \mu_n^\fes \stackrel{\TAU_*}{\To} \mu_n^{\fes'} \stackrel{\TAU'_*}{\To} \mu_n^{\fes''} \]
  is exact. There is another correspondence $\corr{\TAU''}{\fes}{\fes''}$
  such that $\TAU' \circ \TAU = n \TAU''$. 
Then
  \[ U = \ker\Bigl(\G_m^{\fes'} \times \G_m^{\fes} \To \G_m^{\fes'} \times \G_m^{\fes''}, \quad
                   (\ell', \ell) \longmapsto
                     \bigl(\frac{{\ell'}^n}{\TAU_*(\ell)},
                           \frac{\TAU'_*(\ell')}{\TAU''_*(\ell)}\bigr) \Bigr) .
  \]
\end{lemma}

\begin{proof}
Lemma~\ref{L:correspondences surjecting and injecting}\eqref{I:injecting}
applied to the cokernel of~$\TAU_*$ yields $\fes''$ and~$\TAU'$.
By Lemma~\ref{L:correspondence divisible by n},
$\TAU' \circ \TAU$ is $n$ times some~$\TAU''$.
Direct computation shows that $U$ is contained in the kernel.
To see the other inclusion, let $(\ell', \ell)$ be in the kernel. 
Let $\lambda \in \G_m^\fes(k_s)$ be such that $\lambda^n = \ell$, 
and let $\zeta' = \ell'/\TAU_*(\lambda)$. 
Then 
${\zeta'}^n = {\ell'}^n/\TAU_*(\lambda^n) = {\ell'}^n/\TAU_*(\ell) = 1$,
so $\zeta' \in \mu_n^{\fes'}$.
  Also,
  \[ \TAU'_*(\zeta') = \frac{\TAU'_*(\ell')}{\TAU'_* \TAU_*(\lambda)}
                     = \frac{\TAU'_*(\ell')}{(n \TAU'')_*(\lambda)}
                     = \frac{\TAU'_*(\ell')}{\TAU''_*(\lambda^n)}
                     = \frac{\TAU'_*(\ell')}{\TAU''_*(\ell)}
                     = 1 ,
  \]
  which implies that $\zeta' = \TAU_*(\zeta)$ for some $\zeta \in \mu_n^\fes$.
  Then $(\ell',\ell) = (\TAU_*(\zeta \lambda), (\zeta \lambda)^n) \in U$.
\end{proof}

\begin{corollary} \label{C:ComputeUk}
  Keeping the notation from Lemma~\ref{L:UasKernel}, we have
  \[ U(k) = \bigl\{(\ell',\ell) \in {L'}^\times \times L^\times
                     : \TAU_*(\ell) = {\ell'}^n, \TAU''_*(\ell) = \TAU'_*(\ell')\bigr\} .
  \]
\end{corollary}

In this paragraph, suppose that we are in the fake case.
The map $\TAU_* \colon \mu_n^\fes \to \mu_n^{\fes'}$ 
factors through $E^\vee = \mu_n^\fes/\mu_n$,
so it kills the diagonal image of~$\mu_n$.
Let $\iota \colon \Z \to \Z^\fes$ be the diagonal embedding;
the previous sentence shows that 
$(\tau \circ \iota)_* \colon \mu_n \to \mu_n^{\fes'}$
is trivial.
By Lemma~\ref{L:correspondence divisible by n}, 
$\tau \circ \iota = n\theta$ for some correspondence~$\theta$.
The image of $a \in k^\times$ under $\theta_* \colon \G_m \to \G_m^{\fes'}$
in each component of $(L')^\times$ is a fixed power of~$a$.

\begin{lemma} \label{L:Images} \strut
\begin{enumerate}[\upshape (a)]
\item \label{I:true images}
In the true case, $\widetilde{\G_m^\fes}(k)$
is mapped by $\widetilde{\G_m^\fes} \stackrel{n}\to \G_m^\fes$
to $L^{\times n} \subseteq L^\times$ 
and by $(\TAU_*,n)$ to
\begin{equation}
\label{E:true V}
	\bigl\{(\TAU_*(\ell), \ell^n) : \ell \in L^\times\bigr\} \subseteq U(k). 
\end{equation}
\item \label{I:fake images}
  In the fake case, $\widetilde{\G_m^\fes}(k)$
is mapped by $\widetilde{\G_m^\fes} \stackrel{n}\to \G_m^\fes$
to $L^{\times n} k^\times \subseteq L^\times$ 
and by $(\TAU_*,n)$ to
\begin{equation}
\label{E:fake V}
	\bigl\{(\TAU_*(\ell) \theta_*(a), \ell^n a) : 
	\ell \in L^\times, a \in k^\times \bigr\} \subseteq U(k). 
\end{equation}
\end{enumerate}
\end{lemma}

\begin{proof}
The true case is immediate from the definitions,
so assume that we are in the fake case.
We have an exact sequence
\[
	0 \to \G_m \to \G_m^\fes \times \G_m 
		\stackrel{j}\to \widetilde{\G_m^\fes} \to 0
\]
with maps induced by $\ell \mapsto (\ell,\ell^{-n})$ 
and $(\ell,a) \mapsto \ell a^{1/n}$
(the $n^{\tH}$~root is well-defined modulo~$\mu_n$).
Taking cohomology shows that $j$ induces a surjection 
$L^\times \times k^\times \to \widetilde{\G_m^\fes}(k)$.
Following $j$ by $\widetilde{\G_m^\fes} \stackrel{n}\to \G_m^\fes$
yields $(\ell,a) \mapsto \ell^n a$, 
whose image on $k$-points is $L^{\times n} k^\times$.
Following $j$ by $\widetilde{\G_m^\fes} \stackrel{(\TAU_*,n)}\To U \subseteq \G_m^{\fes'} \times \G_m^\fes$
yields $(\ell,a) \mapsto (\TAU_*(\ell) \theta_*(a), \ell^n a)$,
whose image on $k$-points is~$V$.
\end{proof}

\begin{remark}
\label{R:representatives mod nth powers}
The definition of~$j$ in the proof of Lemma~\ref{L:Images}\eqref{I:fake images}
shows that in~\eqref{E:fake V}, it suffices to let $a$ 
run over a set of representatives of~$k^\times/k^{\times n}$.
\end{remark}

Define 
$V\colonequals \im\left(\widetilde{\G_m^\fes}(k) \stackrel{(\tau_*,n)}\To U(k)\right)$,
so $V$ is given by \eqref{E:true V} or~\eqref{E:fake V}.

\begin{lemma}
\label{L:pulling back divisors}
  Suppose that we have a true (respectively, fake) descent setup $(n,\fes,\LL)$.
  Choose $\beta$ as in Definition~\ref{D:true setup} 
  (respectively, $\beta$ and~$D$
  as in Definition~\ref{D:fake setup}).
  Then $\TAU_*(\beta)$ (respectively, $\TAU_*(\beta) - \theta_*(D)$) 
  is a principal divisor on~$X \times \fes'$.
\end{lemma}

\begin{proof}
In the true case, the composition
\[
	(\Z/n\Z)^{\fes'} \stackrel{\TAU^*}\To (\Z/n\Z)^{\fes} \stackrel{\alpha^\vee}\To \widehat{J}[\widehat{\phi}] \subseteq \Pic X_s
\]
is~$0$ by~\eqref{E:tau*},
and it sends a basis element~$P'$ to the class of the divisor
$\TAU_*(\beta)|_{X \times \{P'\}}$.

In the fake case, 
we use the following diagram:
\begin{equation}
  \label{E:annoying}
\begin{split}
\xymatrix{
\Z^{\fes'} \ar[r]^-{(\TAU^*,\theta^*)} \ar@{->>}[d] & \dfrac{\Z^{\fes} \times \Z}{\langle (nP,1) : P \in \fes \rangle} \ar[r]^-{(\beta,-D)}  & \Pic X_s \\
(\Z/n\Z)^{\fes'} \ar[r]^{\TAU^*} & (\Z/n\Z)^{\fes}_{\deg 0} \ar[u] \ar[r]^-{\alpha^\vee} & \widehat{J}[\widehat{\phi}] \ar@{^{(}->}[u] \\
}
\end{split}
\end{equation}
Here the middle vertical map sends $\sum_P \bar{a}_P P$ (where $a_P \in \Z$
reduces to $\bar{a}_P \in \Z/n\Z$)
to $\left( \sum_P a_P P, \frac{\sum_P a_P}{n} \right)$.
The map $(\beta,-D)$ sends 
$(\sum a_P P,b)$ to the class of $\sum a_P \beta_P - b D$,
so $(nP,1)$ goes to the class of  $n\beta_P-D$, 
which is trivial by the definitions of $\beta$ and~$D$.
Both paths from $\Z^{\fes'}$ to 
$\frac{\Z^{\fes} \times \Z}{\langle (nP,1) : P \in \fes \rangle}$
send a basis element~$P'$ to
$\left( \sum_P \tau_{P',P} P, \frac{\sum_P \tau_{P',P}}{n} \right)$,
where $\tau_{P',P}$ is the coefficient of~$P'$ in $\tau(P) \in \Z^{\fes'}$.
Both paths from $(\Z/n\Z)^{\fes}_{\deg 0}$ to~$\Pic X_s$
send $P_1-P_2$ to the class of $\beta_{P_1} - \beta_{P_2}$.
Thus \eqref{E:annoying} commutes.
The composition along the bottom row is zero, so the
composition along the top row is zero, which is the desired result.
\end{proof}

Lemma~\ref{L:pulling back divisors} shows that 
there is a function 
$r = (r_{P'})_{P' \in \fes'} \in k(X \times \fes')^\times$ such that
\[ 
	\divisor(r) = 
        \begin{cases}
          \TAU_*(\beta), & \textup{ in the true case,} \\
	  \TAU_*(\beta) - \theta_*(D) & \textup{ in the fake case.}
        \end{cases}
\]
Recall from Definitions \ref{D:true setup} and~\ref{D:fake setup}
the function $f \in k(X\times \fes)^\times$ satisfying
\[ 
	\divisor(f) = 
        \begin{cases}
          n \beta, & \textup{ in the true case,} \\
	  n \beta - D & \textup{ in the fake case.}
        \end{cases}
\]

\begin{lemma}
\label{L:TAU_* of functions}
  We have $\TAU_*(f) \equiv r^n \pmod{(L')^\times}$. 
If $y \in \calY^0(X^\good)$, then
  $\TAU_*([y, \beta]) = r(y)$ (resp., $\TAU_*([y, \beta]_D) = r(y)$).
\end{lemma}

\begin{proof}
Since $\divisor(\TAU_*(f)) = n \TAU_*(\beta) = \divisor(r^n)$,
the congruence holds.
In the true case,
\[
	\TAU_*([y, \beta]) = [y, \TAU_*(\beta)] = [y, \divisor(r)] = r(y) 
	\quad\in \G_m^{\fes'}(k_s).
\]
In the fake case, 
define $[y,\beta]_D$ using $H$ and $h$ as in~\eqref{E:fake bracket};
then
\[
	\TAU_*([y, \beta]_D) 
	= \frac{[y, \TAU_*(\beta- \iota_* H)]}{\TAU_* (\iota_*(h(y)^{1/n}))}
        = \frac{[y, \TAU_*(\beta) - \theta_*(nH)]}{\theta_*(h)(y)}
        = \frac{[y, \divisor(r \cdot \theta_*(h))]}{\theta_*(h)(y)}
        = r(y).\qedhere
\]
\end{proof}

In the following diagram with exact rows,
the first, second, and fourth rows are as 
in \eqref{E:true nz+y diagram} or~\eqref{E:fake nz+y diagram},
except that the first is a pushout by $J[n] \surjects A[\phi]$.
The third and fourth rows are the same as in~\eqref{E:two rows}.
The two maps from $\calZ^0 \times \calY^0$ to~$U$ coincide by
Lemma~\ref{L:TAU_* of functions}.
Thus the diagram below commutes:
\begin{equation} \label{E:diagramU}
\begin{split}
   \xymatrix{ 0 \ar[r]
                & A[\phi] \ar[r] \ar@{=}@/_12ex/[dd]
                & A \ar[r]^{\phi}
                & J \ar[r]
                & 0 \\
             0 \ar[r]
                & \widetilde{J[n]} \ar[r] \ar[u] \ar[d]
                & \calZ^0 \times \calY^0 \ar[r]^-{nz+y} \ar[u] \ar[d]
                & \calZ^0 \ar[r] \ar[u] \ar[d]^{(r,f)} \ar@/^12ex/[dd]^{f}
                & 0 \\
             0 \ar[r]
                & A[\phi] \ar[r] \ar[d]_{\alpha}
                & \widetilde{\G_m^\fes} \ar[r]^{(\TAU_*,n)} \ar@{=}[d]
                & U \ar[r] \ar[d]^{\pr_2}
                & 0 \\
             0 \ar[r]
                & E^\vee \ar[r]
                & \widetilde{\G_m^\fes} \ar[r]^n
                & \G_m^\fes \ar[r]
                & 0
           }
\end{split}
\end{equation}
Applying cohomology to the first, third and fourth rows
and using Lemma~\ref{L:Images}\eqref{I:fake images},
we obtain the following diagram with exact rows:
\begin{equation}
\label{E:big cohomology diagram}
\begin{split}
 \xymatrix{ 0 \ar[r]
                & \dfrac{J(k)}{\phi A(k)} \ar[r]^-{\gamma}
                & H^1(k, A[\phi]) \ar@{=}[dd] \\
                & \calZ^0(X^\good) \ar[u] \ar[d]^-{(r,f)} \\
              0 \ar[r]
                & \dfrac{U(k)}{V} \ar[r] \ar[d]^{\pr'_2}
                & H^1(k, A[\phi]) \ar[d]^{\alpha} \ar[r]
                & H^1(k, \widetilde{\G_m^\fes}) \ar@{=}[d] \\
              0 \ar[r]
                & \widetilde{\dfrac{L^\times}{L^{\times n} k^\times}} \ar[r]
                & H^1(k, E^\vee) \ar[r]
                & H^1(k, \widetilde{\G_m^\fes})
            }
\end{split}
\end{equation}
We write $H \colonequals \dfrac{U(k)}{V}$
and $\rho_v \colon H \to H_v$, where $H_v$ is the local analogue of~$H$.

\begin{lemma}
\label{L:Sha^1 of G_m}
$\Sha^1(k,\widetilde{\G_m^\fes})=0$.
\end{lemma}

\begin{proof}
In the true case, $H^1(k,\G_m^{\fes})=0$ 
(see the end of Section~\ref{S:twisted powers}).
In the fake case, taking cohomology of
\[
	0 \to \mu_n \to \G_m^{\fes} \to \widetilde{\G_m^{\fes}} \to 0
\]
yields an injection $H^1(k,\widetilde{\G_m^{\fes}}) \injects H^2(k,\mu_n)$,
which restricts to an injection 
$\Sha^1(k,\widetilde{\G_m^{\fes}}) \injects \Sha^2(k,\mu_n)$.
But $\Sha^2(k,\mu_n)=0$ by the local-global property for the Brauer group.
\end{proof}

\begin{proposition}
\label{P:Sel in terms of H}
Assume Hypothesis~\ref{H:circ}.
Then the injection $H \injects H^1(k,A[\phi])$ 
in the third row of~\eqref{E:big cohomology diagram}
identifies 
  \[ 
	\bigl\{h \in H : \rho_v(h) \in (r,f)\bigl(\calZ^0(X_{k_v}^\good)\bigr)
                                      \textup{\ for all places~$v$}\bigr\}.
  \]
with $\Sel{\phi}(J)$.
\end{proposition}

\begin{proof}
Hypothesis~\ref{H:circ} shows that the image of $J(k_v)/\phi A(k_v)$
in~$H^1(k_v,A[\phi])$ equals the image of $\calZ^0(X_{k_v}^{\good})$,
which in~\eqref{E:big cohomology diagram} for~$k_v$
maps to~$0$ in~$H^1(k_v,\widetilde{\G_m^\fes})$.
Thus the image of $\Sel{\phi}(J) \subseteq H^1(k,A[\phi])$ 
in~$H^1(k,\widetilde{\G_m^\fes})$
lies in~$\Sha^1(k_v,\widetilde{\G_m^\fes})$,
which is~$0$ by Lemma~\ref{L:Sha^1 of G_m},
so $\Sel{\phi}(J) \subseteq H$ by~\eqref{E:big cohomology diagram}.
For $h \in H$, the element $\rho_v(h)$ is in the image
of $J(k_v)/\phi A(k_v)$ (or $\calZ^0(X_{k_v}^{\good})$ as above) 
if and only if it is in the image of $(r,f)\bigl(\calZ^0(X_{k_v}^\good)\bigr)$,
by~\eqref{E:big cohomology diagram}.
\end{proof}

We now describe the fibers of the map~$\pr'_2$ in~\eqref{E:big cohomology diagram}.
This will help us obtain a computable description of~$\Sel{\phi}(J)$:
see Proposition~\ref{P:HSreal} below.

Recall the exact sequence (cf.~Remark~\ref{R:critSha1})
\begin{equation} \label{E:AEQseq}
  0 \To A[\phi] \To E^\vee \stackrel{\TAU_*}\To \mu_n^{\fes'}
      \stackrel{q'}{\To} Q \To 0 .
\end{equation}

\begin{lemma} \label{L:lift-to-H1Aphi}
  Let $\ell \in L^\times$ and let $\xi$ be its image in~$H^1(k, E^\vee)$.
  \begin{enumerate}[\upshape (a)]
    \item \label{L:lift-toH1Aphi-a}
          If $\xi$ lifts
          to~$H^1(k, A[\phi])$, then $\TAU_*(\ell)$ is an $n^{\tH}$ power in~$L'$.
    \item \label{L:lift-toH1Aphi-b}
          Conversely, suppose that $\TAU_*(\ell)$ is an $n^{\tH}$ power in~$L'$,
          say $\TAU_*(\ell) = u^n$.
          Let $\lambda \in L_s^\times$ be such that $\lambda^n = \ell$.
          Then:
          \begin{enumerate}[\upshape 1.]
          \item 
		The element $w \colonequals q'(\TAU_*(\lambda)/u) \in Q$ 
		is $\calG$-invariant,
          \item 
		Let $Z \colonequals \{\zeta \in \mu_n(L') : q'(\zeta) = w\}$.
		Then $(\zeta u,\ell) \in U(k)$ if and only if $\zeta \in Z$.
	\item The set of elements of~$H^1(k, A[\phi])$ mapping to~$\xi$
          is the image of 
          $\{(\zeta u, \ell) : \zeta \in Z\}$
          under $U(k) \to H^1(k, A[\phi])$.
	\item 
          Two such pairs $(\zeta u, \ell)$ and $(\zeta' u, \ell)$ have the
          same image in~$H^1(k, A[\phi])$ if and only if
          $\zeta'/\zeta \in \TAU_*(E^\vee(k))$.
          \end{enumerate}
    \item \label{L:lift-toH1Aphi-c}
          In particular, if $\mu_n(L')$ surjects onto $Q(k)$, then $\xi$
          lifts to~$H^1(k, A[\phi])$ if and only if $\TAU_*(\ell)$ is an $n^{\tH}$ power
          in~$L'$.
  \end{enumerate}
\end{lemma}

\begin{proof}
  In~\eqref{E:UH diagram}, the image of $\pr_2^{-1}(\ell)$
  under $U(k) \to H^1(k, A[\phi])$ equals the set $\alpha^{-1}(\xi)$ of lifts of~$\xi$.
  \begin{enumerate}[\upshape (a)]
    \item If $\xi$ lifts, then there is a pair $(u, \ell) \in U(k)$.
          By Lemma~\ref{L:UasKernel}, $\TAU_*(\ell) = u^n$.
    \item 
      \begin{enumerate}[\upshape 1.]
      \item 
	Note that $\TAU_*(\lambda)/u \in \mu_n^{\fes'}$, since
          \[ \left(\frac{\TAU_*(\lambda)}{u}\right)^n
              = \frac{\TAU_*(\lambda^n)}{u^n}
              = \frac{\TAU_*(\ell)}{u^n}
              = 1 .
          \]
          For $\sigma \in \calG$,
          \[ \frac{{}^\sigma w}{w}
              = \frac{{}^\sigma q'\bigl(\TAU_*(\lambda)/u\bigr)}%
                    {q'\bigl(\TAU_*(\lambda)/u\bigr)}
              = q'\Bigl(\frac{{}^\sigma \TAU_*(\lambda)}{\TAU_*(\lambda)}\Bigr)
              = q'\Bigl(\TAU_*\Bigl(\frac{{}^\sigma \lambda}{\lambda}\Bigr)\Bigr)
              = 1,
          \]
          since $q' \circ \TAU_*$ is trivial and $u$ is $\calG$-invariant.
        \item 
          By part~\eqref{L:lift-toH1Aphi-a}, the elements of 
	  $\pr_2^{-1}(\ell) \subset U(k)$ have
          the form $(\zeta u, \ell)$ with suitable $\zeta \in \mu_n(L')$.
          By definition of~$U$,
          the condition on~$\zeta$ is that there exists $\widetilde{\lambda} \in L_s^\times$
          such that $\TAU_*(\widetilde{\lambda}) = \zeta u$ and $\widetilde{\lambda}^n = \ell$.
          Equivalently, writing $\widetilde{\lambda} = \widetilde{\zeta} \lambda$,
          there exists $\widetilde{\zeta} \in \mu_n(L_s)$ such
          that $\TAU_*(\widetilde{\zeta}) = \zeta u/\TAU_*(\lambda)$.
          By~\eqref{E:AEQseq}, $\TAU_*(\mu_n(L_s)) = \ker(q')$,
          so this in turn is equivalent
          to $q'(\zeta) = q'(\TAU_*(\lambda)/u) = w$.
        \item This follows from 2.\ and the first sentence of this proof.
        \item 
          In~\eqref{E:UH diagram}, the intersection of $\ker \pr_2$ with
          the kernel of $U(k) \to H^1(k, A[\phi])$ is 
	  $(\TAU_*,n)(E^\vee(k))=\TAU_*(E^\vee(k)) \times \{1\}$.
      \end{enumerate}
    \item This follows from part~\eqref{L:lift-toH1Aphi-b}. \qedhere
  \end{enumerate}
\end{proof}

\begin{remark} \label{R:compute w}
  We can compute $w \in Q(k)$ by working over a finite field.
  Let $v$ be a finite place of~$k$ such that the characteristic 
  of its residue field $\F_v$ does not
  divide~$n$ and such that $\ell$ is a unit at~$v$. 
  The formula in Lemma~\ref{L:lift-to-H1Aphi}\eqref{L:lift-toH1Aphi-b}1.,\ 
  applied over~$\F_v$ to the mod $v$ reductions of $\ell$ and~$u$, 
  computes an element of $Q(\F_v)$ that is the image in $Q(\F_v)$ of the
  desired $w$.
  Since the reduction map $Q(k) \to Q(\F_v)$ is injective, 
  we can recover $w$ in $Q(k)$.
\end{remark}

Proposition~\ref{P:Sel in terms of H} involves an infinite group~$H$.
Imposing the condition that elements are unramified 
at all places outside a finite set~$\calS$
lets us replace~$H$ by a finite subgroup~$H_\calS$.
This will reduce the determination of~$\Sel{\phi}(J)$ to a finite computation.
Let $H^1(k, A[\phi])_\calS$ be the group of classes unramified 
outside~$\calS$, as in Section~\ref{S:local and global unramified classes}.

\begin{proposition} \label{P:HSreal}
Assume Hypothesis~\ref{H:circ}.
  Let $\calS$ be a finite set of places of~$k$ containing the set of places
  in Proposition~\ref{P:Sel is unramified} 
  and the places at which $E^\vee$ is ramified.
  Let $H_\calS$ be the preimage in~$H = U(k)/V$ of
  $\widetilde{L(n, \calS)} \subset \widetilde{L^\times/L^{\times n} k^\times}$
  under the map~$\pr'_2$ in~\eqref{E:big cohomology diagram}.
Then the injection $H \injects H^1(k,A[\phi])$ 
in the third row of~\eqref{E:big cohomology diagram}
identifies the subgroup~$H_\calS$ with
$H^1(k, A[\phi])_\calS \intersect \im(H \to H^1(k,A[\phi]))$,
and identifies
\begin{equation}
  \label{E:H_S and Selmer}
 \bigl\{h \in H_\calS :
        \rho_v(h) \in (r,f)\bigl(\calZ^0(X_{k_v}^\good)\bigr)
                \textup{\ for all $v \in \calS$}\bigr\}
\end{equation}
with~$\Sel{\phi}(J)$.
\end{proposition}

\begin{proof}
  Let $v \notin \calS$ be a place of~$k$. 
  Since $E^\vee$ is unramified at~$v$, we have $E^\vee(k_{v,u}) = E^\vee(k_s)$,
  so the first map in 
  \[ E^\vee(k_{v,u}) \To R^\vee(k_{v,u}) \To H^1(k_{v,u}, A[\phi])
      \stackrel{\alpha}{\To} H^1(k_{v,u}, E^\vee)
  \]
  is surjective.  This shows that for $\xi \in H^1(k, A[\phi])$,
  $\xi$ is unramified outside~$\calS$ if and only if $\alpha(\xi)$ is unramified
  outside~$\calS$.
  On the other hand, Proposition~\ref{P:unramified classes} 
  shows that $H_\calS$ 
  equals the set of $h \in H$ whose image $\pr_2'(h)$ 
  in $\widetilde{L^\times/L^{\times n} k^\times}$
  is unramified outside $\calS$.
  By~\eqref{E:big cohomology diagram}, the previous two sentences
  yield the first identification.
  Combining this with
  Proposition~\ref{P:Sel is unramified}\eqref{I:description of Selmer}
  and the last sentence of the proof of Proposition~\ref{P:Sel in terms of H}
  yields the second identification.
\end{proof}

We now sketch an algorithm for computing $\Sel{\phi}(J)$,
using the explicit description given in Proposition~\ref{P:HSreal}.
First compute~$\widetilde{L(n, \calS)}$.
Lemma~\ref{L:lift-to-H1Aphi} lets us compute its inverse image~$H_\calS$
under the map~$\pr_2'$ in~\eqref{E:big cohomology diagram}.
To perform the computations required by Lemma~\ref{L:lift-to-H1Aphi},
we need to be able to evaluate $\TAU_* : L^\times \to {L'}^\times$ and
extract $n^\tH$~roots in~$L'$;
the remaining computations use only finite Galois modules like $E^\vee$,
$\mu_n^{\fes'}$, or~$Q$, so they are not difficult.

Similarly, for each $v \in \calS$, 
first compute $\widetilde{L_v^\times/L_v^{\times n} k_v^\times}$,
and use Lemma~\ref{L:lift-to-H1Aphi} (over~$k_v$) to obtain a description of~$H_v$.
The map $H_\calS \to H_v$ is induced by the inclusion $U(k) \injects U(k_v)$,
so it too is easily described.
Next, assuming that we can evaluate $r$ and~$f$
on~$\calZ^0(X_{k_v}^\good)$, 
we can determine the image of~$J(k_v)$ in~$H_v$.
Using all this, the second identification in Proposition~\ref{P:HSreal}
lets us compute~$\Sel{\phi}(J)$.

We summarize this discussion as follows.

\begin{theorem} \label{T:realSel}
  Given a true or fake descent setup with associated isogeny 
  $\phi \colon A \to J$, 
  we can compute the Selmer group $\Sel{\phi}(J)$ if we can
  do the following:
  \begin{itemize}
    \item Compute in the algebras $L$ and~$L'$ and their completions 
		(this includes the ability to take $n^{\tH}$ roots).
    \item Determine a set $\calS$ of places of~$k$ as in Proposition~\ref{P:HSreal}.
    \item Compute $k(n, \calS)$ and~$L(n, \calS)$.
    \item Evaluate $\TAU_*$ on $L^\times$ and $L_v^\times$ for $v \in \calS$,
          and on finite residue fields.
    \item Evaluate $f$ and~$r$ on~$\calZ^0(X_{k_v}^\good)$ for $v \in \calS$.
  \end{itemize}
\end{theorem}

\begin{example}
Consider the case of $2$-descent on Jacobians 
of non-hyperelliptic genus-$3$ curves~$X$, as in Section~\ref{S:genus 3}.
For generic~$X$, the smallest usable set~$\fes'$ is the set of syzygetic
quadruples, which has size~$315$.
We have not had the need to implement the approach of this appendix
on such examples, but it is likely that this could be done if required.
\end{example}


\subsection{Determining the Selmer group in special situations}

In some cases, the computation of~$\Sel{\phi}(J)$ as described in
Section~\ref{S:Det Selmer} can be simplified.
In particular, in the following proposition, the set~$\calS$
is potentially smaller than that required in Proposition~\ref{P:HSreal}.

\begin{proposition} \label{P:specialSel}
  In the situation of Section~\ref{S:Det Selmer}, let $\calS$ be a set of
  places of~$k$ containing the set of places 
  in Proposition~\ref{P:Sel is unramified}.
  Assume that the map $E^\vee(k') \to R^\vee(k')$ is surjective for all field
  extensions $k'$ of $k$ and that the map $q' \colon \mu_n(L') \to Q(k)$ 
  is surjective. 
  Then
  \[ \Sel{\phi}(J) \isom \{\delta = [\ell] \in \widetilde{L(n, \calS)} :
                      \TAU_*(\ell) \in {L'}^{\times n} \textup{\ and\ }
                      \delta_v \in \im C_v \textup{\ for all $v \in \calS$}\} .
  \]
\end{proposition}

\begin{proof}
  Since $E^\vee(k) \to R^\vee(k)$ is surjective,
  the map $H^1(k, A[\phi]) \stackrel{\alpha}{\to} H^1(k, E^\vee)$ is injective.
  So the map $\pr_2'$ in~\eqref{E:big cohomology diagram} 
  identifies $H$ with a subgroup of $\widetilde{L^\times/L^{\times n} k^\times}$,
  and similarly identifies $H_v$ with a subgroup of 
  $\widetilde{L_v^\times/L_v^{\times n} k_v^\times}$, for each~$v$.
  Also $E^\vee(k_{v,u}) \to R^\vee(k_{v,u})$ is surjective for each~$v$,
  so the proof of Proposition~\ref{P:HSreal} shows that 
  $H_\calS$ is the subgroup of elements of~$H$ unramified outside~$\calS$.
  Since $\mu_n(L') \to Q(k)$ is surjective, 
  Lemma~\ref{L:lift-to-H1Aphi}\eqref{L:lift-toH1Aphi-c} yields
 \begin{equation}
 \label{E:H_S formula}
	 H_\calS \isom \{[\ell] \in \widetilde{L(n, \calS)} : 
			\TAU_*(\ell) \in {L'}^{\times n}\}.
  \end{equation}
  In~\eqref{E:big cohomology diagram}, $C_v = \pr_2' \circ (r,f)$, 
  so $(r,f)\bigl(\calZ^0(X_{k_v}^\good)\bigr) = \im C_v$.
  Substituting this and \eqref{E:H_S formula}
  into~\eqref{E:H_S and Selmer} completes the proof.
\end{proof}

\begin{remark}
  In contrast with Theorem~\ref{T:realSel}, 
  Proposition~\ref{P:specialSel} lets us determine $\Sel{\phi}(J)$
  without evaluating $\TAU_*$ on~$L_v$
  and without evaluating (or even constructing)~$r$.
\end{remark}

\begin{remark} \label{R:surj finite comp}
  The surjectivity assumptions in Proposition~\ref{P:specialSel} 
  can be checked by a finite computation.
  This is clear for the statements over~$k$, 
  since they involve only the $k$-points of some finite Galois modules. 
  To check surjectivity of $E^\vee(k') \to R^\vee(k')$
  for all extension fields~$k'$ of~$k$, 
  let $\Gamma$ be the (finite) Galois group over~$k$
  of the splitting field of~$E^\vee$; 
  then the action of $\calG_{k'}$
  on~$E^\vee$ (and therefore also~$R^\vee$) factors through a subgroup
  $\Gamma' \le \Gamma$.
  The surjectivity of $E^\vee(k') \to R^\vee(k')$ is determined by~$\Gamma'$,
  so it suffices to check it for each~$\Gamma'$;
  in fact, we need only consider one~$\Gamma'$ in each conjugacy class
  of subgroups.
\end{remark}

\begin{remark}
\label{R:Sel=Sel}
Suppose that $n$ is prime, and that the surjectivity assumptions 
in Proposition~\ref{P:specialSel} hold.
Then $\Sha^1(k, R^\vee) = 0$ by Remark~\ref{R:critSha1}, 
and $W_v = 0$ for all~$v$ by Lemma~\ref{L:Kvtrivial}\eqref{I:W_v formula},
and $\calK = 0$ by its second definition (preceding~\eqref{E:kappa}).
So by Lemma~\ref{L:Sha1R}\eqref{I:if Sha=0} 
and Corollary~\ref{C:Nicecase}, we actually have
  \[ \Sel{\phi}(J) \isom \Seltf{\alpha}(J) . \]
The advantage of 
computing~$\Sel{\phi}(J)$ using Proposition~\ref{P:specialSel}
instead of computing~$\Seltf{\alpha}(J)$
using Theorem~\ref{T:finite description of Selmer}\eqref{I:good T} 
is that the former requires local computations
at only the places in~$\calS$ instead of the places
in the potentially much larger set~$\calT$
of Theorem~\ref{T:finite description of Selmer}\eqref{I:good T}.
This improvement is possible because of
the additional condition $\TAU_*(\ell) \in {L'}^{\times n}$:
Proposition~\ref{P:specialSel} says that
$\Sel{\phi}(J) \isom S_{\calS} \cap \{[\ell] \in \widetilde{L(n, \calS)} : \TAU_*(\ell) \in {L'}^{\times n}\}$,
which by Lemma~\ref{L:lift-to-H1Aphi}\eqref{L:lift-toH1Aphi-c}
equals $S_\calS \cap \alpha(H^1(k,A[\phi]))$ under the assumptions made.
In other words, enlarging $\calS$ to a~$\calT$ large enough that
$S_\calT = \Seltf{\alpha}(J)$ has the effect of intersecting
$S_\calS$ with $\alpha(H^1(k,A[\phi]))$.
\end{remark}

\begin{example}
  We consider 3-descent on an elliptic curve~$J$,
  as in \cite{Schaefer-Stoll2004}.
  Let $\fes = J[3] - \{0\}$, and let $\beta$ be the graph of the
  map $\fes \to \Div^0 J_s$ sending $P$ to $(P)-(O)$.
  This defines a true descent setup, for which $A[\phi]=J[3]$
  and $L$ is an \'etale algebra of degree~$8$.
  Fix a Weierstrass equation for~$J$; then
  let $\fes'$ be the set of eight lines in~$\PP^2$ passing through
  three of the points in~$\fes$, together with the four vertical lines
  passing through two of them and the origin of~$J$.
  Let the correspondence~$\TAU$ be given by incidence.
  A finite computation as in Remark~\ref{R:surj finite comp}
  shows that the maps $\mu_3(L') \to Q(k)$
  and $\mu_3(L) \to R^\vee(k)$ are surjective over any field.
  So Proposition~\ref{P:specialSel} gives us a way to determine
  $\Sel{3}(J)$, if we can compute $L(3, \calS)$.
\end{example}

The approach described here and in Section~\ref{S:Det Selmer}
has advantages and disadvantages compared
to computing~$\Seltf{\alpha}(J)$ as in Section~\ref{S:compute fake}.
One obvious advantage is that it computes the Selmer group directly.
Also, as in Remark~\ref{R:Sel=Sel},
it requires local computations at only the places in~$\calS$
instead of the places in a potentially much larger set~$\calT$.
So whenever this direct approach is feasible (as in the example above), 
one should use it.

On the other hand, the direct approach
requires not only information on ($\calS$-)class
and unit groups of~$L$,
but also a presentation of~$L'$, the map
$\TAU_* : L^\times \to {L'}^\times$, and possibly the function~$r$, 
which is defined over~$L'$.


\subsection{Passing from a true or fake Selmer group to the actual Selmer group}

Recall the exact sequence
\[ 0 \To \ker \kappa \To \Sel{\phi}(J)
    \stackrel{\alpha}{\To}
      \Seltf{\alpha}(J) \intersect \alpha\bigl(H^1(k, A[\phi])\bigr)
    \To \coker \kappa.
\]
from Theorem~\ref{T:Selseq}. Assuming that we have computed~$\Seltf{\alpha}(J)$,
we can determine the order of~$\Sel{\phi}(J)$ if we can
\begin{itemize}
  \item find the order of $\ker \kappa$;
  \item for any given $\xi \in \Seltf{\alpha}(J)$ check if
        $\xi \in \alpha\bigl(H^1(k, A[\phi])\bigr)$;
  \item and if so, find its image in $\coker \kappa$.
\end{itemize}

In this subsection, we will explain why this seems no easier than computing~$\Sel{\phi}(J)$
directly using the approach of Section~\ref{S:Det Selmer}.

Lemma~\ref{L:lift-to-H1Aphi} implies the following.

\begin{corollary} \label{C:lift}
  If we can evaluate the map $\TAU_*$ explicitly on~$L^\times$
  and if we can extract $n^\tH$ roots in~$L'$,
  then we can determine whether any given
  element in the image of~$L^\times$ in~$H^1(k, E^\vee)$ lies
  in~$\alpha\bigl(H^1(k, A[\phi])\bigr)$.
\end{corollary}

This takes care of the second point in the list above.

For the other two points, we need to find the groups~$W_v$ and the map
$\kappa \colon \calK \to \prod_v W_v$. The group $\calK = R^\vee(k)/q E^\vee(k)$
can be found by a finite computation,
using the inclusion $R^\vee(k) \injects \mu_n(L')$ 
arising from~\eqref{E:Q and q'} to represent elements of~$R^\vee(k)$.
For the finitely many~$v$ for which 
the proof of Corollary~\ref{C:W_v=0 usually}
does not guarantee $W_v=0$, we use the following to compute~$W_v$:

\begin{lemma}\label{L:new computeWv}
Fix $v$.
\begin{enumerate}[\upshape (a)]
\item The map 
\[
	\ker C_v \stackrel{\gamma_v}\injects \ker \alpha_v 
		\isom R^\vee(k_v)/q E^\vee(k_v)
\]
can be described explicitly as follows.
Given $[z] \in \ker C_v \subseteq J(k_v)/\phi A(k_v)$,
represented by some $z \in \calZ^0(X_{k_v}^{\good})$,
the image~$\gamma_v([z])$
is represented by $\zeta \in R^\vee(k_v) \subset \mu_n(L_v')$
defined as follows:\\
      In the true case, $f(z_0) = \ell^n$ for some $\ell \in L_v^\times$; 
         set $\zeta := r(z)/\TAU_*(\ell)$. \\
      In the fake case, $f(z_0) = \ell^n a$ for some $\ell \in L_v^\times$
         and $a \in k_v^\times$; set $\zeta := r(z)/(\TAU_*(\ell) \theta_*(a))$.
\item \label{I:W_v algorithm}
Assume Hypothesis~\ref{H:curve with a k_v-point}.
If we know $\#\bigl(J(k_v)/\phi A(k_v)\bigr)$ 
and we can compute the functions $f$ and~$r$ on~$\calZ^0(X_{k_v}^{\good})$
and $\TAU_*$ on~$L_v^\times$,
then we can compute~$W_v$ in the course of the computation of~$\im C_v$
as in Remark~\ref{R:image of random points} as follows:
  \begin{enumerate}[\upshape 1.]
  \item Search for points on~$X$ over finite extensions of~$k_v$
        until a $0$-cycle $x_0$ of degree~$1$ on~$X_{k_v}$ is found.
  \item Randomly generate elements 
          $z \colonequals \tr_{K/k_v}(x) - dx_0 \in \calZ^0(X_{k_v}^{\good})$
        as in Remark~\ref{R:image of random points},
        let $\calJ$ be the subgroup of~$J(k_v)/\phi A(k_v)$ they generate so far,
        compute $I \colonequals C_v(\calJ) \subseteq 
		\widetilde{L_v^\times/L_v^{\times n} k_v^\times}$,
        and compute $G\colonequals \gamma_v(\ker(C_v|_\calJ)) 
		\subseteq R^\vee(k_v)/q E^\vee(k_v)$,
	until $\#I \cdot \#G = \#\bigl(J(k_v)/\phi A(k_v)\bigr)$.
  \item When equality occurs, $I = \im C_v$ 
	and $W_v \isom \bigl(R^\vee(k_v)/q E^\vee(k_v)\bigr)/G$.
  \end{enumerate}
\end{enumerate}
\end{lemma}

\begin{proof}
\hfill
\begin{enumerate}[\upshape (a)]
\item 
In the true case, 
$C_v([z])=0$ means that $f(z)=\ell^n$ for some $\ell \in L_v^\times$.
Let $V_v$ be the local analogue of~$V$.
Dividing $(r(z),f(z)) \in U(k_v)$
by $(\TAU_*(\ell),\ell^n) \in V_v$ (cf.~Lemma~\ref{L:Images})
yields $(\zeta,1) \in U(k_v)$,
so Corollary~\ref{C:ComputeUk} implies $\zeta \in \mu_n(L')$.
In the fake case, divide instead by $(\TAU_*(\ell) \theta_*(a), \ell^n a) \in V_v$.
Lemma~\ref{L:lift-to-H1Aphi}\eqref{L:lift-toH1Aphi-b}2.\ 
applied with $\ell=1$, $u=1$, $\lambda=1$, and hence $w=1$, yields
$\zeta \in \ker(q'|_{\mu_n(L')}) = R^\vee(k_v)$ (see~\eqref{E:Q and q'}).
\item 
We have $\#I \cdot \#G = \#\calJ \le \#\bigl(J(k_v)/\phi A(k_v)\bigr)$
with equality if and only if $\calJ = J(k_v)/\phi A(k_v)$.
When equality occurs, $I = \im C_v$ and $G=\gamma_v(\ker C_v)$,
and Definition~\ref{D:W_v}\eqref{I:W_v is cokernel} yields
$W_v \isom \bigl(R^\vee(k_v)/q E^\vee(k_v)\bigr)/G$.\qedhere
\end{enumerate}
\end{proof}

\begin{remark}\label{R:stopping rule}
By computing $\#\calJ = \#I \cdot \#G$ 
as the algorithm in Lemma~\ref{L:new computeWv}\eqref{I:W_v algorithm}
progresses, we can detect when $\calJ = J(k_v)/\phi A(k_v)$.
This stopping rule can help make our computation more efficient.
\end{remark}

Once we know $W_v$ as a quotient of~$R^\vee(k_v)$ 
for all~$v$ for which $W_v$ might be nonzero,
the map
$\kappa \colon \calK \to \prod_v W_v$, with $\calK = R^\vee(k)/q E^\vee(k)$,
is induced by the maps $R^\vee(k) \to R^\vee(k_v)$
and hence is computable.
In particular, we can find the size of~$\ker \kappa$, which takes care of
the first point in our list.

The following lemma deals with the last point.

\begin{lemma}\label{L:image in coker kappa}
  Under the assumptions in 
  Corollary~\ref{C:lift} and Lemma~\ref{L:new computeWv}\eqref{I:W_v algorithm},
  given an element $\xi \in \Seltf{\alpha}(J) \intersect \alpha(H^1(k, A[\phi]))$,
  represented by some explicit $\ell \in L^\times$,
  we can find the image of~$\xi$ in $\coker \kappa$.
\end{lemma}

\begin{proof}
  Use Lemma~\ref{L:new computeWv} (and the sentence preceding it) 
  to compute~$W_v$ for all~$v$.
  Compute~$\kappa$ as in the sentences preceding
  Lemma~\ref{L:image in coker kappa}.
  Our task is to compute the image of~$\xi$ under the snake map
  implied by~\eqref{E:big snake}.
  To do this, we do a parallel diagram chase in the more computation-friendly diagram 
  \begin{equation}
    \begin{split}
    \label{E:not snake}
 \xymatrix{    
                0 \ar[r] & R^\vee(k) \ar[r] \ar[d]
                         & U(k) \ar[r]^-{\pr_2} \ar[d]
                         & L^\times \ar[d] \\
                0 \ar[r] & \rule[-0.7em]{0pt}{2.0em}\displaystyle\prod_v R^\vee(k_v) \ar[r]
                         & \rule[-0.7em]{0pt}{2.0em}\displaystyle\prod_v U(k_v) \ar[r]^-{\pr_2}
                         & \rule[-0.7em]{0pt}{2.0em}\displaystyle\prod_v L_v^\times \\
              }
    \end{split}
  \end{equation}
(with exact rows from~\eqref{E:UH diagram}), which maps to~\eqref{E:big snake}.
  
Use Lemma~\ref{L:lift-to-H1Aphi} to compute~$u$ such that $(u,\ell) \in U(k)$,
i.e., such that $\TAU_*(\ell)=u^n$.
  For the places~$v$ for which $W_v=0$,
  set $\zeta_v \colonequals 1 \in R^\vee(k_v)$.
  For the finitely many remaining~$v$, we have $\xi_v = C_v([z_v])$ for some
  $z_v \in \calZ^0(X_{k_v}^{\good})$ found during the local image computation;
  then $f(z_v) \equiv \ell$ in $\widetilde{L_v^\times/L_v^{\times n} k_v^\times}$.
  Dividing $(u,\ell) \in U(k_v)$ by~$(r(z_v),f(z_v))$
  and then by an element of~$V_v$ as in Lemma~\ref{L:Images}
  yields an element $(\zeta_v,1) \in U(k_v)$ with the same image as~$(u,\ell)$
  in the group $H^1(k_v,A[\phi])/\im \gamma_v$ in the corresponding position
  of~\eqref{E:big snake}.
  By exactness of the second row of~\eqref{E:not snake}, 
  we have $\zeta_v \in R^\vee(k_v)$.
  Then $(\zeta_v) \in \prod R^\vee(k_v)$ 
  maps to an element of $\prod W_v$ in~\eqref{E:big snake},
  which represents the image of~$\xi$ in~$\coker \kappa$.
\end{proof}

We conclude that in order to find 
(the order of) $\Sel{\phi}(J)$ from~$\Seltf{\alpha}(J)$,
we need to be able to lift a given $\ell \in L^\times$ to $(u,\ell) \in U(k)$
(or show that such a lift does not exist),
and to evaluate $r$ and~$\TAU_*$ locally at the places with potentially
nontrivial~$W_v$. 
But if we can do all this, then we can also compute~$\Sel{\phi}(J)$
directly as described in Theorem~\ref{T:realSel}.

\section*{Acknowledgments} 

Our research project was begun
at the Mathematical Sciences Research Institute in 2006, 
and continued over several years at several institutions:
Banff International Research Station, 
Institute for Computational and Experimental Research in Mathematics, 
Jacobs University, 
Massachusetts Institute for Technology, 
Mathematisches Forschungsinstitut Oberwolfach, 
Pacific Institute for the Mathematical Sciences, 
and Universit\"at Bayreuth.
We thank them all for their hospitality. 
We thank also David Kohel and Christophe Ritzenthaler for comments
on discriminants of ternary quartic forms, Benedict Gross and
Jack Thorne for the content of Remark~\ref{R:Thorne},
and Edward Schaefer for suggestions of references.

\end{document}